\documentclass[11pt,a4]{amsart}
\usepackage{amscd,amsmath,amssymb,amsfonts,ifthen,hyperref,tikz-cd}
\usepackage{calligra}
\usepackage{mathtools}
\usepackage{hyperref}
\usepackage{bbold}
\usepackage{color} 
\usepackage{graphicx}
\usepackage[all]{xy} 

\usepackage{fourier} 
\usepackage{srcltx}
 
\newtheorem{theorem}{Theorem}
\newtheorem{proposition}[theorem]{Proposition}

\newtheorem{thm}[subsubsection]{Theorem}
\newtheorem{lem}[subsubsection]{Lemma}
\newtheorem{cor}[subsubsection]{Corollary}
\newtheorem{prop}[subsubsection]{Proposition}
\theoremstyle{definition}
\newtheorem{defn}[subsubsection]{Definition}

\newtheorem{rem}[subsubsection]{Remark}

\DeclareMathAlphabet{\mathcalligra}{T1}{calligra}{m}{n}
\DeclareFontFamily{OT1}{pzc}{}
\DeclareFontShape{OT1}{pzc}{m}{it}{<-> s * [1.10] pzcmi7t}{}
\DeclareMathAlphabet{\mathpzc}{OT1}{pzc}{m}{it}

\newcommand{\sttimes}{\mbox{\hspace{.2ex}${\hspace{.2ex}}_s\times_t$\hspace{.2ex}}}

\newcommand{\stotimes}{\mbox{\hspace{.2ex}${\hspace{.2ex}}_s\otimes_t$\hspace{.2ex}}}

\newcommand{\ga}[2]{\begin{gather}\label{#1}#2 \end{gather}}

\newcommand{\Rep}{\text{\rm Rep\hspace{.1ex}}}

\renewcommand{\inf}{\text{\rm inf\hspace{.1ex}}}
\newcommand{\Spec}{\text{\rm Spec\hspace{.1ex}}}
\newcommand{\Hom}{\text{\rm Hom\hspace{.1ex}}}
\renewcommand{\mod}{\text{\rm mod\hspace{.1ex}}}
\newcommand{\id}{\text{\rm id\hspace{.1ex}}}

\usepackage{geometry}
\date{\today}
\author[PH Hai]{Ph\`ung H\^o Hai}
\email[PH Hai]{phung@math.ac.vn}

\author[VQ Bao]{V\~o Qu\^oc Bao}
\email[VQ Bao]{vqbao@math.ac.vn}

\author[TPQ Bao]{Tr\^an Phan Qu\^oc Bao}
\email[TPQ Bao]{phan.tran@unibe.ch}

\address[TPQ Bao]{University of Bern, Switzerland}

\address[VQ Bao, PH Hai]{Institute of Mathematics, Vietnam Academy of Science and Technology}

\title{Tannakian duality and Gauss-Manin connections for a family of curves}  

\begin{document}
\makeatletter
\@namedef{subjclassname@2020}{\textup{2020} Mathematics Subject Classification}
\makeatother
\maketitle  
 
\begin{abstract} 
Let $X/S$ be a smooth family of smooth projective varieties, where $S$ is a smooth affine curve
over a field $k$ of characteristic $0.$
We relate the differential fundamental groupoid scheme of $X/k$ with
the differential fundamental groupoid scheme of $S/k$ and the 
relative differential fundamental group of $X/S$ in a short exact sequence. This yields natural maps
from the group cohomology of the geometric relative fundamental group
to the Gauss-Manin connections. For families of curves of genus at least $1,$ 
we prove that these maps are isomorphisms. This gives
an interpretation of the Gauss-Manin connection in terms of cohomology of the
differential fundamental group. As a consequence we can shrink $X$  
(as a family on $S$) to obtain a de Rham $K(\pi,1)$ surface.  
\end{abstract} 

\section{Introduction}\label{sect_intro}
For a connected noetherian scheme $X$, the \'etale homotopy group of $X$ is defined in terms
of the \'etale homotopy type $X_\mathrm{et}$. Here $X_\mathrm{et}$ is an object in
the pro-category of simplicial sets. A geometric point $\bar x$ of $X$ determines a point of 
$X_\mathrm{et}$ and the \'etale fundamental group $\pi^\mathrm{et}(X,\bar x)$ is isomorphic
to the fundamental group $\pi_1(X_\mathrm{et},\bar x)$, cf. \cite{AM69}, see also \cite{SS16,Ach17}. 
$X$ is said to be \textit{\'etale} $K(\pi,1)$ if the \'etale homotopy groups vanish in
degree $>1$. This property is equivalent to the requirement that the natural maps
$$\delta^i : \mathrm{H}^i (\pi_1(X_\mathrm{et},\bar x),\mathcal F_{\bar x}) \longrightarrow \mathrm{H}_\mathrm{et}^i (X,\mathcal F )$$
are isomorphisms for every locally constant constructible abelian sheaf $\mathcal F$ and all $i\geq 0$.
There are many examples of \'etale $K(\pi,1)$. For instance, in characteristic $0$, if $X$ is a normal, connected $\mathbb{C}$-scheme of finite type such that the group $\pi_1 := \pi_1(X(\mathbb{C}))$ is good (i.e., $\mathrm{H}^i(\widehat{\pi_1}, M) \cong \mathrm{H}^i(\pi_1, M)$ for any finite continuous representation $M$ of $\widehat{\pi}$) and the universal cover of $X(\mathbb C)$ is weakly contractible, then $X$ is an \'{e}tale $K(\pi, 1)$. In positive characteristic,
Achinger \cite{Ach17} showed that every connected affine scheme 
is \'etale $K(\pi,1)$. 

In \cite{EH08} the authors study a similar map to $\delta^i$ for the differential fundamental group
and the de Rham cohomology and show that they are isomorphisms when $X$ is an affine curve
in characteristic zero. In \cite{BHT25} it is shown that $\delta^i$ are isomorphisms for projective
curves of genus at least 1. A similar statement in positive characteristic for the stratified
fundamental group and infinitesimal cohomology was obtained in \cite{BHT25b}.
See also \cite{BN26} for the similar claim for abelian varieties. 
In this work we consider the problem for the case of one-parametric families of smooth projective 
curves of genus at least 1.

Let $f:X\longrightarrow S$ be a smooth morphism of smooth schemes over a field $k$ of characteristic zero.
Let $(\mathcal V,\nabla)$ be a coherent sheaf on $X$ equipped with a flat connection on $X/k$ 
(see \ref{sect-connection}). 
``Inflate'' it to a relative connection $(\mathcal V,\nabla_{/S})$  on $X/S$. Then 
the de Rham cohomology of $(\mathcal V,\nabla_{/S})$ is a sheaf on $S$ which is equipped with
a connection, called \textit{Gauss-Manin connection}. 

Fix now a $k$-point $x$ in $X$ and let $s:=f(x)$. Tannakian duality yields an affine group scheme
$\pi(X,x)$ and the representation category of which is equivalent to the category of coherent sheaves
with flat connections on $X/k$. $\pi(X,x)$ is called the 
\textit{differential fundamental group scheme} of $X$
with base point at $x$. 
The map $f$ induces a group homomorphism $\pi(X,x)\longrightarrow \pi(S,s)$.
Further, assuming that $f$ is proper, this homomorphism fits in an exact sequence, similar to the
homotopy exact sequence for the topological fundamental groups. 

H\'el\`ene Esnault raised the following question: can one describe the Gauss-Manin connection
on $\mathrm{H}^i_\mathrm{dR}(X/S,(\mathcal V,\nabla_{/S}))$ in terms of the cohomology
of the fundamental group schemes. This question would consist of more precise sub-questions. 
\begin{enumerate}
\item Is the natural map from the cohomology of the differential fundamental group scheme with
coefficients in a representation to the de Rham cohomology of the corresponding connection
an isomorphism?
\item Is there a notion of the relative differential fundamental group (for the relative scheme $X/S$)
and how is it related to the (absolute) differential fundamental groups mentioned above?
\end{enumerate}

It is not expected that these questions will have affirmative answers in the above generality.
Instead, one would ask for settings where the answers are affirmative.
In \cite{EH06}, the following case was 
studied: $S=\mathrm{Spec}(K)$, where $K$ is a function field over $k$ and $X/K$ is a smooth
curve. The answer to question (1) is affirmative for the first and second cohomology groups 
(see also \cite{BHT25}). 
An answer to question (2) is also given. 
Here, $S$ is chosen to be a point in order to apply the Tannakian 
duality. In \cite{BHT25}, question (1) is affirmatively answered for any curve of genus at least 1.
We note that, as the de Rham cohomology on curves vanish in degree larger than 2, this isomorphism
implies the similar vanishing property for the cohomology of the differential fundamental group
scheme. With this property, one can say that smooth projective curves of genus at least 1 are
\textit{de Rham $K(\pi,1)$}.

In this work, we consider the following setting: $S$ is a smooth affine curve over $k$ (thus $S$ is
the spectrum
of a Dedekind domain $R=\mathcal O(S)$ over $k$);
$f:X\longrightarrow S$ is a relative smooth, proper curve of
genus $g\geq 1$. Assuming the existence of a section $\eta:S\longrightarrow X$ of $f$, 
we show that the answers to both questions mentioned above are affirmative.
In particular, we show that $X$, as a surface over $k$, becomes de Rham $K(\pi,1)$
after a suitable shrinking. Below are more detailed
descriptions of the work. 

The initial motivation of our study are the two exact sequences of \'etale fundamental groups due to
A. Grothendieck. 
For algebraic schemes, Grothendieck's \'etale fundamental group serves as a replacement
for the topological fundamental group. 
Let $k$ for the moment be an arbitrary field. 
Let $X/k$ be a separated scheme, and let $x\in X\times_k\bar k$ be a $\bar k$-point of 
$X\times_k\bar k$, where $\bar k$ is an algebraic closure of $k$. Grothendieck defined in
\cite{SGA1} the \'etale fundamental group $\pi^\text{et}(X, x)$.
This group is a pro-finite group which classifies Galois coverings of $X$ 
(with a distinguished point above $ x$). Then Grothendieck deduced the following exact sequences.
\begin{enumerate}\item
The fundamental exact sequence which relates the \'etale fundamental groups of $X$, with
the geometric fundamental group $\pi^\text{et}(X\times_k\bar k , x)$ and
the Galois group $\text{\rm Gal}(\bar k/k)$:
$$1\longrightarrow \pi^\text{et}(X\times_k\bar k, x)\longrightarrow 
\pi^\text{et}(X, x)  \longrightarrow \mathrm{Gal}(\bar k /k)     \longrightarrow  1.$$
\item The homotopy exact sequence that relates the fundamental group of $X$, of $S$ and
of the fiber $X_s:=f^{-1}(s)$:
$$\pi^\text{et}(X_s, x)\longrightarrow 
\pi^\text{et}(X, x)\longrightarrow \pi^\text{et}(S, s)\longrightarrow  1,$$ 
provided that $f$ is proper with connected fibers. 
\end{enumerate} 
 
There is another algebraic replacement of the topological fundamental group
motivated by the Riemann-Hilbert correspondence.
Recall that on a smooth complex algebraic variety, there is a one-to-one 
correspondence between complex local systems
(i.e. complex finite-dimensional representations of the topological fundamental group) and
regular singular flat connections on algebraic vector bundles.
For a smooth scheme $X$ over a field $k$ of 
characteristic $0,$ consider the category 
of flat connections on $X/k$. A $k$-rational point $x$ of $X$ yields a fiber functor, 
i.e., it associates 
to each connection its fiber at $x$. Thus, Tannakian duality applies and yields a pro-algebraic 
affine group scheme. This group scheme is called the {\em differential fundamental group of $X$ at $x$}
and is denoted by $\pi(X,x)$. For simplicity, from now on we shall assume that \textit{$k$
is an algebraically closed field (of characteristic 0). }

Assume that $f:X\longrightarrow S$ is smooth, projective morphism of smooth schemes.
Further, assume that $f$ has connected fibers in the sense that $f_*(\mathcal O_X)=\mathcal O_S$. 
Fix a point $x\in X(k)$ and its image $s=f(x)\in S(k)$. There is a sequence
similar to the second one mentioned above: 
\begin{equation}\label{IntHES}
\pi(X_s/k,x)\longrightarrow \pi(X/k, x)\longrightarrow \pi(S/k,s)\longrightarrow 1,
\end{equation}
which was shown to be exact in characteristic zero case by L. Zhang \cite{Zha14} and  in any
characteristic by J. P. dos Santos \cite{dS15} (see Theorem \ref{HES}). It is called
the homotopy exact sequence for the differential fundamental group scheme.

Let $S$ be a smooth affine curve over $k$ and $\eta:S\longrightarrow X$ be a section to $f$. 
Then by the general Tannakian duality (cf. Appendix \ref{sect_tann_dual}) $\eta$ determines the \textit{differential fundamental groupoid scheme} $\Pi(X/k)$ as well as
the \textit{relative differential fundamental group scheme} $\pi(X/S)$.

An absolute connection on $X/k$ can be \textit{inflated} to a relative
connection on $X/S$ (cf. \ref{sect-GM})   
This gives a 
homomorphism $\pi(X/S)\longrightarrow \Pi(X/k)$.
This map fits into a sequence
\begin{align}\label{eq-3}
    \pi(X/S)\longrightarrow \Pi(X/k)\longrightarrow \Pi(S/k)\longrightarrow 1,
\end{align}
which is shown to be exact (see  \ref{sec-exact}), and we call it the \textit{fundamental exact sequence}. More precisely, we introduce the
category of relative connections of geometric origin, which determines 
the  \textit{geometric relative fundamental group} $\pi^\text{geom}(X/S)$. Then the above 
sequence splits into the following exact sequences:
$$\pi(X/S)\longrightarrow\pi^\text{geom}(X/S)\longrightarrow1$$
and
\begin{equation}\label{IntFES}1\longrightarrow \pi^\text{geom}(X/S)\longrightarrow \Pi(X/k)\longrightarrow \Pi(S/k)\longrightarrow 1.\end{equation}

We note that this result generalizes Theorem 5.11 of \cite{EH06} and
was implicitly mentioned in \cite[Theorem~9.2]{DHdS18}. It was formulated
and proved in the PhD thesis of Hugo Bay-Rousson \cite{BR19}; 
however, some parts of the argument in the proof there are missing. For the sake of completeness, 
we provide here a detailed proof. 

Given the above exact sequences, our next step is to compare the de Rham cohomology with
the group cohomology and to express the Gauss-Manin connection in terms of group 
cohomology. Let $(\mathcal V,\nabla)$ be a flat connection on $X/k$ and $(\mathcal V,\nabla_{/S})$
the inflated connection. 
Let $V:=\eta^*(\mathcal V)$ be the ``fiber'' at $\eta$. The exact sequence \eqref{IntFES} yields a map
$$
\mathrm{H}^{0}(\pi^{\mathrm{geom}}(X/S), V) \longrightarrow \mathrm{H}_\mathrm{dR}^{0}(X/S, (\mathcal V,\nabla_{/S})),
$$
which is bijective, and hence the natural maps of the derived functors
$$ 
 \beta^i: \mathbf{R}_{\mathrm{Rep}(S:\Pi(X / k))}^{i} \mathrm{H}^{0}(\pi^{\mathrm{geom}}(X/S), V) \longrightarrow 
\mathbf{R}_{\mathrm{MIC}(X/k)}^{\mathrm{i}} 
\mathrm{H}_{\mathrm{dR}}^{0}(X/S, (\mathcal V,\nabla_{/S}))=
\mathrm{H}_{\mathrm{dR}}^{\mathrm{i}}(X/S, (\mathcal V,\nabla_{/S})).
$$ 
We note that the derived functors on the source of the map are taken in $\mathrm{Rep}(S:\Pi(X / k))$
which is smaller than the category $ \mathrm{MIC}(X/k)$ of all connections on $X/k$ where the derived functors on the 
target of the map are taken. Similar constructions yield  the following diagram:
$$\xymatrix{
\mathrm{H}^i(\pi^\mathrm{geom}(X/S),V)\ar[r]^{\gamma^i} & \mathrm{H}^i(\pi(X/S), V) \ar[d]^{\delta^i}\\
 \mathbf{R}_{\mathrm{Rep}(S:\Pi(X / k))}^{i} \mathrm{H}^{0}(\pi^{\mathrm{geom}}(X/S), V) 
 \ar[u]^{\alpha^i}\ar[r]_{\qquad \beta^i}&
  \mathrm{H}^i_\mathrm{dR}(X/S, (\mathcal{V},\nabla_{/S})).
}$$

  \begin{theorem}[Theorem \ref{thm-gmt}, Lemma \ref{lem_comparison_functors} 
and Theorem \ref{thm-comparison}, Corollary \ref{cor-114}]
\label{thm-equivariant}
Let $f:X\longrightarrow S$ be a smooth projective relative curve of genus $g\geq 1$ over a smooth affine curve $S,$ equipped with a section $\eta:S\longrightarrow X$. 
Let $(\mathcal V,\nabla)$ be a vector bundle with flat connection on $X/k.$ Then: 
\begin{enumerate}
\item 
For $i\geq 0,$ the diagram above is commutative and all maps in it are isomorphisms. 
\item The Gauss-Manin connection on 
$\mathrm{H}^i_\mathrm{dR}(X/S, (\mathcal{V},\nabla_{/S}))$ corresponds to the action of 
$\Pi(S/k)$ on $\mathrm{H}^i(\pi^\mathrm{geom}(X/S),V)$ 
induced from the Lyndon-Hochschild-Serre sequence. 
\item The morphisms $\gamma^i$ and $\delta^i$ are isomorphisms for all $i\geq 0$ and $V \in \mathrm{Obj}(\mathrm{Rep}^\mathrm{f}(\pi^\mathrm{geom}(X/S))).$
\end{enumerate}
\end{theorem} 

\begin{proposition}[Proposition \ref{pro_surf_kpi1}]
Let $f:X\longrightarrow S$ be a relative smooth projective curve with genus $g\geq 1$ over a base $S$ which is 
a smooth affine curve. Assume that $f$ admits a section. Then we can shrink $S$ to an 
open $U$ so that $X_{U}$ as a $k$-surface is de Rham $K(\pi,1)$.
\end{proposition}

This result can be compared with the known result for \'etale $K(\pi,1)$, cf. \cite[Theorem~2]{Sch96}.

\subsection{Outline of the paper}
\begin{enumerate}
\item Sections \ref{sect_flat_conn} and \ref{sect_fund_group} are preliminaries.
We recall the basic notions of de Rham cohomology, Gauss-Manin connections and the
Tannakian affine group and groupoid schemes associated with various categories of connections.

\item   In Section \ref{sect_exact_seq} we prove the exactness of the fundamental exact sequence 
\eqref{eq-3}, see Theorem \ref{FES}. We introduce the affine group scheme $\pi^\mathrm{geom}(X/S)$
which  is Tannakian dual to the category of connections on $X/S$ of ``geometric origin'' . These are
connections that can be presented as sub-quotients of connections on $X/k$ inflated to $X/R$.
The proof is lengthy with many lemmas being just extensions of results obtained previously
for schemes over discrete valuation rings in \cite{DHdS18}, so some proofs are put in the Appendix. 
 
\item Given this short exact sequence, we construct natural maps to compare different group cohomologies.
The construction of our comparison maps follows the pattern given in Lemma  \ref{lem_comparison}. 
In Theorem \ref{thm-gmt} we show that
the group cohomology of $\pi^\mathrm{geom}(X/S)$ with coefficient in a representation
of $\Pi(X/k)$ is isomorphic with the derived functors of the  $\pi^\mathrm{geom}(X/S)$-invariant
functors in $\mathrm{Rep}(S:\Pi(X/k))$. As a consequence, we obtain a version of the
Lyndon-Hochschild-Serre spectral sequence for the cohomology of $\Pi(X/k)$, 
cf.  Proposition \ref{prop-21}. 

\item Section \ref{sect-compa} is devoted to the proof of the comparison theorem mentioned
above. The comparison for the first cohomology groups:
$$ \mathrm{H}^1(\pi^\mathrm{geom}(X/S),V) \cong \mathrm{H}^1(\pi(X/S),V)$$
follows from the \textit{universal extension theorem} (see Theorem \ref{univ-ext}).
The idea of the proof for higher cohomology groups is to check the bijectivity on each fiber. 
 
For closed fibers, we use a generalization of the shifting dimension argument  (see Lemma \ref{lem-injective}).
The generic fiber case is more challenging. Proposition \ref{130}
treats the case of $X$ has genus $g\geq 2$. It  relies on the following results:
    \begin{itemize}
        \item [$\bullet$] Poincar\'e duality for de Rham cohomology on $X_K$.
        \item [$\bullet$] The existence of infinitely many non-isomorphism simple connections in 
        $\mathrm{MIC}^{\mathrm{geom}}(X/S)_K.$
        \item [$\bullet$] The non-vanishing of the first de Rham cohomology $\mathrm{H}^{1}_{\mathrm{dR}}(X/K, \mathcal{V}_K)$.
    \end{itemize}
The case $g=1$ follows from the fact that the fundamental group $\pi(X_K/K)$ is commutative.

\item In Appendix \ref{app_group_scheme}, we recall definitions and basic notions of flat affine groups  
and homomorphisms between them, for instance:  the Tannakian description of exact sequences 
homomorphisms of affine group schemes (Theorem \ref{th1}) and (Theorem \ref{thm-ExactSequence}).  In 
Appendix \ref{sect_morphism_group_schemes}, we prove Theorem \ref{lem-DHdS}, which generalizes the 
results in \cite{DHdS18}, where the ring of interest is  a  discrete valuation ring. The proof is similar, but we 
need to extend the necessary results to Dedekind rings. These results are given in Proposition \ref{proplocaltoglobal} and Proposition \ref{PropTannaka}.

\item In Appendix \ref{sect_repr}, we recall the concept of an affine $k$-groupoid scheme acting on $S$ and its representations in \cite[Section 3]{De90}.    In Appendix \ref{sect_tann_dual},   we review the general Tannakian duality.  In Appendix \ref{sub-induction},  we  define the induction functor of groupoid schemes homomorphism  (see \ref{A.2.5}), which helps us  prove that any $\Pi(X/k)^{\Delta}$-module is a quotient of a $\Pi(X/k)$-module (see Corollary \ref{lem-RepGDiagonal}).  In Appendix \ref{sect_cohomology_of_groupoid}, we define the cohomology of groupoid schemes, which helps us to prove Theorem \ref{thm-gmt}.

\end{enumerate}

\subsection{Notations and Conventions}
\begin{enumerate}
    \item  $k$ is algebraically closed field of characteristic 0.
    \item  $R$ is a Dedekind domain with quotient field $K$ and residue field $k.$
    \item $S$ is the spectrum of $R$, $f:X\longrightarrow S$ is smooth morphism.
    \item $x:\mathrm{Spec}(k)\longrightarrow X$ is a $k$-point of $X$ and similarly 
    $s$ is a $k$-point  of $S$. $\eta: S\longrightarrow X$ is a section to $f$, with its existence, 
    we shall assume that $\eta s=x$. 
    \item  $\mathrm{MIC}(X/S)$ denotes the category of quasi-coherent
$\mathcal O_X$-sheaves equipped with an $R$-linear flat connection. 
 \item  For a connection $(\mathcal V,\nabla)$ in
$\mathrm{MIC}(X/k)$ (i.e. an absolute connection), 
the inflated connection to $X/S$ is denoted by  $(\mathcal V,\nabla_{/S})$ or $\mathrm{inf}(\mathcal{V})$. 
    \item   $\mathrm{MIC}^\mathrm{coh}(X/S)$ denotes the full subcategory of $\mathrm{MIC}(X/S)$ of objects whose underlying sheaf is coherent. 
    \item $\mathrm{MIC}^{\circ}(X/S)$ denotes the full subcategory of $\mathrm{MIC}(X/S)$ of objects whose underlying sheaf is locally free.  
    \item $\mathrm{MIC}^\mathrm{se}(X/S)$ denotes the full subcategory of 
$\mathrm{MIC}(X/S)$ of objects which can be presented as quotients of objects from $\mathrm{MIC}^{\circ}(X/S)$.
     \item  $\mathrm{MIC}^\mathrm{geom}(X/S)$ denotes the full subcategory of $\mathrm{MIC}^\mathrm{se}(X/S)$ of objects which can be presented as a sub-quotient of an inflated object from 
     $\mathrm{MIC}^\circ(X/k)$.
   
    \item  When we compare the group cohomology of the fundamental group and the de Rham cohomology, we always denote by $V$ representation corresponding to the flat connection $(\mathcal{V},\nabla).$
    \item For a flat affine group scheme $G$ over $S$, the category of representations of $G$
    in $R$-modules is denoted by $\mathrm{Rep}_R(G)$, the subcategories of representations in
    finite  (resp. finite projective) $R$-modules  will be denoted by 
    $\mathrm{Rep}^\mathrm{f}_R(G)$ (resp. $\mathrm{Rep}^\circ_R(G)$).
    \item If $G$ is a groupoid scheme acting upon $S$, the category of representations of $G$ in
    (finite)  $R$-modules is denoted by $\mathrm{Rep}(R:G)$ 
    (resp. $\mathrm{Rep}^\mathrm{f}(R:G)$).
\end{enumerate}

\section{Flat connections and  de Rham cohomology}\label{sect_flat_conn}
Let $k$ be an algebraically closed field of characteristic $0,$ and let 
$R$ be a Dedekind domain over $k$, $S:=\mathrm{Spec} (R)$.
Let $f:X\longrightarrow S$ be a smooth morphism with geometrically connected fibers.

\subsection{Connections}
\label{sect-connection}
Let $\Omega^1_{X/S}$ denote the sheaf of relative K\"ahler differentials on $X/S$. Since 
$X$ is smooth over $S,$ $\Omega^1_{X/S}$ is a locally free sheaf. 

\subsubsection{Flat connections} A  \textit{connection} on a quasi-coherent sheaf of $\mathcal O_X$-modules 
$\mathcal M$ on $X$ is an $R$-linear map
$$\nabla:\mathcal M\longrightarrow \Omega^1_{X/S} \otimes \mathcal{M},$$
satisfying the Leibniz rule and is \textit{flat}
in the sense that the composed map $\nabla_1\circ\nabla=0$, where
$$\nabla_1: \Omega^1_{X/S}\otimes \mathcal M\longrightarrow 
\Omega^2_{X/S}\otimes\mathcal M; \quad \omega\otimes e\longmapsto d\omega\otimes e-\omega\otimes\nabla(e)$$
(cf. \cite[(1.0)]{Ka70}).
When no confusion may arise, we shall address a sheaf with a flat connection simply as a connection. The notation is $(\mathcal M,\nabla)$ and is usually abbreviated to $\mathcal M$.


\subsection{De Rham cohomology}\label{sect-deRham-cohomology}
Let $f: X \longrightarrow S = \mathrm{Spec}(R)$ be a smooth scheme over a noetherian ring $R.$
For a sheaf with flat connection $(\mathcal{M},\nabla)$ on $X/S$, 
the \textit{sheaf of horizontal sections} is defined to be 
\begin{align}\label{H0} 
\mathcal M^\nabla:= \mathop{\mathrm{Ker}}(\nabla : \mathcal{M} \to 
\Omega ^1_{X/S} \otimes \mathcal{M}).     
\end{align}
This is a sheaf of $R$-modules. The \textit{$0$-th de Rham cohomology} of $\mathcal M$ is defined to be 
$$\mathrm{H}_\mathrm{dR}^0(X/S,  (\mathcal M,\nabla)):=f_*(\mathcal M^\nabla).$$
The module $\mathrm{H}^0_\mathrm{dR}(X/S, (\mathcal{M},\nabla))$ can be identified with the hom-set of connections in $\mathrm{MIC}(X/S)$
$$\{\varphi:(\mathcal O_X,d) \longrightarrow (\mathcal M,\nabla)\}.$$

Since $\text{\rm MIC}(X/S)$ is equivalent to the category of left modules on the sheaf of differential operators $\mathcal{D}_{X/S}$, it has enough injectives (cf. \cite{Ka70}).
Thus, we can define the \textit{higher de Rham cohomologies }$\mathrm{H}^i_\mathrm{dR}(X/S, -)$ to be the derived functors of the functor 
$$\mathrm H^0_\mathrm{dR}(X/S, -):\mathrm{MIC}(X/S)\longrightarrow \mathrm{Mod}_R.$$

\noindent{\bf Convention.} When it is clear which connection on a sheaf $\mathcal M$
is taken, we shall omit it in the notation of the de Rham cohomology and simply write:
$$\mathrm{H}^i_\mathrm{dR}(X/S, \mathcal{M}).$$
This cohomology group can also be computed as the groups of extensions in $\mathrm{MIC}(X/S)$,
$\mathrm{Ext}^i_{\mathrm{ MIC}(X/S)}(\mathcal M,\mathcal N)$, which count $i$-extensions in ${\mathrm{ MIC}(X/S)}$, that is, exact sequences
$$0\longrightarrow\mathcal  N\longrightarrow\mathcal  N_1\longrightarrow \ldots\longrightarrow\mathcal 
N_i\longrightarrow\mathcal  M\longrightarrow 0,$$
up to an equivalence define as follows: two sequences are equivalent if there is a map of sequence between them which restricts to the identity maps on $\mathcal M$ and $\mathcal N$. 
Since $\mathrm{MIC}(X/S)$ has enough injectives, $\mathrm{Ext}^i_{\mathrm{ MIC}(X/S)}(\mathcal M,-)$ are the right derived functors of the hom-functor
$$\Hom_{\mathrm{ MIC}(X/S)}(\mathcal M,-) :\mathrm{ MIC}(X/S)\longrightarrow
\mathrm{Mod}_R.$$
That is, we have the following:
\begin{lem}\label{lemm1}
Let $(\mathcal{M},\nabla)$ be a flat connection in $\mathrm{MIC}(X/S)$. 
Then we have:
\begin{center}
$\mathrm{Ext}^i_{\mathrm{ MIC}(X/S)}(\mathcal O_X,\mathcal{M}) \cong \mathrm{H}^i_\mathrm{dR}(X/S, \mathcal{M}).$
\end{center}
\end{lem} 

According to \cite{Gr69} or \cite{HKR62}, the de Rham cohomology can be computed in terms of the hyper-derived functors on the de Rham complex:
$$\mathrm{H}^i_\mathrm{dR}(X/S, \mathcal{M})\cong \mathbf{R}^if_{\ast}(\Omega^{\bullet}_{X/S}\otimes \mathcal{V}).$$

\subsection{The Gauss-Manin connection}\label{sect-GM} 
We briefly describe the construction of Gauss-Manin connection which is based on \cite{Ka70}, 
see also \cite{ABC20}. 

\subsubsection{Absolute connections}
If $R$ is of finite type (resp. essentially of finite type) over $k$, then
$X$ is smooth (resp. essentially smooth) as a scheme over $k$. 
We denote by $\mathrm{MIC}(X/k)$  the category of $\mathcal O_X$-quasi-coherent sheaves 
with $k$-linear flat connections. 
It is known that each $\mathcal O_X$-coherent sheaf with flat connection over $k$ is locally free, and the 
dual sheaf is equipped with a (dual) connection in a canonical way. Hence, 
$\mathrm{MIC}^\mathrm{coh}(X/k)=\mathrm{MIC}^{\circ}(X/k)$ 
is a rigid tensor $k$-linear abelian category. 
Similarly, we have the categories $\mathrm{MIC}^{\circ}(S/k)$ 
and $\mathrm{MIC}^\mathrm{se}(X/S),$ where $S=\mathrm{Spec}(R).$  
The pull-back by $f$ yields the functor (denoted by the same notation)
 $f^*: \mathrm{MIC}^{\circ}(S/k)\longrightarrow \mathrm{MIC}^{\circ}(X/k).$ 

\subsubsection{The inflation functor}
We have the \textit{inflation functor}
$$\inf:\mathrm{MIC}(X/k)\longrightarrow \mathrm{MIC}(X/S)$$
which assigns to each connection $(\mathcal V,\nabla)$ in $\mathrm{MIC}(X/k)$ the $R$-linear connection $(\mathcal V,\nabla_{/S})$ in $\mathrm{MIC}(X/S)$:
$$\nabla_{/S}:\mathcal V\stackrel{\nabla}\longrightarrow
\Omega^1_{X/k}\otimes \mathcal V
\longrightarrow \Omega^1_{X/S}\otimes \mathcal V.$$
This relative connection is called an {\em inflated connection}, and is denoted by 
$(\mathcal V, \nabla_{/S})$ or $\inf(\mathcal V)$.

Let $(\mathcal V,\nabla)$ be an absolute connection, that is, an object of $\mathrm{MIC}(X/k)$.
Consider the de Rham cohomology of the inflated connection $\mathrm{inf}(\mathcal V)=
(\mathcal V,\nabla_{/S})$.
In what follows we shall adopt the abbreviation:
$$\mathrm{H}_\mathrm{dR}^0(X/S, \mathcal V):=\mathrm{H}_\mathrm{dR}^0(X/S, (\mathcal V, \nabla_{/S})),$$
when it is clear that $\mathcal{V}$ is equipped with an absolute connection and the cohomology is taken 
over $X/S$ for the inflated connection on it.

\subsubsection{The Gauss-Manin connection}
The $R$-module $\mathrm{H}_\mathrm{dR}^0(X/S, \mathcal V)$,  is
equipped with a connection over $S/k$, known as
the  \textit{0-th Gauss-Manin connection}. The explicit construction is as follows. 
The smoothness of $f$ implies the following exact sequence of K\"ahler differentials:
\begin{equation}\label{eq-kaehler}
0\longrightarrow f^*\Omega^1_{S/k} \longrightarrow
\Omega^1_{X/k}\longrightarrow \Omega^1_{X/S}\longrightarrow 0.
\end{equation} 
This filtration of $\Omega^1_{X/k}$ induces a filtration on $\Omega^2_{X/k}$ which is compatible with the connection, and we obtain the following commutative diagram (cf. \cite[(3.2)]{Ka70}) with 
exact columns: 
$$
\xymatrix{
&& f^*\Omega^1_{S/k}\otimes\mathcal V\ar[r]_{\nabla_0\quad } \ar[d]_i & 
 f^*\Omega^1_{S/k}\otimes\mathcal \Omega^1_{X/k}\otimes V\ar[d]_j\\
&\mathcal V\ar[r]^{\nabla\quad }\ar@{=}[d]& \Omega^1_{X/k}\otimes\mathcal V\ar[d]\ar[r]^{\nabla_1} &
\Omega^2_{X/k}\otimes\mathcal V\ar[d]  \\
\mathcal V^{\nabla_{/S}}\ar[r]\ar@/^20pt/[rruu]^\partial&
 \mathcal V\ar[r]_{\nabla_{/S}}&\Omega^1_{X/S}\otimes\mathcal V\ar[r]& 
\Omega^2_{X/S}\otimes \mathcal V,}$$
here $\nabla_0: f^*\Omega^1_{S/k}\otimes\mathcal V\longrightarrow 
f^*\Omega^1_{S/k}\otimes \Omega^1_{X/k}\otimes\mathcal V$ is given by 
(for   $r\in R$  and $v\in \mathcal V$):
$$ dr\otimes v\longmapsto  - dr\otimes \nabla(v),$$
and  $\nabla_1:\Omega_{X/k}^1\otimes \mathcal V\longrightarrow
\Omega_{X/k}^2\otimes \mathcal V$ is given by 
(for $a \in \mathcal O_X$   and $v\in \mathcal V$)
$$da\otimes v\longmapsto  - da\wedge \nabla(v).$$
 
Recall that $\mathcal V^{\nabla_{/S}}$ is the kernel of the
$R$-linear map $\nabla_{/S}$. 
Diagram chasing yields a map 
$$\partial: \mathcal V^{\nabla_{/S}}\longrightarrow f^* \Omega^1_{S/k}\otimes\mathcal V, $$
with the property $i\circ\partial=\nabla$.  
Hence $j\circ\nabla_0\circ\partial=\nabla_1\circ\nabla=0$. Therefore, $\partial$ factors as:
$$\partial: \mathcal V^{\nabla_{/S}}\longrightarrow \Omega^1_{S/k}\otimes_R
\mathcal V^{\nabla_{/S}}.$$
Applying $f_*,$ we obtain a flat connection $\partial$ on 
$\mathrm H^0_\mathrm{dR}(X/S,\mathcal V)$ over $S/k$. 
The resulting connection is denoted by 
$\mathbf{R}_\mathrm{dR}^0f_*(\mathcal V,\nabla_{/S})$ or 
$\mathbf{R}_\mathrm{dR}^0f_*\mathcal V$ for short. 
Thus, we have a left exact functor
$$\mathbf{R}_\mathrm{dR}^0f_*:
\mathrm{MIC}(X/k)\longrightarrow \mathrm{MIC}(S/k).$$
The $i$-th derived functor of this functor is called the \textit{$i$-th Gauss-Manin connection}:
$$\mathbf R_\mathrm{dR}^if_*:\mathrm{MIC}(X/k)\longrightarrow \mathrm{MIC}(S/k).$$

\subsubsection{Expression in terms of hyper-cohomology}
As argued in \cite[23.2.5]{ABC20}, the module 
$\mathbf{R}_\mathrm{dR}^if_*(\mathcal{V},\nabla)$
is canonically isomorphic to the hyper-cohohomology $\mathbf{R}^if_{\ast}(\Omega^{\bullet}_{X/S}\otimes \mathcal{V}).$
The sequence in \eqref{eq-kaehler} yields an exact sequence of complexes: 
 \begin{equation}\label{eq-ExSeqDiff}
0 \longrightarrow f^*{\Omega^1_{S/k}} \otimes (\Omega^{\bullet-1}_{X/S}\otimes \mathcal{V}) \longrightarrow  \Omega^{\bullet}_{X/k}\otimes \mathcal{V} \longrightarrow \Omega^{\bullet}_{X/S}\otimes \mathcal{V} \longrightarrow 0. \end{equation}
Applying the hyper-derived functor $\mathbf{R}^{i}f_{*}$ 
to this exact sequence of complexes, we obtain the long exact sequence:
 \begin{center}
     $0 \longrightarrow \cdots \longrightarrow \mathbf{R}^if_{*}(\Omega^{\bullet}_{X/S}\otimes \mathcal{V}) \overset{\partial^i}{\longrightarrow} \mathbf{R}^{i+1}f_{*}(f^*{\Omega^1_{S/k}} \otimes (\Omega^{\bullet-1}_{X/S}\otimes \mathcal{V})) \longrightarrow \cdots.$
 \end{center}
 The connecting map 
 $$\partial^i: \mathbf{R}^if_{*}(\Omega^{\bullet}_{X/S}\otimes \mathcal{V}) \longrightarrow \mathbf{R}^{i+1}f_{*}(f^*{\Omega^1_{S/k}} \otimes (\Omega^{\bullet-1}_{X/S}\otimes \mathcal{V})) $$
 can be written as:
 \begin{center}
     $\partial^i: \mathbf{R}^if_{*}(\Omega^{\bullet}_{X/S}\otimes \mathcal{V}) \longrightarrow \Omega^1_{S/k} \otimes \mathbf{R}^if_{*}(\Omega^{\bullet}_{X/S}\otimes \mathcal{V})$
 \end{center}
 by projection formula. 
The map $\partial^i$ defines a flat connection on the $R$-module 
$\mathbf{R}^if_{*}(\Omega^{\bullet}_{X/S}\otimes \mathcal{V})=
\mathrm{H}^i_{\mathrm{dR}}(X/S, \mathcal{V}).$
Moreover, we get a long exact sequence:
 \begin{align}\label{GM-hyper}
\cdots \rightarrow \mathrm{H}^i_\mathrm{dR}(X/k,(\mathcal{V},\nabla)) \rightarrow \mathrm{H}^i_\mathrm{dR}(X/S, (\mathcal V, \nabla_{/S})) \overset{\partial^i}{\rightarrow} 
\Omega^1_{S/k} \otimes \mathrm{H}^i_\mathrm{dR}(X/S, (\mathcal{V},\nabla_{/S})) 
\rightarrow \cdots
\end{align}
 
\subsubsection{Finiteness of de Rham cohomology}
\begin{lem}\label{coh1}
Let $f: X \longrightarrow \Spec (R)$ be a proper smooth morphism of smooth $k$-schemes. Let 
$\mathcal{V}$ be an object in $\mathrm{MIC}^\circ(X/k)$. 
Then $\mathrm{H}^i_\mathrm{dR}(X/S,\mathcal V)$ is a finite 
projective $R$-module.
\end{lem}
\begin{proof} 
We have the Hodge-to-de Rham spectral sequence (see \cite[(3.5.2.0)]{Ka70}):
\[ E_{1}^{p,q} = \mathrm{R}^qf_{*}(\Omega^p_{X/S}\otimes \mathcal V) \Longrightarrow \mathbf{R}^{p+q}f_{\ast}(\Omega_{X/S}^{\bullet}\otimes\mathcal V).\]
 Since $X$ is proper,  every term of $$E^{p,q}_1 = \mathrm{R}^qf_{\ast}(\Omega^p_{X/S}\otimes \mathcal V)$$ is coherent. It follows  that $\mathbf{R}^{p+q}f_{\ast}(\Omega_{X/S}^{\bullet}\otimes\mathcal V)$ is also coherent. Moreover, because the  de Rham cohomology sheaf $\mathrm{H}^i_\mathrm{dR}(X/S,\mathcal V)$ can be equipped with Gauss-Manin connection, the underlying $R$-module  must be projective as an $R$-module.
 \end{proof}

Consequently, the functors $\mathbf{R}_{\mathrm{dR}}^if_*$ restrict to the functors $\mathrm{MIC}^\circ(X/k)\longrightarrow
\mathrm{MIC}^\mathrm{\circ}(S/k)$ and we have the following commutative diagram:
\[ \begin{tikzcd}
 \mathrm{MIC}^\circ(X/k) \arrow{r}{\inf } \arrow[swap]{d}{\mathrm R^i_{\mathrm{dR}}f_*(-)} & \mathrm{MIC}^\mathrm{se}(X/S) \arrow{d}{\mathrm{H}^i_\mathrm{\mathrm{dR}}(X/S,-)} \\%
\mathrm{MIC}^{\circ}(S/k) \arrow{r}{\Gamma(S,-)} & \mathrm{ Mod}^\text{coh}(R),
\end{tikzcd}
\]
where $\mathrm{ Mod}^\text{coh}(R)$ is the category of finite modules over $R$.
\\
\\

Recall that a connection on $X/S$ is called \textit{trivial} if it has the form $M\otimes_R(\mathcal O_X,d)$
where $M$ is an $R$-module. More precisely, the underlying sheaf has the form $M\otimes_R\mathcal O_X$ and the connection is induced from the differential on $\mathcal O_X$.
\begin{lem}\label{lem-condA}
Let $(\mathcal{V},\nabla)$ be an object of $\mathrm{MIC}^\circ(X/k)$. Then the connection
$f^*(\mathbf{R}_\mathrm{dR}^0f_*(\mathcal V, \nabla))$ is the maximal
sub-object of $(\mathcal{V},\nabla)$ in $\mathrm{MIC}^\circ(X/k)$ with the property: 
its inflation to $\mathrm{MIC}^\mathrm{coh}(X/S)$ is a trivial connection. 
\end{lem} 
\begin{proof}
The connection on the pull-back 
$$f^*\mathbf{R}_\mathrm{dR}^0f_*\mathcal V=
f^*\mathrm H^0_\mathrm{dR}(X/S,\mathcal V)=
 \mathcal O_X\otimes \mathrm H^0_\mathrm{dR}(X/S,\mathcal V)$$
is given by $\nabla'(a\otimes e)= da\otimes e+a\otimes \partial e$. Hence on the inflated connection
$$\nabla'_{/S}(a\otimes e)=da\otimes e.$$
This means $(\mathcal O_X \otimes H^0_\mathrm{dR}(X/S,\mathcal V) ,\nabla')$ 
is a sub-connection of $(\mathcal V,\nabla)$ with the property that its inflation is a trivial relative connection.  

Actually, it has to be the maximal such as any other $\mathcal{W}\subset \mathcal{V}$ would have the property that $(\nabla'_{/S})|_{\mathcal{W}}$ is generated by horizontal sections.
\end{proof}

\begin{rem}\label{rmk-condA}
The Tannakian interpretation of Lemma \ref{lem-condA} is as follows. Let $V=\eta^*(\mathcal V)$ be the 
representation of $\Pi(X/k)$ that corresponds to $(\mathcal{V},\nabla).$ Then $f_\ast\inf(\mathcal V)
$ corresponds through $\eta^*$ to a representation of $\Pi(S/k)$
and $f^*f_*\inf(\mathcal{V})$ corresponds to a sub-representation of $V$ on which the action of 
$\Pi(X/k)$ factors through the action of $\Pi(S/k)$, or equivalently, the group $L$ acts trivially, where 
$L$ is the kernel of $f:\Pi(X/k)\longrightarrow \Pi(S/k)$, see Appendix \ref{sect_repr}. 
\end{rem}

\section{Fundamental groups and groupoids}\label{sect_fund_group}
Let $\eta: S=\text{\rm Spec} (R)\longrightarrow X$ be an $R$-point of $X$ (as an $R$-scheme).
This section also yields a $k$-point $x$ in the fiber $X_s$ of $X$:
$$\xymatrix{
\mathrm{Spec}(k)\ar[r]^s\ar[d]_x & \text{\rm Spec}(R)\ar[d]^\eta\\
X_s\ar[r]& X.
}$$
These points yield various Tannakian group and groupoid schemes which we shall refer to as
groups and groupoids for short.
\subsection{Categories of consideration} We consider the following categories:
\begin{itemize}
\item $\mathrm{MIC}(X/S)$ -- the category of \textit{quasi-coherent}
$\mathcal O_X$-sheaves equipped with $R$-linear flat connections;
\item $\mathrm{MIC}^\mathrm{coh}(X/S)$ -- the full subcategory of \textit{coherent}
$\mathcal O_X$-sheaves equipped with $R$-linear flat connections;
\item $\mathrm{MIC}^{\circ}(X/S)$ -- the full subcategory of locally free
$\mathcal O_X$-sheaves equipped with $R$-linear flat connections;
\item $\mathrm{MIC}^\mathrm{se}(X/S)$ -- the full subcategory of  $\mathrm{MIC}^\mathrm{coh}(X/S)$
consisting of objects which can 
be presented as quotients of objects from $\mathrm{MIC}^{\circ}(X/S)$.
\end{itemize}
$\mathrm{MIC}^\mathrm{se}(X/S)$ satisfies Serre's condition (whence the
superscription ``se''). Hence, it is a Tannakian category over $R$
in the sense of Saavedra (cf. \cite{Sa72} or \cite{DH18}). 

According to \cite{DH18}, an object in $\mathrm{MIC}^\mathrm{coh}(X/S)$  is locally free if (and only if) it is torsion free over $R$. We recall the following concept for easy referencing.
\begin{defn}[Special sub-quotients]
\label{defn_speical_subquotient}
Let $\mathcal{M}$ be an object in $\mathrm{MIC}^{\circ}(X/S).$ We write $T^{a,b}(\mathcal{M})$ for  $\mathcal{M}^{\otimes a}\otimes \mathcal{M}^{\vee \otimes b}.$
\begin{itemize} 
    \item [(i)] The full subcategory in $\mathrm{MIC}^\mathrm{coh}(X/S)$
of all sub-quotients of objects of the form 
$T^{a_1, b_1} (\mathcal{M}) \oplus \cdots \oplus T^{a_m, b_m} (\mathcal{M})$ is 
denoted by $\langle\mathcal{M}\rangle_{\otimes}.$
    \item [(ii)] We recall the notion of a special sub-quotient of a connection whose sheaf is locally free as 
an $\mathcal{O}_X$-module. A monomorphism (i.e. injective map) $\alpha: \mathcal{M}^{\prime} \rightarrow \mathcal{M}$
is said to be special if $\operatorname{Coker}(\alpha)$ is locally free. Then $\mathcal{M}^{\prime}$
is called a special sub-object of $\mathcal{M}$. 
\item[(iii)]
We call an object $\mathcal{M}^{\prime \prime}$ in $\mathrm{Obj}(\mathrm{MIC}^{\circ}(X/S))$ a \textit{special sub-quotient} of $\mathcal{M}$ if there exists a special 
monomorphism $\mathcal{M}^{\prime} \rightarrow \mathcal{M}$ and 
an epimorphism $\mathcal{M}^{\prime} \rightarrow \mathcal{M}^{\prime \prime}.$  
\item[(iv)] The category of all special sub-quotients of various 
$T^{a_1, b_1} (\mathcal{M}) \oplus \cdots \oplus T^{a_m, b_m} (\mathcal{M})$ is denoted by $\langle\mathcal{M}\rangle_{\otimes}^s.$
\end{itemize}
\end{defn} 

\subsection{The fundamental groups}
The assumption that $f:X\longrightarrow \mathrm{Spec}(R)$ has connected fibers allows us to define
the  fundamental group $\pi(X/k)=\pi(X/k,x)$ is the Tannakian dual of $\mathrm{MIC}^\circ(X/k)$ with 
respect to the fiber functor $x^*$ (cf. Appendix \ref{the-main-theoremA}):
$$\mathrm{Rep}^\mathrm{f}(\pi(X/k))\cong \mathrm{MIC}^\circ(X/k).$$
The fundamental group $\pi(S/k)$ is the Tannakian dual of $\mathrm{MIC}^\circ(X/k)$ with respect to the fiber functor $s^*$:
$$\mathrm{Rep}^\mathrm{f}(\pi(S/k))\cong \mathrm{MIC}^\circ(S/k).$$
The map $f:X\longrightarrow S$ induces a group homomorphism
$$f_*: \pi(X/k)\longrightarrow \pi(S/k).$$
This map is surjective as it admits a section induced from the section $\eta:S\longrightarrow X$.

\subsection{The fundamental groupoid}\label{subsub-1}
Assume that $\mathrm{End}_{\mathrm{MIC}^\circ(X/k)}(R,d_{S/k}) = k$, then by 
\cite[Th\'eor\`eme 1.12]{De90} we have the absolute fundamental groupoid (scheme) 
$\Pi(X/k)=\Pi(X/k,\eta)$ which is the Tannakian dual of $\mathrm{MIC}^\circ(X/k)$ 
with respect to the fiber functor $\eta^*$ (cf. Appendix \ref{the-main-theoremB} for more details):
$$\mathrm{Rep}^\mathrm{f}(S:\Pi(X/k))\cong \mathrm{MIC}^\circ(X/k).$$
The absolute fundamental groupoid $\Pi(S/k) = \Pi(S/k,\id)$ is the Tannakian dual of 
$\mathrm{MIC}^\circ(S/k)$ with respect to the forgetful functor $\id:(V,\nabla)\longmapsto V$:
$$\mathrm{Rep}^\mathrm{f}(S:\Pi(S/k))\cong \mathrm{MIC}^\circ(S/k).$$
The map $f:X\longrightarrow S$ induces a groupoid homomorphism
$$f^{\ast}: \Pi(X/k)\longrightarrow \Pi(S/k).$$
This map is surjective as it admits a section induced from the section $\eta:S\longrightarrow X$. 

We note that $\pi(X/k)$ is the base change of $\Pi(X/k)$ with respect to the map
$(s,s):\Spec (k)\longrightarrow \mathrm{Spec} (R)\times \mathrm{Spec}(R)$.

\subsection{The relative fundamental group scheme}
\label{sect-RFG} 
The \textit{relative fundamental group scheme} $\pi(X/S)=\pi(X/S,\eta)$ is the 
Tannakian dual of $\mathrm{MIC}^\mathrm{se}(X/S)$ with respect to the fiber functor $\eta^{\ast}$ (cf. Appendix \ref{th_duality}): 
$$\Rep^\mathrm{f}(\pi(X/S))\cong \mathrm{MIC}^\mathrm{se}(X/S).$$
By definition, the inflation functor has image in $\mathrm{MIC}^\mathrm{se}(X/S)$:
$$\inf:\mathrm{MIC}^\circ(X/k)\longrightarrow \mathrm{MIC}^\mathrm{se}(X/S).$$
%
%
The functor $\inf$ induces a homomorphism of group-groupoid (schemes):
$\pi(X/S)\longrightarrow \Pi(X/k),$
which factors through the diagonal subgroup scheme $\Pi(X/k)^\Delta$ of $\Pi(X/k)$ (see Appendix \ref{sect_groupoid_scheme}):
$$\xymatrix{\pi(X/S)\ar[r] \ar[rd]&\Pi(X/k)^\Delta\ar[d]\\ & \Pi(X/k).}$$

\subsection{The Ind-categories}\label{sect-indTanDual}
The Tannakian duality mentioned above extends to the ind-categories. 
Recall that the ind-category of a (noetherian) category is the category of inductive directed systems
of object from the original category. We shall use the notation
$\mathrm{MIC}^\mathrm{ind}(X/S)$ to denote the ind-category of  $\mathrm{MIC}^\mathrm{se}(X/S)$
and $\mathrm{MIC}^\mathrm{ind}(X/k)$ to denote the ind-category of
$\mathrm{MIC}^\circ(X/k)$. 
Since the category of all connections is co-complete, these ind-categories
can be identified with the category of connections
which can be presented as the union of their coherent sub-connections (or
equivalently, as direct limits of coherent connections).

We have the following  equivalences:
\begin{eqnarray*}
\Rep(\pi(X/k))\cong\Rep(S:\Pi(X/k))&\cong& \mathrm{MIC}^\mathrm{ind}(X/k)\subsetneq \mathrm{MIC}(X/k);\\ 
\Rep(\pi(X/S))&\cong&\mathrm{MIC}^\mathrm{ind}(X/S)\subsetneq \mathrm{MIC}(X/S);\\
\Rep(\pi^\mathrm{geom}(X/S))&\cong&\mathrm{ind-}\mathrm{MIC}^\mathrm{geom}(X/S)\subsetneq \mathrm{MIC}(X/S).
\end{eqnarray*}
We notice that there are quasi-coherent connections which cannot be presented as a
union of coherent sub-connections; for instance, the sheaf of algebras of differential operators. 
Therefore, the rightmost inclusions above are proper.


\section{Exact sequences of fundamental group and groupoid schemes}\label{sect_exact_seq}
\subsection{The homotopy exact sequence}
We continue to assume that $f: X \to S$ is a smooth projective morphism of smooth schemes over a
field $k$ of characteristic zero.  
The associated sequence of differential fundamental group schemes 
is shown to be exact by L. Zhang \cite{Zha14} in characteristic zero and J. P. 
dos Santos \cite{dS15} in the general case and is called the homotopy exact sequence, which resembles the
topological homotopy exact sequence. 
\begin{thm}[Zhang, dos Santos]
\label{HES}
The following sequence is exact
$$\pi(X_s/k,x)\longrightarrow \pi(X/k, x)\longrightarrow \pi(S/k)\longrightarrow 1.$$
\end{thm}
This exact sequence generalizes Grothendieck's  homotopy exact sequence for \'etale fundamental groups \cite[Th\'eor\`eme~IX.6.1]{SGA1}.

\subsection{The geometric relative fundamental group}\label{sect-groupH}
One of our key observations is that the group scheme which is the image of the map 
$$\inf :\pi(X/S)\longrightarrow \Pi(X/k)^\Delta$$
turns out to be the Tannakian dual to the category of \textit{relative connections of
geometric origin}, i.e., sub-quotients of inflated connections. 

\begin{defn} 
\label{lem-categoryC}
Let $\mathrm{MIC}^\mathrm{geom}(X/S)$ be the full subcategory of $\mathrm{MIC}^\mathrm{se}(X/S)$ of objects which can be presented as sub-quotients of objects of the form $\text{\rm inf}(\mathcal{V})$, where $\mathcal{V}\in \mathrm{Obj}(\mathrm{MIC}^\circ(X/k)$).\end{defn}

\begin{lem}\label{lem-groupH}  The category $\mathrm{MIC}^\mathrm{geom}(X/S)$ defined above, together with the fiber functor $\eta^{\ast},$  is a Tannakian category. 
The canonical map
$\pi(X/S)\longrightarrow \pi^\mathrm{geom}(X/S)$ is surjective (i.e. faithfully flat).  \end{lem}
 
\subsection{The fundamental exact sequence} \label{sec-exact}
The functor $f^*$ and the inflation functor induce the following sequence of homomorphisms of group and groupoid schemes: 
$$\pi(X/S)\stackrel {\inf} \longrightarrow \Pi(X/k)\stackrel {f_\ast} \longrightarrow \Pi(S/k)\longrightarrow 1.$$

We first notice that the composition $f_\ast\circ \mathrm{inf}$ corresponds to the functor that takes the pull-back of connections on $S/k$ along $f$ and then inflates to relative connections on $X/S,$ thereby sending any connection on $S/k$ to a trivial relative connection on $X/S.$ Therefore, it is a trivial homomorphism.

Furthermore, we notice that the map $\mathrm{inf}$ factors through the composition of
$$\pi(X/S) \longrightarrow  \pi^\mathrm{geom}(X/S) 
 \longrightarrow \Pi(X/k).$$
The aim of this section is to show the following theorem.

\begin{thm}\label{FES} Let $f:X\longrightarrow \mathrm{Spec} (R)$ be a smooth projective map with  geometrically connected fibers.
Then the following sequences 
$$\pi(X/S) \longrightarrow  \pi^\mathrm{geom}(X/S) 
 \longrightarrow 1$$
 and 
$$1\longrightarrow \pi^\mathrm{geom}(X/S) \longrightarrow \Pi(X/k)\stackrel {f^\ast}\longrightarrow \Pi(S/k)\longrightarrow 1$$
of flat affine group schemes and groupoids schemes over $R$ are exact. 
\end{thm}

We notice that the exactness of the second sequence amounts to the exactness of the following sequence of $R$-group schemes (see Appendix \ref{sect_morphism_group_schemes}):
\begin{equation}\label{eq_exact_diag}
1\longrightarrow\pi^\mathrm{geom}(X/S) \longrightarrow \Pi(X/k)^\Delta\longrightarrow \Pi(S/k)^\Delta\longrightarrow 1.\end{equation}

The name ``fundamental exact sequence'' is motivated by Grothendieck's fundamental exact sequences of the \'etale fundamental groups:
$$1\longrightarrow \pi^{\rm et}(X\otimes_k\bar k,\bar x)\longrightarrow \pi^{\rm et}(X,\bar x)
\longrightarrow \text{\rm Gal}(\bar k/k)\longrightarrow 1.$$
Indeed, the fundamental groupoid $\Pi(S/k)$ can be seen as a generalization of the Galois group $\text{\rm Gal}(\bar k/k)$, while the relative fundamental group plays the role of the geometric \'etale fundamental group.  

For the \'etale fundamental groups, Grothendieck used the fundamental exact sequence to deduce the homotopy exact sequence. For the differential fundamental groups, dos Santos provided a direct proof using his criterion for the exact sequence of group schemes. In what follows, we shall use Theorem \ref{HES} to deduce Theorem \ref{FES}.  
The proof will be given in Section \ref{proof-FES}. First, we need some lemmas.  

Combining Theorem \ref{HES} with the criterion for exact sequences of affine group schemes \cite[Theorem A.1]{EHS08}, we deduce the following lemma (cf. \cite[Theorem~9.1]{DHdS18}).
\begin{lem}\label{lem-HES}
Let $\mathcal{V}$ be an object of $\mathrm{MIC}^\circ(X/k)$.
\begin{enumerate}
\item The maximal trivial sub-object of $\mathcal{V}|_{X_s}$ is the restriction of a sub-object $\mathcal{T} \longrightarrow \mathcal{V}.$ Moreover, $\mathcal{T}$ is the pull-back to $\mathrm{MIC}^\circ(X/k)$ of an object of $\mathrm{MIC}^\circ(S/k)$;
\item If $\mathcal{N}$ belongs to $\langle \mathcal{V}|_{X_s}\rangle_\otimes$, then there exists $\mathcal{\tilde{N}}$ in $\langle \mathcal{V}\rangle_\otimes$ and a monic $\mathcal{N}\longrightarrow  \mathcal{\tilde{N}}|_{X_s}.$
\end{enumerate}
\end{lem}
\begin{proof}
According to Theorem \ref{HES}, the following sequence
$$\pi(X_s/k)\longrightarrow \pi(X/k)\longrightarrow \pi(S/k)\longrightarrow 1.$$
is exact. Using the characterization of the exactness presented in \cite[Theorem A.1]{EHS08}, we immediately arrive at the desired conclusion. Moreover, the proof in loc. cit. shows that $\tilde{\mathcal{N}}$ can be chosen in $\langle\mathcal{V}\rangle_{\otimes}.$
\end{proof}

The following lemma is analogous to
\cite[Theorem~9.2]{DHdS18}, which generalizes a result of Deligne (\cite[Theorem~5.10]{EH06}).  
\begin{lem}\label{lem-Deligne} 
Let $\mathcal{V}$ be an object of $\mathrm{MIC}^\circ(X/k)$ and let $\mathcal{M}\longrightarrow \inf (\mathcal{V})$ be a special sub-object in $\mathrm{MIC}^\mathrm{se}(X/S)$. Then there exists 
$\mathcal{N}$  in $\mathrm{Obj}(\mathrm{MIC}^\circ(X/k))$ and an epimorphism 
$\text{\rm inf}(\mathcal{N})\longrightarrow \mathcal{M}.$ Moreover, 
$\mathcal{N}$ can be chosen in $\langle \mathcal{V}\rangle_\otimes$.\end{lem}

Combining Lemma \ref{lem-Deligne} and Lemma \ref{lem-HES}, we deduce the following property of the 
category $\mathrm{MIC}^\mathrm{geom}(X/S)$. The proof of the following lemma requires techniques of flat affine group scheme and will be given in  Appendix B, see Theorem \ref{prop-ismorphicofgalois}.   
\begin{lem}[{compare with \cite[Theorem~8.2]{DHdS18}}]
\label{lem-DHdS}
Let $(\mathcal{V},\nabla)$ be an absolute connection on $X/k$. Then each  relative connection that is locally free as an $\mathcal{O}_X$-module
in  $\langle \inf (\mathcal{V})\rangle_\otimes$ is indeed a \emph{special} sub-object of a tensor generated object from $\inf (\mathcal{V})$.
\end{lem}

\begin{cor}\label{cor-groupH} 
The map
$\pi^\mathrm{geom}(X/S)\longrightarrow \Pi(X/k)^\Delta$ is a closed immersion.  \end{cor}
\begin{proof}
 Since every object of $\mathrm{MIC}^\mathrm{geom}(X/S)^{\circ}$ is isomorphic to a special sub-quotient of an object of the form $(\mathcal{V},\nabla_{/S})$ 
via Lemma \ref{lem-DHdS}, Theorem \ref{th1}
implies that $\pi^\mathrm{geom}(X/S)\longrightarrow \Pi(X/k)^\Delta$ 
is a closed immersion.
\end{proof}

\begin{cor}\label{cor-Lrepr}
Any $R$-projective finite representation of $\pi^\mathrm{geom}(X/S)$ is also a special sub-object of a finite representation of $\Pi(X/k)$ considered as a representation of $\pi^\mathrm{geom}(X/S).$ Consequently, any finite representation of $\pi^\mathrm{geom}(X/S)$ is a quotient of a finite representation of $\Pi(X/k)$ considered as a $\pi^\mathrm{geom}(X/S)$-representation.  
\end{cor}
\begin{proof}
According to Lemma \ref{lem-DHdS}, every object of $\mathrm{MIC}^{\mathrm{geom}}(X/S)^{\circ}$ is a special sub-object  of the form $(\mathcal{V},\nabla_{/S}).$  Hence, by Lemma \ref{lem-Deligne}, any object of $\mathrm{MIC}^{\mathrm{geom}}(X/S)^{\circ}$ is a quotient of some inflated connection.
\end{proof}

Let $L$ be a kernel of the map $f: \Pi(X/k) \longrightarrow \Pi(S/k).$ Notice that $L$ is also equal to the kernel of $\Pi(X/k)^\Delta\longrightarrow \Pi(S/k)^\Delta$ (cf. Appendix \ref{sect_morphism_group_schemes}). 
\begin{lem}\label{lem-Lrepr}
Every finite  representation of $L$ can be represented as a quotient of the restriction to $L$ of a finite representation of $\Pi(X/k).$ Consequently, an $R$-projective finite representation
of $L$ can be embedded into the restriction to $L$ of a finite representation of $\Pi(X/k).$ 

\end{lem}
\begin{proof}
Since $L$ is normal in $\Pi(X/k)^{\Delta}$, according to \cite[Theorem 4.2.2 (c)]{DH18},
any finite  representation of $L$ can be represented as a quotient of the restriction to $L$ of a finite representation of $\Pi(X/k)^{\Delta}$. Now, Corollary \ref{lem-RepGDiagonal} yields
the conclusion. \end{proof}

%
%
%

\subsubsection{Proof of \ref{FES}}\label{proof-FES}
Our strategy is to show that the group $\mathrm{MIC}^\mathrm{geom}(X/S)$ equals the kernel $L$ of the map $f:\Pi(X/k)\longrightarrow \Pi(S/k).$  

By construction, we have closed immersion $\pi^\mathrm{geom}(X/S)\subset L$, 
which induces a functor
$\mathcal F:\mathrm{Rep}_R(L)\longrightarrow \mathrm{Rep}^\mathrm{f}(\pi^\mathrm{geom}(X/S)).$ 
This functor is faithful by construction. We will show that it is full and essentially surjective. 

\noindent 
{\em Step 1}. We show that 
$\mathcal{F}^{\circ}: \mathrm{Rep}_R^{\circ}(L) \longrightarrow \mathrm{MIC}^\mathrm{geom}(X/S)^{\circ}$ is full. 

Let $U_0, U_1$ be objects in $\mathrm{Rep}^{\circ}_R(L)$ and 
$\phi: \mathcal{F}^{\circ}(U_0) \rightarrow \mathcal{F}^{\circ}(U_1)$ a $\mathrm{MIC}^\mathrm{geom}(X/S)$-morphism, meaning $\phi$ is a $R$-linear map $U_0 \rightarrow U_1$ that is $\pi^\mathrm{geom}({X/S})$-linear. We  show that $\phi$ is indeed $L$-linear.  By Lemma \ref{lem-Lrepr} there are $L$-linear morphisms $\pi: V_0 	\twoheadrightarrow U_0$ and $i: U_1 \xhookrightarrow{} V_1,$ where $V_0$ and $V_1$ are objects in $\mathrm{Rep}^\mathrm{f}(S:\Pi(X/k)).$ We set $\psi =i\phi \pi,$ and then we  have the following diagram
$$\xymatrix{
U_0\ar[r]^{\phi}  & U_1 \ar@{^(->}[d]^i\\
V_0  \ar[r]_{\psi}   \ar@{->>}[u]^{\pi}& V_1.
}$$
Thus, $\phi$ is $L$-linear if and only if $\psi$ is. This leads us  to showing  that
$$\Hom_L(V_0,V_1)=\Hom_{\pi^\mathrm{geom}(X/S)}(V_0,V_1),$$
for any $V_0,V_1\in \Rep_f(S:\Pi(X/k)),$  which amounts to showing
$V^L=V^{\pi^\mathrm{geom}(X/S)}$ for $V=V_0^\vee\otimes V_1,$  where $V^L$ is the sub-module of $V$ consisting of
elements which are stable under the action of $L$. 

Now, through Tannakian duality, Lemma \ref{lem-condA} tells us that (see Remark \ref{rmk-condA}), for $V\in \mathrm{Obj}(\mathrm{Rep}^\mathrm{f}(S:\Pi(X/k)))$ that corresponds to the connection $(\mathcal V,\nabla)\in \mathrm{Obj}(\mathrm{MIC}^\circ(X/k)),$
$$V^{\pi^\mathrm{geom}(X/S)}=V^{\pi(X/S)}=\mathrm H^0_\mathrm{dR}(X/S,(\mathcal{V},\nabla_{/S})) \subset V^L.$$
On the other hand, we obviously have $V^L \subset V^{\pi^\mathrm{geom}(X/S)}. $ Thus, we have an equality.
 
\noindent 
{\em Step 2}. We show that $\mathcal F$ is essentially surjective. 
Let $\mathcal V$ be an object  in $\mathrm{MIC}^\mathrm{geom}(X/S)^{\circ}$. Then, according to Lemma 
\ref{lem-DHdS}, there exists an object $(\mathcal{W},\nabla)$ in $\mathrm{MIC}^\circ(X/k)$ such that $\mathcal V$ is 
a special sub-object of $\inf  (\mathcal{W})$. According to 
Lemma \ref{lem-Deligne}, $\mathcal V$ 
can be realized as the image of a morphism $\varphi$ between inflated objects. By the 
above discussion, $\varphi$ is also in the image of $\mathcal{F},$ hence so is 
$\mathrm{Im}(\varphi).$ It means that $\mathcal{V}$ corresponds to a representation of 
$L$ and we finish the proof.
\hfill $\Box$

\subsection{Comparison of cohomology and the Lyndon--Hochschild--Serre spectral sequence} 
Our next aim is to compare various cohomologies. The comparison homomorphisms are by definition
functorial homomorphism between cohomology functors. We first  mention a lemma to fix
our construction.

The settings are as follows. Let $\mathcal C$ and $\mathcal D$ be abelian categories which have
enough injective. Let $\iota:\mathcal C\longrightarrow \mathcal D$ be an exact functor. Let $F,G$
be left exact functors from respectively $\mathcal C$ and $\mathcal C$
to some abelian category, which are compatible with $\iota$ in the sense: there exists a natural isomorphism
$\chi: F\circ \iota\longrightarrow G.$ Let $F^i$ (resp. $G^i$) denote the right derived functors of $F$ 
(resp. $G$).  Then we have an important lemma.
\begin{lem} \label{lem_comparison}

We have the following
 \begin{enumerate}
\item With the assumption as above, there exist a uniquely natural comparison map $\chi^i:F^i
\longrightarrow  G^i\circ \iota$, for each $i\geq 0$, $\chi_0=\chi$, which are
compatible with the connecting maps.
\item Let $j:\mathcal D\longrightarrow E$ be another exact functor and $H$ be a left exact functor on
$\mathcal E$ with the same target as that of $F$ and $G$, such that $H\circ j$ is isomorphic to $G$
by a natural map $\psi.$
Let $\psi^{i}:G^i\longrightarrow H^i\circ j$ be the comparison map. 
Let $\psi(\iota): G\circ \iota \longrightarrow H\circ j\circ \iota$ and $\omega:=\psi(\iota)\circ\chi: F\longrightarrow H\circ j\circ \iota.$ 
Then we have the following commutative diagram of comparison maps:
$$\xymatrix{F^i\ar[r]^{\chi^i}\ar[rd]_{\omega^i} & G^i\circ \iota \ar[d]^{\psi^i(\iota)}\\
& H^i\circ j\circ \iota.}$$
\end{enumerate}
\end{lem} 
\begin{proof}
By definition, $(F^i)_{i\neq 0})$ are right derived functors of $F$ while $(G^i\circ \iota)_{i\neq 0})$
form a $\delta$-functors of the same functor. According to \cite[Cor.~1.4, Chap.~III]{Ha77}, there exist
a unique natural map $\chi ^i:F^i\longrightarrow G^i\circ \iota$ as required. The second claim
follows from the uniqueness of the maps involved.
\end{proof}

According to Theorem \ref{FES}, we have
$$1\longrightarrow \pi^{\mathrm{geom}}(X/S) \longrightarrow \Pi(X/k)\stackrel f\longrightarrow \Pi(S/k)\longrightarrow 1.$$

Then hence we have the following diagram of functors:
\begin{align}\label{dia-2}
    \xymatrix{\mathrm{Rep}(S:\Pi(X/k)) 
    \ar[d]_{\mathrm{H}^0(\pi^{\mathrm{geom}}(X/S),-)}
    \ar[rr]^{\mathrm{Res}^{\Pi(X/k)}_{\pi^{\mathrm{geom}}(X/S)}} & & 
      \mathrm{Rep}(\pi^{\mathrm{geom}}(X/S)) \ar[d]^{\mathrm{H}^0(\pi^{\mathrm{geom}}(X/S),-)}  \\
\mathrm{Rep}(S: \Pi(S/k))  \ar[rr]^{\mathrm{For}} &&   \mathrm{Mod}_R
%
   }
\end{align}
Apply Lemma \ref{lem_comparison} we get comparison maps
$$\alpha^i:   \mathbf{R}_{\mathrm{Rep}(S:\Pi(X / k))}^{n} \mathrm{H}^{0}(\pi^{\mathrm{geom}}(X/S), V) \longrightarrow   \mathrm{H}^i(\pi^{\mathrm{geom}}(X/S), V)$$
(the former derived functor is taken in $\mathrm{Rep}(S:\Pi(X / k))$ and the latter one is taken in
$ \mathrm{Rep}(\pi^{\mathrm{geom}}(X/S))$. 

\begin{thm}\label{thm-gmt}
Let $X$ be a projective smooth scheme over a Dedekind domain $R$ with geometrically 
connected fibers. Let $V$ be an object in $\mathrm{Rep}^\mathrm{f}(S:\Pi(X/k)).$ For all $i\geq 0,$   the canonical homomorphism
$$\alpha^i:
  \mathbf{R}_{\mathrm{Rep}(S:\Pi(X / k))}^{i} \mathrm{H}^{0}(\pi^{\mathrm{geom}}(X/S), V)\longrightarrow \mathrm{H}^{i}(\pi^{\mathrm{geom}}(X/S), V)
$$
is an isomorphism.
\end{thm}
\begin{proof}
As an initial step, we describe the injective objects in the category $\mathrm{Rep}(R: \Pi(X/k)).$ Indeed, each injective object $F_1$ in $\mathrm{Rep}(R: \Pi(X/k))$ is a direct summand of $I \otimes_t \mathcal{O}(\Pi(X/k)),$ where $I$ is regarded as a trivial $\Pi(X/k)$-module (see Lemma \ref{lem-enoughinjectives}). Next, we prove that the restriction functor maps  an injective $\Pi(X/k)$-module to one that is acyclic for the  functor $(-)^{\pi^{\mathrm{geom}(X/S)}}$. Indeed, thanks to Proposition \ref{lem-11}, we obtain
\begin{align*}
\mathrm{R}^i \mathrm{ Ind }_{\pi^{\mathrm{geom}(X/S)}}^{\Pi(X/k)} (I) \simeq \mathrm{H}^i(\pi^{\mathrm{geom}(X/S)}, I \otimes_t \mathcal{O}(G)) = \mathrm{H}^i(\pi^{\mathrm{geom}(X/S)}, F_1) \oplus \mathrm{H}^i(\pi^{\mathrm{geom}(X/S)}, F_2),
\end{align*}
where $F_2$ is a $\Pi(X/k)$-module such that $F_1 \oplus F_2 = I \otimes_t \mathcal{O}(\Pi(X/k))$. Since the induction functor $\mathrm{ Ind }_{\pi^{\mathrm{geom}(X/S)}}^{\Pi(X/k)}$ is exact (see Lemma \ref{lem-999}), the $\Pi(X/k)$-module $F_1$ is acyclic for the functor $(-)^{\pi^{\mathrm{geom}}(X/S)}$. Therefore, Remark \ref{rem-2}-(1),(2) implies that the
$$ V \mapsto \mathrm{H}^i(\pi^{\mathrm{geom}(X/S)},V)   $$
can be regarded as the derived functors of
$$  V \mapsto V^{\pi^{\mathrm{geom}}(X/S)} $$
from  $\mathrm{Rep}(S:\Pi(X/k))$ to $\mathrm{Rep}(R: \Pi(S/k)).$
\end{proof}

\begin{lem}\label{lem-999}
    The functor $\mathrm{Ind}_{\pi^{\mathrm{geom}}(X/S)}^{\Pi(X/k)}$ is exact.
\end{lem}
\begin{proof}
    According to Theorem \ref{FES}, the geometric relative differential fundamental group $\pi^\mathrm{geom}(X/S)$ is the normal subgroup of $\Pi(X/k)^{\Delta}$. Then, the functor $\mathrm{Ind}_{\pi^\mathrm{geom}(X/S)}^{\Pi(X/k)^{\Delta}}$ is exact by the same argument as in the proof of \cite[Theorem 4.2.2]{DH18}.  Using  Theorem \ref{thm-exact} combined with the following equality (Corollary \ref{cor-23}):
 \begin{align*}
\mathrm{Ind}^{\Pi(X/k)}_{\Pi(X/k)^{\Delta}} \circ \mathrm{Ind}_{\pi^\mathrm{geom}(X/S)}^{\Pi(X/k)^{\Delta}} \cong \mathrm{Ind}_{\pi^\mathrm{geom}(X/S)}^{\Pi(X/k)},
 \end{align*}
 we conclude that $\mathrm{Ind}_{\pi^\mathrm{geom}(X/S)}^{\Pi(X/k)}$ is exact.
\end{proof}

Let $V$ be an object in $\mathrm{Rep}^\mathrm{f}(S:\Pi(X/k)).$ Thanks to Lemma \ref{lem-13}, the invariant $V^{\pi^\mathrm{geom}(X/S)}$ is $\Pi(X/k)$-module. This yields a functor
\begin{align*}
    (-)^{\pi^\mathrm{geom}}: \mathrm{Rep}^\mathrm{f}(S:\Pi(X/k)) &\longrightarrow \mathrm{Rep}^\mathrm{f}(S:\Pi(S/k))\\
    V & \mapsto V^{\pi^\mathrm{geom}(X/S)}.
\end{align*}
\begin{lem}\label{lem-25}
    The functor $(-)^{\pi^\mathrm{geom}(X/S)}$ sends injective objects to injective objects.
\end{lem}
\begin{proof}
    We define the  functor $ \alpha$ from $ \mathrm{Rep}^\mathrm{f}(S:\Pi(S/k))$  to $\mathrm{Rep}^\mathrm{f}(S:\Pi(X/k)):$ Let $A$ be a $k$-algebra.  If $V$ is a $\Pi(S/k)$-module, then set $\alpha(M)$ equal to $M$ as a $R$-module, and let any $\Pi(X/k)$ act on $A\otimes_sV$ via the homomorphism $f(A): \Pi(X/k)(A) \longrightarrow \Pi(S/k)(A)$ and the given action of $\Pi(S/k).$ 

    We have for any $\Pi(S/k)$-module $U$ and any $\Pi(X/k)$-module $V:$
    $$ \mathrm{Hom}_{\Pi(X/k)}(\alpha(U),V) \cong \mathrm{Hom}_{\Pi(X/K)}(\alpha
    (U), V^{\pi^\mathrm{geom}(X/S)})\cong \mathrm{Hom}_{\Pi(S/k)}(U,V^{\pi^\mathrm{geom}(X/S)}),$$
 where the first isomorphism is induced by the inclusion $V^{\pi^\mathrm{geom}(X/S)} \subset V$ and the second isomorphism follows from the fact that $\alpha$ is fully faithful functor. This shows that $(-)^{\pi^\mathrm{geom}(X/S)}$ is right adjoint to the exact functor $\alpha.$ Thus injective
objects are mapped to injective objects.
\end{proof}

\begin{prop}[Lyndon--Hochschild--Serre spectral sequence]\label{prop-21}
    Let $V$ be an object in $\mathrm{Rep}^\mathrm{f}(S:\Pi(X/k)).$ Then there exist a spectral sequence
    $$ E_2^{p,q}= \mathrm{H}^p(\Pi(S/k), \mathrm{H}^q(\pi^\mathrm{geom}(X/S), V)) \Rightarrow \mathrm{H}^{p+q}(\Pi(X/k),V).$$
    \end{prop}
 \begin{proof}
     We have the following diagram:
     $$ \mathrm{Rep}(R: \Pi(X/k)) \overset{(-)^{\pi^\mathrm{geom}(X/S)}}{\longrightarrow}\mathrm{Rep}(R: \Pi(S/k)) \overset{(-)^{\Pi(S/k)}}{\longrightarrow}\mathrm{Vec}_k.$$
Since functor $(-)^{\pi^\mathrm{geom}(X/S)}$ sends injective objects to injective objects (cf. Lemma \ref{lem-25}), we have 
$$ E_2^{p,q} = \mathrm{H}^p(\Pi(S/k), \mathbf{R}_{\mathrm{Rep}(S:\Pi(X / k))}^{q} \mathrm{H}^{0}(\pi^{\mathrm{geom}}(X/S), V)) \Rightarrow \mathrm{H}^{p+q}(\Pi(X/k),V).$$
Thanks to Theorem \ref{thm-gmt}, we obtain
$$ E_2^{p,q}= \mathrm{H}^p(\Pi(S/k), \mathrm{H}^q(\pi^\mathrm{geom}(X/S), V)) \Rightarrow \mathrm{H}^{p+q}(\Pi(X/k),V).$$
 \end{proof}

\section{Comparison between de Rham cohomology and group cohomology of  \texorpdfstring{$\pi(X/S)$}{pi(X/S)}}  \label{sect-compa}
In this section we aim to compare, through Tannakian duality, the de Rham cohomology on 
$X/S$ and the group cohomology of $\pi^\mathrm{geom}
(X/S)$.  More precisely, given Theorem \ref{thm-gmt}, our next question is to compare  
$
\mathrm{H}^{i}(\pi^{\mathrm{geom}}(X/S),V)$ with
$\mathrm{H}_\mathrm{dR}^{i}(X/S, \mathcal V), $
where $\mathcal V$ is an absolute connection and $V=\eta^*(\mathcal V)$ is the corresponding
representation. Using Lemma \ref{lem_comparison}, there are several comparison maps, which all
agree.

Our strategy is to investigate the bijectivity the maps on the generic and the closed 
fibers of $\mathrm{Spec}(R)$. 

\subsection{The comparison maps between de Rham cohomology and group cohomology}

We recall from Subsection 
\ref{sect-RFG} and Lemma \ref{lem-groupH}  
that the fiber functor $\eta^{\ast}$ induces 
$$\Rep^\mathrm{f}(\pi(X/S))\cong \mathrm{MIC}^\mathrm{se}(X/S),$$
and
$$\Rep^\mathrm{f}(\pi^\mathrm{geom}(X/S))\cong \mathrm{MIC}^\mathrm{geom}(X/S).  $$
By taking the ind-categories of these categories, we have 
\begin{align}\label{eq-5}
   \Rep(\pi(X/S))\cong \mathrm{MIC}^\mathrm{ind}(X/S),  
\end{align}
and
\begin{align*}
       \Rep(\pi^\mathrm{geom}(X/S))\cong \mathrm{ind-}\mathrm{MIC}^\mathrm{geom}(X/S).  
\end{align*}


Starting with the natural isomorphism
$$\mathrm{H}^0(\pi(X/S),M) \longrightarrow \mathrm{H}^0_{\mathrm{dR}}(X/S,\mathcal{M})$$
applying Lemma \ref{lem_comparison} (where $\alpha$ is the inclusion functor
$\Rep(\pi(X/S))\longrightarrow \mathrm{MIC}(X/S)$) we obtain the comparison maps
\begin{align}\label{eq-4}
  \delta^i: \mathrm{H}^i(\pi(X/S),M) \longrightarrow \mathrm{H}^i_{\mathrm{dR}}(X/S,\mathcal{M}).  
\end{align}
By a similar argument, we also obtain the following map:
$$\gamma^i: \mathrm{H}^i(\pi^\mathrm{geom}(X/S),V) \longrightarrow \mathrm{H}^i(\pi(X/S),V),$$
for any $V$ from $\mathrm{MIC}^\mathrm{geom}(X/S).$ 
Hence, we obtain the following map:
$$ \nu^i:\mathrm{H}^i(\pi^{\mathrm{geom}}(X/S),V)\overset{\gamma^i} \longrightarrow \mathrm{H}^i(\pi(X/S),V) \overset{\delta^i}{\longrightarrow} \mathrm{H}_\mathrm{dR}^i(X/S, \mathcal V). $$
This map is a priory not compatible with the action of $\Pi(S/k)$, since the middle term is not
equipped with such an action. To complete the picture, we need to relate it with the map
$\alpha$ of \ref{thm-gmt}.

 Let  $\mathcal{V} $ be an object in   $\mathrm{MIC}^\circ(X/k),$ together with its fiber  
    $V= \eta^{\ast}(\mathcal{V})$ which is an $R$-module.  
According to Lemma \ref{lem-13}, the finite $R$-module $\mathrm{H}^{0}(\pi^{\mathrm{geom}}(X/S), V)$ 
is a $\Pi(R / k)$-representation in a natural way. 
According to Lemma \ref{lem-condA}, the canonical morphism of objects in $\mathrm{MIC}(S/k)$
$$
\nu^0:\mathrm{H}^{0}(\pi^{\mathrm{geom}}(X/S), V) \longrightarrow \mathrm{H}_\mathrm{dR}^{0}(X/S, \mathcal V)
$$
is an isomorphism. 
In other words, the following diagram of functors commutes:
\begin{align}\label{dia-1}
        \xymatrix{ \mathrm{Rep}(S:\Pi(X/k))
         \ar[d]_{\mathrm{H}^0(\pi^{\mathrm{geom}}(X/S),-)} \ar@{^(->}[rr]  & & 
    \mathrm{MIC}(X/k) \ar[d]^{\mathrm{H}_\mathrm{dR}^0(X/S, \mathrm{inf}(-))} \\
    \mathrm{Rep}(R: \Pi(S/k)) \ar@{^(->}[rr] &   & \mathrm{MIC}(S/k).} 
\end{align}
Applying Lemma \ref{lem_comparison}, we obtain a homomorphism
$$ 
\beta^i:     \mathbf{R}_{\mathrm{Rep}(S:\Pi(X / k))}^i \mathrm{H}^{0}(\pi^{\mathrm{geom}}(X/S), V) \longrightarrow \mathbf{R}_{\mathrm{MIC}(X/k)}^i \mathrm{H}_\mathrm{dR}^0(X/S,\mathcal V)
$$
where, on the left-hand side, the derived functor is taken in $\operatorname{Rep}(R: \Pi(X / k))$ and on the right-hand side, the derived functor is taken in $\operatorname{MIC}(X / k)$.  We know that the right-hand side is the n-th relative de Rham cohomology of the inflated connection (cf. \cite[Remark~3.1]{Ka70},
\cite[23.2.5]{ABC20}):
$$
\mathbf{R}_{\mathrm{MIC}(X/k)}^i \mathrm{H}_{\mathrm{dR}}^{0}(X/S, \mathcal V)=
\mathrm{H}_{\mathrm{dR}}^i(X/S, \mathcal V).
$$

\begin{lem}\label{lem_comparison_functors} For the absolute connection $\mathcal V$ in $\mathrm{MIC}^\mathrm{o}(X/k)$ and
its fiber $V=\eta^*\mathcal V$ we have the following commutative diagram:
$$\xymatrix{
\mathrm{H}^i(\pi^{\mathrm{geom}}(X/S),V) \ar[rr]^{\gamma^i}
&& \mathrm{H}^i(\pi(X/S),V)\ar[d]^{\delta^i}\\
\mathbf{R}_{\mathrm{Rep}(S:\Pi(X / k))}^i \mathrm{H}^0(\pi^{\mathrm{geom}}(X/S), V) 
\ar[rr]_{\beta^i} \ar[u]_{\alpha^i}
&&   \mathbf{R}_{\mathrm{MIC}(X/k)}^i \mathrm{H}_\mathrm{dR}^0(X/S,\mathcal V)
}$$
\end{lem}
\begin{proof}
Consider the following commutative diagram
$$\xymatrix{&&
\mathrm{Rep}(R:\Pi(X/k))
\ar[lld]_{\mathrm{Res}}\ar[d]_{\mathrm{Res}}\ar[rrd]^{\mathrm{inf}}
\\
\mathrm{Rep}\pi^{\mathrm{geom}}(X/S)
\ar@{^(->}[rr] \ar[rrd]_{\mathrm{H}^0(\pi^\mathrm{geom}(X/S),-)\qquad } &&
\mathrm{Rep}\pi(X/S)
 \ar@{^(->}[rr] \ar[d]_{\mathrm{H}^0(\pi(X/S),-)} &&
\mathrm{MIC}(X/S) \ar[lld]^{\mathrm{H}_{\mathrm{dR}}^0(X/S,-)} 
\\
&&\mathrm{Mod}_R.
}$$
The composite functor is $V\mapsto \mathrm{H}^0(\pi^\mathrm{geom}(X/S),V)$.
Thus, Lemma \ref{lem_comparison} (2) applies and yields the required diagram. 
\end{proof}
It follows from the Tannakian duality mentioned above and the description of de Rham cohomology as 
ext-group (cf. subsection \ref{sect-deRham-cohomology}) that for any 
$\mathcal V\in \mathrm{Obj}(\mathrm{ind-}\mathrm{MIC}^\mathrm{geom}(X/S))$ and $V=\eta^*(\mathcal V)$ we 
have an isomorphism
\begin{align}\label{eq-113}
\mathrm H^0_\mathrm{dR}(X/S,\mathcal V)\cong  \mathrm H^0(\pi^\mathrm{geom}(X/S), V).
\end{align}
Let $\mathcal{M}$ be an object in $\mathrm{MIC}^\mathrm{se}(X/S).$ Then any extension of $\mathcal{M}$ by $\mathcal{O}_X$ is again an object in $\mathrm{MIC}^\mathrm{se}(X/S),$ this implies that the natural map
\begin{align}\label{eq-22}
   \delta^1: \mathrm{H}^1(\pi(X/S),M)\longrightarrow \mathrm{H}^1_{\mathrm{dR}}(X/S, \mathcal{M}) 
\end{align}
is bijective. The proof for the bijectivity of $\nu^1$ is not straightforward; we will need to use the Universal Extension Theorem in the following section.

\subsection{The universal extension theorem}
By the nature of $\pi^\mathrm{geom}(X/S),$ we have the following result, which is an adaptation of \cite[Thm.~4.2]{EH06}.
\begin{thm}[\textit{Universal extension}]\label{univ-ext}
Let $(\mathcal{V}, \nabla)$ be an object in $ \mathrm{MIC}^\circ(X/k)$, 
then there exists an extension in $\mathrm{MIC}^\circ(X/k)$:
\begin{center}
$0 \longrightarrow  \mathcal{V} \longrightarrow\mathcal  W\longrightarrow
f^*(\mathrm R^1_{\mathrm{dR}}f_\ast \mathcal{V}) \longrightarrow 0$
\end{center}
with the property that the connecting morphism in the long exact sequence of de Rham cohomology on $X/S$:
\begin{center}
$ \mathrm{H}^0_\mathrm{dR}(X/S, f^*(\mathrm R^1_{\mathrm{dR}}f_* \mathcal{V}))=
 \mathrm{H}^1_\mathrm{dR}(X/S,\mathcal V)\stackrel{\text{\rm connecting }} 
 \longrightarrow
  \mathrm{H}^1_\mathrm{dR}(X/S,\mathcal V)$
\end{center}
is the identity map.
\end{thm}
\begin{proof}
Let $\mathcal{Z} =f^*(\mathrm R^1f_* (\mathcal{V},\nabla))$ (see Subsection \ref{sect-GM}). Then
$$\inf(\mathcal{Z})=\mathcal O_X\otimes\mathrm 
H^1_\mathrm{dR}(X/S,\inf(\mathcal V)).$$
 Let 
\[ \mathcal{W}= \mathcal V\otimes \mathcal Z^\vee\]
as objects in $\mathrm{MIC}^\circ(X/k)$. 
Then we have isomorphisms in $\mathrm{MIC}^\mathrm{se}(X/S)$ of the cohomologies of Gauss-Manin connections: 
\begin{eqnarray*}
\mathrm{H}^1_\mathrm{dR}(X/S,\inf(\mathcal{W}))&
\cong& \mathrm{Ext}^1_{\mathrm{MIC}(X/S)}(\mathcal O_X,\inf(\mathcal V)\otimes 
\inf (\mathcal{Z}^\vee))\\
&\cong& \mathrm{Ext}^1_{\mathrm{MIC}(X/S)}(\inf (\mathcal{Z}),\inf(\mathcal V))\\
&\cong& \mathrm H^1_\mathrm{dR}(X/S,\inf(\mathcal V))^\vee\otimes  
\mathrm{Ext}^1_{\mathrm{MIC}(X/S)}(\mathcal O_X,\inf(\mathcal V))\\
&\cong&\mathrm H^1_\mathrm{dR}(X/S,\inf(\mathcal V))^\vee\otimes \mathrm H^1_\mathrm{dR}(X/S,\inf(\mathcal V))\\
&\cong& \mathrm{End}_R(\mathrm H^1_\mathrm{dR}(X/S,\inf(\mathcal V)).
\end{eqnarray*} 
 Let $\varepsilon$ be the element in 
$\mathrm{Ext}^1_{\mathrm{MIC}(X/S)}(\inf (\mathcal{Z}), \inf(\mathcal V))$, the image of which the identity map in 
$\mathrm{End}_R(\mathrm H^1_\mathrm{dR}(X/S,\inf(\mathcal V)))$. As the identity map is killed by the Gauss-Manin connections, so is $\varepsilon$. 

Now consider the exact sequence of complexes:
\begin{center}
$0 \longrightarrow f^*\Omega^1_{S/k}\otimes_R (\Omega_{X/S}^{\bullet-1}\otimes \mathcal{W}) \longrightarrow \Omega_{X/k}^{\bullet}\otimes \mathcal{W} \longrightarrow \Omega_{X/S}^{\bullet}\otimes \mathcal{W} \longrightarrow 0$.
\end{center}
In our situation, the base scheme is an affine scheme, so we get the long exact sequence \eqref{GM-hyper}:
\begin{center}
$\cdots \longrightarrow \mathrm{H}^1_\mathrm{dR}(X/k,\mathcal{W} ) \longrightarrow \mathrm{H}^1_\mathrm{dR}(X/S,\inf(\mathcal{W})) \longrightarrow \Omega^1_{S/k}\otimes_R \mathrm{H}^1_\mathrm{dR}(X/S,\inf (\mathcal{W})) \longrightarrow \cdots$
\end{center}
Now, the homomorphism:
\begin{center}
$\partial^1: \mathrm{H}^1_\mathrm{dR}(X/S,\inf (\mathcal{W})) \longrightarrow \Omega^1_{S/k}\otimes_R \mathrm{H}^1_\mathrm{dR}(X/S,\inf (\mathcal{W}))$
\end{center}
is the Gauss-Manin connection. Hence, $\varepsilon \in \mathrm{ Ker}\delta_1$ and thus it 
is lifted to $\Tilde{\varepsilon} \in \mathrm{H}^1_\mathrm{dR}(X/k,(\mathcal{W},\nabla_{\mathcal{W}}))$ by the long exact sequence. 
Consequently, there exists an extension of  connections in $\mathrm{MIC}^\circ(X/k)$:
\[
\Tilde{\varepsilon}: \quad 0 \longrightarrow (\mathcal{W},\nabla_{\mathcal{W}}) \longrightarrow (\mathcal{T}',\nabla_{\mathcal{T}'})\longrightarrow (\mathcal O_X,d) \longrightarrow 0.\]
Notice that the inflation of this sequence to $\mathrm{MIC}^\mathrm{se}(X/S)$ is $\varepsilon$. Hence, the induced connecting map is the identity map by construction of $\varepsilon$.  
\end{proof}
With Theorem $\ref{univ-ext}$, we have a statement on the comparison of  cohomologies.
\begin{cor}\label{cor-comparison}
Let $\mathcal{V}$ be an object in $\mathrm{MIC}^{\mathrm{geom}}(X/S)$ which is locally free as an $\mathcal{O}_X$-module. Let $V=\eta^*(\mathcal V)$ be the corresponding representation of $\pi^\mathrm{geom}(X/S)$. Then 
    $$  \nu^1:\mathrm{H}^{1}(\pi^{\mathrm{geom}}(X/S),V)
    \overset{\gamma^1} \longrightarrow \mathrm{H}^1(\pi(X/S),V) \overset{\delta^1}{\longrightarrow} \mathrm{H}_\mathrm{dR}^{1}(X/S, \mathcal V)$$
    is bijective.  
\end{cor}
\begin{proof} 
Thanks to \eqref{eq-22}, we only need to prove that the  map 
$$\gamma^1: \mathrm{H}^1(\pi^\mathrm{geom}(X/S),V)
 \longrightarrow \mathrm{H}^1(\pi(X/S),V) $$
is bijective.

\noindent\textit{{Case 1}}. Assume that  $V$ corresponds to an inflated connection $\mathrm{inf}(\mathcal{V}).$ The category $\mathrm{MIC}^\mathrm{geom}(X/S)$ is a full subcategory of $\mathrm{MIC}^\mathrm{se}(X/S),$ so the homomorphism 
\[\gamma^1:  \mathrm{H}^1(\pi^\mathrm{geom}(X/S),V)
 \longrightarrow \mathrm{H}^1(\pi(X/S),V)\]
is injective.

We prove the surjectivity, which, through Tannakian duality, amounts to saying that each extension 
\[ e:\quad 0 \longrightarrow (\mathcal V,\nabla_{/S})\longrightarrow (\mathcal{V}',\nabla_{\mathcal{V}'}) \longrightarrow \mathcal{O}_X \longrightarrow 0\]
is in $\mathrm{MIC}^\mathrm{geom}(X/S)$, which means $(\mathcal V', \nabla_{\mathcal{V}'})\in \mathrm{Obj}(\mathrm{MIC}^\mathrm{geom}(X/S)).$ Indeed, the element $e\in \mathrm{H}^1_\mathrm{dR}(X/S,(\mathcal V,\nabla_{/S}))$ can be
seen as a map in $\mathrm{MIC}^\mathrm{se}(X/S)$
\[\tilde e:\mathcal O_X\longrightarrow
\mathcal O_X\otimes \mathrm{H}^1_\mathrm{dR}(X/S,(\mathcal{V},\nabla_{/S}).\]

Let $\varepsilon$ be the universal extension of Theorem
\ref{univ-ext}. Then $e$ can be obtained from from $\varepsilon$ by pulling back a long $\tilde e$: 
\[ \begin{tikzcd}
e:\quad 0 \arrow{r} &  (\mathcal{V},\nabla_{/S}) \arrow{r} \arrow[swap]{d}{=} & (\mathcal{V}',\nabla') \arrow{r} \arrow{d} & (\mathcal{O}_X,d) \arrow{r}
 \arrow[swap]{d}{\tilde e} & 0  \\%
\varepsilon:\quad0 \arrow{r} &  (\mathcal{V},\nabla_{/S}) \arrow{r} &
 (\mathcal{W},\nabla_{/S}) \arrow{r} &  (\mathcal{Z},\nabla_{/S}) \arrow{r} & 0.
\end{tikzcd}
\]
The module $\mathrm{H}^1_\mathrm{dR}(X/S,(\mathcal V,\nabla_{/S}))$ on $S$ is equipped
with the Gauss-Manin connection (cf. Lemma \ref{coh1}), hence is projective, hence the map $\tilde e$ is injective.
Thus, $(\mathcal{V}',\nabla)$ is a sub-connection of $(\mathcal{W},\nabla_{/S}),$ 
and hence is an object of $\mathrm{MIC}^\mathrm{geom}(X/S)$. This completes the proof.

\noindent\textit{{Case 2}}. Assume that $\mathcal{V} \in \mathrm{Obj}(\mathrm{MIC}^\mathrm{geom}(X/S))$ and is locally free as an $\mathcal{O}_X$-module. According to Lemma \ref{lem-DHdS}, we can consider $\mathcal{V}$ as special sub-object of some inflated connection $\mathrm{inf}(\mathcal{U}),$ where $\mathcal{U} \in \mathrm{Obj}(\mathrm{MIC}^\circ(X/k)).$ Consider the following short exact sequence in $\mathrm{MIC}^\mathrm{geom}(X/S):$
$$0 \longrightarrow \mathcal{V} \longrightarrow \mathrm{inf}(\mathcal{U})\longrightarrow \mathcal{Q}\longrightarrow 0,$$
where $\mathcal{Q}:= \mathrm{inf}(\mathcal{U})/\mathcal{V}.$ By Tannakian duality, we have the following short exact sequence in $\mathrm{MIC}^\mathrm{geom}(X/S):$
$$ 0 \longrightarrow V \longrightarrow U \longrightarrow Q \longrightarrow 0.$$
Applying long exact sequence, we have:
$$\xymatrix{\mathrm{H}^{0}(\pi^\mathrm{geom}(X/S), Q) \ar[r] \ar[d]^{\cong}&\mathrm{H}^{1}(\pi^\mathrm{geom}(X/S), V) \ar[r] \ar@{^{(}->}[d]^{\nu^1(V)} & \mathrm{H}^{1}(\pi^\mathrm{geom}(X/S), V) \ar[r] \ar[d]^{\cong}  & \mathrm{H}^{1}(\pi^\mathrm{geom}(X/S), Q) \ar@{^{(}->}[d] \ar[r] &0\\
 \mathrm{H}^0_\mathrm{dR}(X/S, \mathcal{Q})   \ar[r]  & \mathrm{H}^1_\mathrm{dR}(X/S, \mathcal{V})\ar[r] & \mathrm{H}^1_\mathrm{dR}(X/S, \mathrm{inf}(\mathcal{U})) \ar[r] &  \mathrm{H}^1_\mathrm{dR}(X/S, \mathcal{Q}) \ar[r]&0. }$$ 
The Four-Lemma implies that $\nu^1(V)$ is bijective. 
\end{proof}

\subsection{The injectivity of \texorpdfstring{$\nu^2$}{nu²}}
Before proving that the map $\nu^2$ is injective, we need the following lemma.
\begin{lem}\label{lem-injective}
Let $\mathcal{B}$ be an abelian category that has enough injectives and $\mathcal{A}$ be an abelian (has enough injective) full subcategory of $\mathcal{B}.$ Let $F$ be a left exact functor: 
$$ F: \mathcal{B} \longrightarrow \mathrm{Ab}.$$
Then we have the following natural transformations: 
$\psi^i: R^i_{\mathcal{A}}F \longrightarrow R^i_{\mathcal{B}}F,$
for each $i\geq 1.$
Suppose that the following condition holds:
\begin{itemize}
    \item [i)]  $X\in \mathrm{Obj}(\mathcal{A})$ such that $\psi^i(X)$ is bijective.
    \item [ii)] there exists $J \in \mathrm{Obj}(\mathcal{A})$ and  injective map $X \hookrightarrow J$ in $\mathcal{A}$ such that 
    $R^i_{\mathcal{B}}F(J)=0$ and
    $R^k_\mathcal{A}F(J)=0$ for $k\in \{i,i+1\} ,$ and $\psi^i(J/X)$ is bijective.
\end{itemize}
Then $\psi^{i+1}(X)$ is injective. 
\end{lem}
\begin{proof}
    Let $X$ be an object of $\mathcal{A}$ that satisfies the conditions. Consider the following short exact sequence:
    $$ 0 \longrightarrow X \longrightarrow J \longrightarrow J/X \longrightarrow 0$$
    in $\mathcal{A},$ this short exact sequence can be considered in $\mathcal{B}.$ Using long exact sequence, we obtain:
    $$\xymatrix{
	0=R^i_{\mathcal{A}}F(J) \ar[r] \ar[d] &R^i_{\mathcal{A}}F(J/X) \ar[r]^{\cong}\ar[d]^{\psi^i(J/X)}_\cong & R^{i+1}_{\mathcal{A}}F(X) \ar[r]\ar[d]^{\psi^{i+1}(X)}& R^{i+1}_{\mathcal{A}}F(J) = 0 \ar[d]^{\psi^{i+1}(J)}
\\
0=R^i_{\mathcal{B}}F(J) \ar[r] &	{R^i_{\mathcal{B}}F(J/X)} \ar@{^{(}->}[r]& R_{\mathcal{B}}^{i+1}F(X) \ar[r]& R^{i+1}_{\mathcal{B}}F(J). }$$
This implies that $\psi^{i+1}(X)$ is injective.
\end{proof}

\begin{lem}\label{133}
Let $\mathcal V$ be an object in $ \mathrm{MIC}^\mathrm{geom}(X/S)$ such that $\mathcal{V}$ is locally free as $\mathcal{O}_X$-module. Denote $V:=\eta^*(\mathcal{V}).$ Then the map
 $$\nu^2:\mathrm{H}^2(\pi^\mathrm{geom}(X/S), V) \rightarrow \mathrm{H}^2_\mathrm{dR}(X/S, \mathcal{V})$$ 
is injective.  
\end{lem}
\begin{proof}
We prove by applying Lemma \ref{lem-injective} to the map
$$\nu^1: \mathrm{H}^1(\pi^\mathrm{geom}(X/S), V)\longrightarrow \mathrm{H}^1_{\mathrm{dR}}(X/S,\mathcal{V}).$$

Thanks to Corollary \ref{cor-comparison}, we have that $\nu^1(\mathcal{V})$ is bijective.
  
Choose $J = V\otimes_R \mathcal{O}(\pi^\mathrm{geom}(X/S)).$ Then $\mathrm{H}^k(\pi^\mathrm{geom}(X/S),J) = 0$ for every $i\geq 1$ (see Lemma \ref{cor-2}). 
Since $\pi^\mathrm{geom}(X/S)$ is flat over $R,$ 
$V\otimes_R \mathcal{O}(\pi^\mathrm{geom}(X/S))$ is a direct limit of locally free of $R$-finite  
projective $\pi^\mathrm{geom}(X/S)$-representations. 
Now, Corollary \ref{cor-comparison} and the fact that group cohomology commutes with direct limits 
imply that $\nu^1(V\otimes \mathcal{O}(\pi^\mathrm{geom}(X/S))$ is bijective.
Hence, $$\mathrm{H}_{\mathrm{dR}}^1(X/S, \mathcal{V}\otimes \mathcal{W})=0,$$ 
where $\mathcal{W}$ is the integrable connection corresponding to regular representation 
$\mathcal{O}(\pi^\mathrm{geom}(X/S))$ of $\pi^\mathrm{geom}(X/S)$. 
    
    The morphism $V \longrightarrow V \otimes_R \mathcal{O}(\pi^\mathrm{geom}(X/S))$ is 
special as it admits a splitting given by the counit map,
that is, the quotient $Q:= V \otimes_R \mathcal{O}(\pi^\mathrm{geom}(X/S))/V$ is also flat,
hence it is also a direct limit of $R$-finite projective
representations. As before, this implies that $\nu^1(Q)$ is bijective.
Therefore, $\nu^2(\mathcal{V})$ is injective.
\end{proof}
For $i\geq 2,$ we establish the comparison by reducing the problem to comparisons in both generic and closed fibers. This approach requires an understanding of the fibers of Tannakian groups and the scalar extension of Tannakian categories.
\subsection{Comparison at the generic fiber}
In this subsection, we show that the maps $\delta^i_K$ are isomorphisms. Recall
that $\delta^i$ is an isomorphism in the case of smooth projective curves over fields of characteristic zero, cf. \cite{BHT25}. However, we cannot apply this result straightforwardly as the two
groups $\pi(X/S)_K$ and $\pi(X_K/K)$ are different. 
Fortunately, their cohomology groups with coefficients in representations of $\pi(X/S)_K$ are
isomorphic, cf. Proposition \ref{130}.

Before proceeding, let us remind the reader of
the just mentioned result:
\begin{lem}\cite[Theorem 4.1]{BHT25} \label{lem-998}
The map $$\delta^i(X_K/K): \mathrm{H}^i(\pi(X_K/K),M_K) \rightarrow \mathrm{H}^i_{\mathrm{dR}}(X_K/K, \mathcal{M}_K) $$
 is bijective for all $i\geq 0.$ Further, both sides 
 are non-vanishing for any non-zero connection $\mathcal M$, 
 provided that $X_K$ has genus $\geq 2$. The same holds for the map $\delta^i(X_s/k)$. 
\end{lem}

To describe the generic fiber of $\pi(X/S),$ we recall the category $\mathrm{MIC}^\mathrm{se}(X/S)_K $
introduced in  \cite[1.3.1, 1.3.2]{DH18}. The objects of $\mathrm{MIC}^\mathrm{se}(X/S)_{K}$ 
are the same as those of $\mathrm{MIC}^\mathrm{se}(X/S)$ and for two objects $\mathcal{M}$ and  
$\mathcal{N}$ in $\mathrm{MIC}^\mathrm{se}(X/S)_K$ their hom-set is
$$ \mathrm{Hom}_{\mathrm{MIC}^\mathrm{se}(X/S)_K}(\mathcal{M},\mathcal{N}) := \mathrm{Hom}_{\mathrm{MIC}^\mathrm{se}(X/S)}(\mathcal{M},\mathcal{N}) \otimes_R K. $$
The category $\mathrm{MIC}^\mathrm{se}(X/S)_K \cong \mathrm{MIC}^{\circ}(X/S)_K$ is abelian 
 \cite[Lemma 5.1.2]{DH18}. Moreover, this category is Tannakian  with
the fiber functor 
$$ \eta^{\ast} \otimes K: \mathrm{MIC}^\mathrm{se}(X/S)_K \rightarrow \mathrm{Vec}_K,$$
and it is Tannakian dual to the generic fiber of $\pi(X/S),$ that is, 
$$ \mathrm{MIC}^\mathrm{se}(X/S)_K \cong \mathrm{Rep}^\mathrm{f}(\pi(X/S)_K).$$
We can similarly define $\mathrm{MIC}^\mathrm{geom}(X/S)_K$, and we also obtain 
$$ \mathrm{MIC}^\mathrm{geom}(X/S)_{K} \cong \mathrm{Rep}^\mathrm{f}(\pi^\mathrm{geom}(X/S)_K).$$
 
\begin{lem}\label{131}
The category $\mathrm{MIC}^\mathrm{se}(X/S)_{K}$ is full subcategory of category $\mathrm{MIC}^\circ(X_K/K).$ Moreover, $$\pi(X_K/K) \twoheadrightarrow \pi(X/S)_K.$$
\end{lem}
\begin{proof}
The category $\mathrm{MIC}^\mathrm{se}(X/S)_K$ can be considered as full subcategory of $\mathrm{MIC}^\circ(X_K/K)$ through the following restriction functor
\begin{align*}
\mathrm{MIC}^\mathrm{se}(X/S) &\rightarrow \mathrm{MIC}^\circ(X_K/K)\\
(\mathcal{M},\nabla) & \mapsto (\mathcal{M}_K, \nabla_K).
\end{align*}
The restriction functor is compatible with the fiber functors of these two categories. Therefore, we have 
homomorphims
$$\pi(X_K/K) \rightarrow \pi(X/S)_K.$$

We show the surjectivity of this map using the criterion of Deligne-Milne, \cite[Theorem 2.21]{DM82}. 
Thus, we need to show that for each object $\mathcal{M}_{0} \in \mathrm{MIC}^\mathrm{se}(X/S)_K,$ when considered as object in $\mathrm{MIC}^{\circ}(X_K/K)$, all its sub-objects will be objects in $\mathrm{MIC}^\mathrm{se}(X/S)_K.$ 

There exists, by the definition $\mathrm{MIC}^\mathrm{se}(X/S)_K,$ 
an $\mathcal{M} \in \mathrm{Obj}(\mathrm{MIC}^{\circ}(X/S))$ such that $i^{\ast}\mathcal{M} = \mathcal{M}_0,$ where $i = i_{U/R}\circ i_{K/U}$ (see the maps $i_{K/U}$ and $i_{U/R}$ in the  diagram below). Let $\mathcal{N}_0$ be a sub-object of $\mathcal{M}_0.$ Then there exists an open sub-scheme $U$ of $\mathrm{Spec}(R)$  and $\mathcal{N}_{U}$  is a sub-connection of ${\mathcal{M}}_{|_{X_U}}$ such that $i_{K/U}^{\ast}\mathcal{N}_{U} = \mathcal{N}_{0},$ where $i_{K/U}$ is the map in the following Cartesian diagram:
$$\xymatrix{
X_K\ar[r]^{i_{K/U}} \ar[d]& X_{U} \ar[r]^{i_{U/R}} \ar[d]      & X \ar[d]\\
\mathrm{Spec}(K)  \ar[r]   &   U \ar[r]     &\mathrm{Spec}(R).
}$$
We set $\mathcal{N}:= {i_{U/R}}_{\ast}\mathcal{N}_{U} \cap \mathcal{M}.$ The sheaf $\mathcal{N}$ is defined as the sheafification of the subpresheaf of $i_{\ast}i^{\ast}\mathcal{M}.$ Since $\mathcal{N}$ can be considered as the subconnection of $\mathcal{M},$ the sheaf $\mathcal{N}$ is locally free due to Proposition 5.1.1  in \cite{DH18}.

We are  now ready to prove $i^{\ast}\mathcal{N}= \mathcal{N}_0,$ i.e., $\mathcal{N}_{0} \in \mathrm{Obj}(\mathrm{MIC}^\mathrm{se}(X/S)_K).$   Then 
$$ i^{\ast}\mathcal{N} = i^{\ast}{i_{U/R}}_{\ast} \mathcal{N}_U \cap i^{\ast}\mathcal{M} = i^{\ast}{i_{U/R}}_{\ast} \mathcal{N}_U \cap \mathcal{M}_0,    $$
so we will prove that 
$$i^{\ast}{i_{U/R}}_{\ast} \mathcal{N}_U = \mathcal{N}_{0}$$
to finish the proof. Indeed, since $i_{U/R}^\ast {i_{U/R}}_\ast \mathcal{N} = \mathcal{N}$ (as $i_{U/R}$ is an open immersion) and $i^\ast = i^\ast_{K/U}\circ i^\ast_{U/R},$ the result follows.
\end{proof}
Using Lemma \ref{131} and Lemma \ref{lem-groupH}, we obtain the following morphism:
\begin{align}\label{eq-24}
     \pi(X_K/K) \twoheadrightarrow \pi(X/S)_K \twoheadrightarrow \pi^{\mathrm{geom}}(X/S)_K.
\end{align}
 We have $(-)^{\pi^\mathrm{geom}(X/S)_K} = (-)^{\pi(X/S)_K} = (-)^{\pi(X_K/K)},$ hence from the "universal $\delta$-functor" formalism we get for every $\pi^{\mathrm{geom}}(X/S)_K$-module $V_K$ a system of maps
$$ \lambda_K^i: \mathrm{H}^{i}(\pi^{\mathrm{geom}}(X/S)_K, V_K) \longrightarrow \mathrm{H}^i(\pi(X/S)_K,V_K)\longrightarrow \mathrm{H}^{i}(\pi(X_K/K),V_K).$$

The main result of this subsection is the following: 
\begin{prop}\label{130}
    Let $R$ be a Dedekind domain. Let $X$ be a projective curve over $R$ with genus $g\geq 1.$ Let $V_K$ be an object in $ \mathrm{Rep}^\mathrm{f}(\pi^\mathrm{geom}(X/S)_K).$ Then the maps
\begin{align}\label{eq-6}
\lambda^i_{K}: \mathrm{H}^{i}(\pi^\mathrm{geom}(X/S)_K, V_K)\stackrel{\gamma^i} \longrightarrow \mathrm{H}^i(\pi(X/S)_K,V_K)   \longrightarrow \mathrm{H}^{i}(\pi(X_K/K),V_K)
\end{align}
for all $i\geq 0$ are isomorphisms.
\end{prop}
\begin{proof}
 
\textit{Bijectivity of $\lambda^1_K$}. Let $V$ be an object in $\mathrm{Rep}^\mathrm{f}(\pi^\mathrm{geom}(X/S))$ that is finite projective as an $R$-module. According to Corollary \ref{cor-comparison}, we have 
\begin{align}\label{160}
\nu^1: \mathrm{H}^1(\pi^\mathrm{geom}(X/S),V) \overset{\cong}{\rightarrow} \mathrm{H}^1_\mathrm{dR}(X/S, \mathcal{V}).
\end{align}
Applying the flat base change theorems for de Rham cohomology \cite[23.5]{ABC20} and for group cohomology  \cite[Proposition 4.13]{Ja03}, we obtain the following diagram:
$$\xymatrix{
{\mathrm{H}^1(\pi^\mathrm{geom}(X/S)_K,V_K)}\ar[rr]^{\lambda^1_K} && \mathrm{H}^1(\pi(X_K/K),V_K)\ar[dd]^{\delta^1(X_K/K)} \\
	{\mathrm{H}^1(\pi^\mathrm{geom}(X/S),V)\otimes_RK}\ar[d]_{\nu^1\otimes K}\ar[u]^\cong	&&  \\
	{\mathrm{H}^1_\mathrm{dR}(X/S, \mathcal{V})\otimes_RK}\ar[rr]_\cong && {\mathrm{H}^1_\mathrm{dR}(X_K/K,\mathcal{V}_K)}.
}$$

From \eqref{160}, we can see that the map $\nu^1 \otimes K $ is bijective. Thanks to Lemma \ref{lem-998}, the morphism $\delta^1(X_K/K)$ is also bijective. Thus, the morphism $\lambda_K^1$ is bijective. 

We now prove that $\lambda^1_K(V_K)$ is bijective for an arbitrary object $V$ in $\mathrm{MIC}^\mathrm{geom}(X/S).$ Let $\mathcal{V}$ be an corresponding integrable connection corresponding to $V.$ We define $\mathcal{V}^{\mathrm{tor}}$ as the torsion sub-sheaf of $\mathcal{V},$ it is a sub-sheaf of $\mathcal{V}$ such that the quotient sheaf $\mathcal{V}^{\mathrm{fr}}:= \mathcal{V}/\mathcal{V}^\mathrm{tor} $ is a locally free $\mathcal{O}_X$-module. Consider the following short exact sequence in $\mathrm{MIC}^\mathrm{geom}(X/S):$
$$ 0 \longrightarrow \mathcal{V}^\mathrm{tor} \longrightarrow \mathcal{V} \longrightarrow \mathcal{V}^\mathrm{fr}\longrightarrow 0.$$ Hence, 
$$\mathcal{V}_K \cong \mathcal{V}^{\mathrm{fr}}_K.$$

Applying again the flat base change theorems for de Rham cohomology \cite[23.5]{ABC20} and for group cohomology  \cite[Proposition~4.13]{Ja03}, we obtain the following diagram:
$$\xymatrix{
\mathrm{H}^1(\pi^\mathrm{geom}(X/S)_K,V_K)\ar[rr]^{\lambda^1_K}  \ar[d]^\cong&& {\mathrm{H}^1(\pi(X_K/K),V_K)}\ar[d]^{\delta^1(X_K/K)} \\
	\mathrm{H}^1(\pi^\mathrm{geom}(X/S)_K,V^{\mathrm{fr}}_K)	&&  \mathrm{H}^1(X_K/K, \mathcal{V}_K)\ar[d]_\cong\\
      &&    \mathrm{H}^1(X_K/K, \mathcal{V}^\mathrm{fr}_K)\\
		\mathrm{H}^1(\pi^\mathrm{geom}(X/S),V^{\mathrm{fr}})\otimes_R K\ar[rr]^{\nu^1\otimes K} \ar[uu]_{\cong} && \mathrm{H}^1_\mathrm{dR}(X/S,\mathcal{V}^\mathrm{fr})\otimes_R K.  \ar[u]^{\cong}
}$$
From \eqref{160}, we can see that the map $\nu^1 \otimes K $ is bijective. Thanks to Lemma \ref{lem-998}, the morphism $\delta^1(X_K/K)$ is also bijective. Thus, the morphism $\lambda_K^1$ is bijective for every representation $V_{K}$ in $\mathrm{Rep}^\mathrm{f}(\pi(X/S)_K).$

\noindent\textit{Injectivity of $\lambda_K^2$}. We apply the Lemma \ref{lem-injective} to the map
$$ \lambda^1_K: \mathrm{H}^1(\pi^\mathrm{geom}(X/S)_K,V_K)\longrightarrow \mathrm{H}^1(\pi(X_K/K), V_K).$$
Since we have proved that $\lambda^1_K(V_K)$ is bijective for an arbitrary representation $V_K$ in $\mathrm{Rep}^\mathrm{f}(\pi^\mathrm{geom}(X/S)),$ it suffices to choose $J$ to be an injective envelope of $V_K$ in $\mathrm{ind}$-$\mathrm{Rep}^\mathrm{f}(\pi^\mathrm{geom}(X/S)).$

\noindent\textit{Surjectivity of $\lambda_K^2$}. There are two cases: the genus is $1$ or larger than 
or equal to $2.$ 

\noindent Case $g\geq 2.$ We need a version of Lemma 4.4 in \cite{BHT25}.

\noindent\textbf{Claim}. 
\textit{Let $M$ be an object in $\mathrm{Rep}^\mathrm{f}(\pi^\mathrm{geom}(X/S)_K).$ Then there exists an object $M'$ in $\mathrm{Rep}^\mathrm{f}(\pi^\mathrm{geom}(X/S)_K)$ and an injective map $j: M\hookrightarrow M'$ such that 
$$\mathrm{H}^2(\pi(X_K/K),M')=0.$$}
\noindent    \textit{Verification}.  
Thanks to the bijectivity of $\delta^{\bullet}(X_K/K):\mathrm{H}^i(\pi(X_K/K),M_K) \rightarrow \mathrm{H}^i_{\mathrm{dR}}(X_K/K, \mathcal{M}_K) $ (Lemma \ref{lem-998}), 
and the Poincar\'e duality for de Rham cohomology on $X_K/K$,  
the claim is equivalent to showing that
\begin{center}\it  there exists a surjective morphism
$ N\twoheadrightarrow M^\vee$ in $\mathrm{Rep}^\mathrm{f}(\pi^\mathrm{geom}(X/S)_K)$ such that  
$\mathrm{H}^{0}(\pi(X_K/K), N) =0.$\end{center}
By induction on the dimension of $\mathrm H^0(\pi(X_K/K),N),$ it suffices to
show that there exists  $N$ surjecting on $M^\vee$ and satisfying the strict inequality:
$$\mathrm{h}^{0}(\pi(X_K/K), N) < 
\mathrm{h}^0(\pi(X_K/K),M^\vee),$$
as long as $\mathrm{h}^0(\pi(X_K/K),M^\vee)\neq 0$.

Consider $K$ as the trivial representation of $\pi^\mathrm{geom}(X/S)_K$, assume that there exists an inclusion 
$K\hookrightarrow M^\vee$ and let 
 $F:= M^{\vee}/K $ be the quotient representation. 
 Consider an arbitrary, non-trivial, irreducible representation $P$
and apply the functor 
$\mathrm{Ext}^i_{\mathrm{Rep}(\pi(X_K/K))}(-,P)$
to the short exact sequence in (the full subcategory) $\mathrm{Rep}^\mathrm{f}(\pi^\mathrm{geom}(X/S)_K):$
$$0 \longrightarrow K \overset{e}{\longrightarrow} 
 M^{\vee} \longrightarrow F \longrightarrow 0, $$
we obtain a long exact sequence:
$$  \mathrm{Ext}^1_{\mathrm{Rep}(\pi(X_K/K))}(F,P) \longrightarrow 
\mathrm{Ext}^1_{\mathrm{Rep}(\pi(X_K/K))}({M}^{\vee},P)
\overset{\mathrm{ev}_{e}}{\longrightarrow} 
\mathrm{Ext}^1_{\mathrm{Rep}(\pi(X_K/K))}(K,P)
\longrightarrow \mathrm{Ext}^2_{\mathrm{Rep}(\pi(X_K/K))}(F,P).$$
By Poincar\'e duality we have 
\begin{align*}\mathrm{Ext}^2_{\mathrm{Rep}(\pi(X_K/K))}(F,P) &\cong  \mathrm{H}^2(\pi(X_K/K), F^\vee\otimes P)\\
&\cong \mathrm{H}^0(\pi(X_K/K), F\otimes P^\vee)^\vee\\
&\cong
\mathrm{Hom}_{\mathrm{Rep}^\mathrm{f}(\pi^\mathrm{geom}(X/S)_K)}(P,F)^\vee.
\end{align*}

According to Lemma \ref{lem-simple}, there exist infinitely many non-isomorphic simple objects in $\mathrm{MIC}^\mathrm{geom}(X/S)_K.$ Since $\mathrm{MIC}^\mathrm{geom}(X/S)_K$ is equivalent to 
$\mathrm{Rep}^\mathrm{f}(\pi^\mathrm{geom}(X/S)_K)$, we can choose a non-trivial irreducible representation $P$  such that 
$$\mathrm{Hom}_{\mathrm{Rep}^\mathrm{f}(\pi^\mathrm{geom}(X/S)_K)}(P,F)^\vee= 0,$$ 
consequently, the map 
$$\mathrm{ev}_e:\mathrm{Ext}^1_{\mathrm{Rep}(\pi(X_K/K))}
(M^{\vee},P) {\longrightarrow} 
\mathrm{Ext}^1_{\mathrm{Rep}(\pi(X_K/K))}(K,P)$$
is surjective. 
 
Since $\delta^{1}(X_K/K)$ is bijective  and the first group cohomology $\mathrm{H}^{1}(\pi(X_K/K),M_K)$ is non-vanishing (Lemma \ref{lem-998}),  there exists  
a non-split extension of connections:
$$\varepsilon:  P\longrightarrow  Q\longrightarrow K.$$
With the assumption that $P$ 
is non-trivial and simple, we have  
 $\mathrm{H}^0(\pi(X_K/K), Q) =0$.   
Let $\epsilon$ be a pre-image of $\varepsilon$ along the surjective map   
$\mathrm{ev}_{e}$.  
In terms of extensions, these are related by the following diagrams in $\mathrm{Rep}^\mathrm{f}(\pi^\mathrm{geom}(X/S)_K)$
\begin{align}\label{eq-23}
    \xymatrix{ 
\qquad\qquad \epsilon: 0\ar[r]& P\ar[r] \ar@{=}[d]& 
N\ar[r]& M^{\vee} 
 \ar[r]& 0\\ 
 \varepsilon= \mathrm{ev}_e(\epsilon): 0\ar[r]&  P  \ar[r]&
  Q\ar[r]\ar[u]& K\ar[r]\ar[u]_e&0.}
\end{align}
The object $N$ still belongs to $\mathrm{Rep}^\mathrm{f}(\pi^\mathrm{geom}(X/S)_K),$ since
\begin{align*}
     \mathrm{Ext}^1_{\mathrm{Rep}^\mathrm{f}(\pi^\mathrm{geom}(X/S)_K)}(M^\vee, P) &\cong \mathrm{Ext}^1_{\mathrm{Rep}(\pi^\mathrm{geom}(X/S))}(K, M \otimes P) \\
     &\cong \mathrm{H^1}(\pi^\mathrm{geom}(X/S)_K, M\otimes P)\\
     & \cong \mathrm{H}^1(\pi(X_K/K),M\otimes P)\\
     & \cong \mathrm{Ext}^1_{\mathrm{Rep}(\pi(X_K/K))}(K, M \otimes P)\\
     & \cong \mathrm{Ext}^1_{\mathrm{Rep}^\mathrm{f}(\pi(X_K/K))}(M^\vee,  P).
\end{align*}
  
Taking the long exact sequence of group cohomology associated with Diagram \eqref{eq-23},  we obtain
$$\xymatrix{
0=\mathrm{H}^{0}(\pi(X_K/K), P) \ar[r]&
\mathrm{H}^{0}(\pi(X_K/K), N)\ar[r] &
\mathrm{H}^{0}(\pi(X_K/K), M^\vee)\ar[r]&
\mathrm{H}^{1}(\pi(X_K/K),P)\\
0=\mathrm{H}^{0}(\pi(X_K/K),P) \ar[r]\ar@{=}[u]&
\mathrm{H}^{0}(\pi(X_K/K),Q) =0\ar[r]\ar[u] &
\mathrm{H}^{0}(\pi(X_K/K), K)\ar@{^(->}[r]\ar@{^(->}[u]&
\mathrm{H}^{1}(\pi(X_K/K),P).\ar@{=}[u]
}$$
Since $\mathrm{H}^{0}(\pi(X_K/K),K)=K\neq 0$, the rightmost
upper horizontal map
$$\mathrm{H}^{0}(\pi(X_K/K),  M^\vee)\longrightarrow
\mathrm{H}^{1}(\pi(X_K/K),P)$$
should be non-zero, that is
$\mathrm{H}^{0}(\pi(X_K/K), N)\neq 
\mathrm{H}^{0}(\pi(X_K/K), M^\vee),$
whence the desired inequality. $\vartriangleright$

Let $M$ be an object in $\mathrm{Rep}^\mathrm{f}(\pi^\mathrm{geom}(X/S)_K)$ and let $M \hookrightarrow M'$ as in Claim above, that is,  
${\mathrm{H}^2(\pi(X_K/K), M')}=0$.
The short exact sequence 
$$0 \rightarrow M  \rightarrow M' \rightarrow M'/M \rightarrow 0,$$ in $\mathrm{Rep}^\mathrm{f}(\pi^\mathrm{geom}(X/S)_K)$
can be considered as in $\mathrm{Rep}^\mathrm{f}(\pi(X_K/K)).$ Since the map
$\lambda_K^i$ in \eqref{eq-6} is a morphism of $\delta$-functors, we have the following diagram:
$$\xymatrix{
	{\mathrm{H}^1(\pi^\mathrm{geom}(X/S)_K,M'/M)}\ar[r]\ar[d]^{\lambda_{K}^1}_\cong & {\mathrm{H}^2(\pi^\mathrm{geom}(X/S)_K,M)} \ar[r]\ar@{^(->}[d]^{\lambda_{K}^2}& {\mathrm{H}^2(\pi^\mathrm{geom}(X/S)_K,M')} \ar@{^(->}[d]^{\lambda_{K}^2}
\\
	{\mathrm{H}^1(\pi(X_K/K), M'/ M)} \ar@{->>}[r]& {\mathrm{H}^2(\pi(X_K/K), M)} \ar[r]& {\mathrm{H}^2(\pi(X_K/K), M')}=0. }$$
We conclude that the middle arrow is surjective. 

\noindent Case $g=1.$ Since  $X$ is an elliptic curve, it has a group structure. 
By the K\"unneth theorem for $\pi(X_K/K)$ \cite[Corollary 10.47]{De89}, it possesses another multiplication induced by the group structure on $X$. 
Using the Eckmann-Hilton argument, we conclude that $\pi(X_K/K)$ is abelian, leading to the decomposition
$$ \pi(X_K/K) = \pi(X_K/K)^{\mathrm{uni}} \times \pi(X_K/K)^{\mathrm{red}}. $$
Furthermore, $\pi^\mathrm{uni}(X_K/K)\cong \mathbb{G}_a\times \mathbb{G}_a$, by means of \cite[(1.2), (1.16)]{LM82}. 

On the other hand, Lemma \ref{131} implies that the group $\pi(X/S)_K$ is also abelian, and we have the surjection
\begin{align*}
  \pi^\mathrm{uni}(X_K/K) \twoheadrightarrow \pi(X/S)_K^{\mathrm{uni}} 
\end{align*}
Since $\mathrm{H}^1(  \pi^\mathrm{uni}(X_K/K), K) = \mathrm{H}^1 (\pi(X/S)_K^{\mathrm{uni}},  K), $ it follows that $\pi(X/S)_K \cong \mathbb{G}_a\times \mathbb{G}_a.$ Now, using Lyndon-Hochschild-Serre spectral sequence, we rewrite the map $\lambda^i_K$ for $i\geq 2$ as
$$ \mathrm{H}^i(\pi(X/S)_K^{\mathrm{uni}}, V_K) \longrightarrow \mathrm{H}^i(  \pi^\mathrm{uni}(X_K/K), V_K).  $$
 Since $  \pi^\mathrm{uni}(X_K/K) \cong \pi(X/S)_K^{\mathrm{uni}}$   the map $\lambda^i_K$ is an  isomorphism for $i\geq 2.$

\noindent\textit{Bijectivity of $\lambda^i_K$  ($i\geq 3$).}
Since de Rham cohomology vanishes in 
degrees $i\geq 3$, it suffices to show that $\lambda^i_K$ is injective. We apply the same argument as in the proof of the injectivity of  $\lambda^2_K$ and the result follows.
\end{proof}

\begin{lem}\label{lem-simple}
    There exist infinitely many non-isomorphic simple objects in $\mathrm{MIC}^\mathrm{geom}(X/S)_K.$
\end{lem}
\begin{proof}
    The idea is to identify irreducible objects among finite connections in the sense of \cite[Definition~2.4]{EH08}. We construct  a quotient $\pi^{\mathrm{geom,et}}(X/S)_K$ of $\pi^{\mathrm{geom}}(X/S)_K$, which is pro-\'etale but not finite. Consequently, it admits infinitely many irreducible representations.

    Recall that the finite objects in $\mathrm{MIC}^{\mathrm{coh}}(X/k)$ correspond to finite-dimensional representations of the \'etale fundamental group $\pi^{\mathrm{et}}(X)$, which can be regarded as a profinite quotient of $\pi(X/k)$. Under our assumptions, there is a short exact sequence of \'etale fundamental groups:
\begin{align}\label{40}
        1\longrightarrow \pi^{\mathrm{et}}(X_s) \longrightarrow \pi^{\mathrm{et}}(X) \longrightarrow \pi^{\mathrm{et}}(S)\longrightarrow 1.
\end{align}
This follows from the compatibility of the \'etale fundamental group with base change 
between two algebraically closed fields. We may assume, after replacing $k$ 
with an algebraically closed field extension of a subfield finitely generated over  
$\mathbb{Q},$ that  $k \subset \mathbb{C},$ and then base change to $\mathbb{C}.$ 
 Hence, we may assume $k = \mathbb{C}$. In this case, $S$ is a $K(\pi,1)$-space, and the homotopy exact sequence of the topological fundamental groups associated with $f\colon X \to S$ is short exact. Using Anderson's criteria for preserving injectivity after pro-finite completion \cite[Section 2]{And74} then implies that the sequence above is exact; see also \cite[Proposition~2.5]{EK24}.

 The inflated finite connections generate a full tensor subcategory of $\mathrm{MIC}^{\mathrm{geom}}(X/S)$, closed under taking sub-quotients, and hence define a quotient $\pi^{\mathrm{geom,et}}(X/S)$ of $\pi^{\mathrm{geom}}(X/S).$ This group scheme is strictly profinite in the sense of \cite[Definition~8.7]{HdS23}.
More precisely, the full subcategory generated by a single inflated finite connection $\mathrm{inf}(\mathcal{V})$ is Tannakian dual to a finite group scheme: indeed, the \'etale Galois covering that trivializes $\mathrm{inf}(\mathcal{V})$ yields a pushforward of the structure sheaf equipped with a canonical connection, which is finite and generates a Tannakian subcategory containing $(\mathcal{V}, \nabla).$ As a consequence, the generic fiber of
$\pi^{\mathrm{geom,et}}(X/S)$ is pro-finite, too.

On the other hand, the Tannakian criterion for exact sequences of affine group schemes \cite[Appendix]{EHS08}, together with the exactness of the \'etale sequence \eqref{40}, implies that the closed fiber of $\pi^{\mathrm{geom,et}}(X/S)$ at $s$ is isomorphic to $\pi^{\mathrm{et}}(X_s)$. Indeed, any connection on $X_s$ arises as a sub-quotient of the restriction of a connection on $X/k$, and thus belongs to $\mathrm{MIC}^{\mathrm{geom,et}}(X/S)_s$.
Since $\pi^{\mathrm{et}}(X_s)$ is not finite, the same holds for $\pi^{\mathrm{geom,et}}(X/S)$ and its generic fiber.

Therefore, $\pi^{\mathrm{geom,et}}(X/S)_K$ is pro-\'etale and not finite, and hence admits infinitely many irreducible representations. As it is a quotient of $\pi^{\mathrm{geom}}(X/S)_K$, the same holds for $\pi^{\mathrm{geom}}(X/S)_K$.
\end{proof}

\begin{cor}\label{143}
Let $X$ be a projective curve with genus $g\geq 1.$ Let $\mathcal V$ be an object in $ \mathrm{MIC}^\mathrm{geom}(X/S).$ Denote $V=\eta^*(\mathcal V)$. Then the map
 $$\nu^i \otimes K:\mathrm{H}^i(\pi^\mathrm{geom}(X/S), V) \otimes K \longrightarrow \mathrm{H}^i_\mathrm{dR}(X/S,\mathcal{V}) \otimes K$$ 
is an isomorphism for  $i\geq 0.$ 
\end{cor}
\begin{proof}
Using flat base change for de Rham cohomology and group cohomology, we have the following diagram:
$$\xymatrix{
\mathrm{H}^i(\pi^\mathrm{geom}(X/S),V)\otimes_R K\ar[rr]^{\nu^i\otimes K} \ar[d]^\cong && \mathrm{H}^i_{\mathrm{dR}}(X/S,\mathcal{V})\otimes K\ar[dd]^{\cong} \\
	\mathrm{H}^i(\pi^\mathrm{geom}(X/S)_K,V_K)\ar[d]_{\lambda^i_K}	&&  \\
	\mathrm{H}^i(\pi(X_K/K), V_K)\ar[rr]^{\delta^i(X_K/K)} && {\mathrm{H}^i_\mathrm{dR}(X_K/K,\mathcal{V}_K)}.
}$$
As $\lambda^i_K$ is an isomorphism via Proposition \ref{130} and $\delta^i(X_K/K)$ is also an
isomorphism, so is  $\delta^i \otimes K$.  
\end{proof}

\subsection{Comparison at the closed fiber}
Let $s$ be a closed point of $S$. In what follows we 
identify $k$ with the residue field $k_s$ at $s$ and
shorten the notation for the tensor product $-\otimes_Rk_s$ to $(-)_s$.
Our aim now is to verify that the map
$$\nu^2_s: \mathrm{H}^2(\pi^\mathrm{geom}(X/S), V)_s\longrightarrow  \mathrm{H}^2_\mathrm{dR}(X/S, \mathcal{V})_s$$
is injective.
To define the special fiber of $\pi(X/S)$ at the closed point $s$ of $S =\mathrm{Spec}(R),$ we recall the category $\mathrm{MIC}^\mathrm{se}(X/S)_s $. 
Let $\mathfrak{m}_s$ be the maximal ideal that determines $s.$ Then
$\mathrm{MIC}^\mathrm{se}(X/S)_{s}$ is the full subcategory in 
$\mathrm{MIC}^\mathrm{se}(X/S)$ of objects annihilated by $m_s$.
We then have (cf. \cite[Chap.~10]{Ja03}):
$$\mathrm{MIC}^\mathrm{se}(X/S)_s\cong \mathrm{Rep}^\mathrm{f}(\pi(X/S)_s).$$
We define $\mathrm{MIC}^\mathrm{geom}(X/S)_{s}$ similarly and we also obtain 
$$ \mathrm{MIC}^\mathrm{geom}(X/S)_{s} \cong \mathrm{Rep}^\mathrm{f}(\pi^\mathrm{geom}(X/S)_{s}).$$

Similarly to \eqref{eq-24}, we also have the following morphism:
\begin{align}\label{163}
    \pi(X_s/k) \twoheadrightarrow \pi(X/S)_{s} \twoheadrightarrow \pi^{\mathrm{geom}}(X/S)_{s}.
\end{align}
 We have $(-)^{\pi^\mathrm{geom}(X/S)_{s}} = (-)^{\pi(X/S)_{s}} = (-)^{\pi(X_{s}/k)},$ hence from the "universal $\delta$-functor" formalism we get for every $\pi^{\mathrm{geom}}(X/S)_{s}$-module $V_{s}$ a system of maps
$$ \lambda^i_{s}: \mathrm{H}^{i}(\pi^{\mathrm{geom}}(X/S)_{s}, V_{s}) \longrightarrow \mathrm{H}^i(\pi(X/S)_{s},V_{s})\longrightarrow \mathrm{H}^{i}(\pi(X_{s}/k),V_{s}).$$

\begin{prop}\label{prop-20}
        Let $R$ be a Dedekind domain. Let $X$ be a projective curve over $R$ with genus $g\geq 1.$ Let $V_{s}$ be an object in $ \mathrm{Rep}^\mathrm{f}(\pi^\mathrm{geom}(X/S)_{s}),$ where $V$ corresponds to an inflated connection. Then the map
\begin{align*}
     \lambda^2_{s}: \mathrm{H}^{2}(\pi^\mathrm{geom}(X/S)_{s}, V_{s}) \longrightarrow \mathrm{H}^2(\pi(X/S)_{s},V_{s})   \longrightarrow \mathrm{H}^{2}(\pi(X_{s}/k),V_{s})
\end{align*}
is injective.
\end{prop}
\begin{proof} Our strategy is to utilize Lemma \ref{lem-injective}. Thus our first task will be
checking the bijectivity of $\lambda^1_s$.

\noindent  \textit{Bijectivity of $\lambda^1_{s}$}. Thanks to \eqref{163}, the map
    $$ \lambda^1_{s}: \mathrm{H}^1(\pi^\mathrm{geom}(X/S)_{s},V_{s}) \longrightarrow \mathrm{H}^1(\pi(X_{s}/k),V_{s})$$
is injective for an arbitrary representation $V_{s}$ in $\mathrm{Rep}^\mathrm{f}(\pi(X/S)_{s}).$
    
    Let $\mathcal{V}$ be an inflated connection. According to Corollary \ref{cor-comparison}, we have 
\begin{align}\label{161}
\nu^1: \mathrm{H}^1(\pi^\mathrm{geom}(X/S),V) \overset{\cong}{\rightarrow} \mathrm{H}^1_\mathrm{dR}(X/S, \mathcal{V}).
\end{align}
According to  \cite[4.18 - p.64]{Ja03}, the  map
$$ \mathrm{H}^i(\pi(X/S),V) \otimes k  \hookrightarrow \mathrm{H}^i(\pi(X/S)_s, V_s), $$
for $i\geq 0,$ is injective. Combined with the base change theorem for de Rham cohomology \cite[(8.0.4.1)]{Ka70}, we obtain the following diagram:
$$\xymatrix{
{\mathrm{H}^1(\pi^\mathrm{geom}(X/S)_{s},V_{s})}\ar@{^{(}->}[rr]^{\lambda^1_{s}} && \mathrm{H}^1(\pi(X_{s}/k),V_{s})\ar[dd]^{\delta^1(X_{s}/k)} \\
	{\mathrm{H}^1(\pi^\mathrm{geom}(X/S),V)\otimes_R k}\ar[d]_{\nu^1\otimes k}\ar@{^{(}->}[u]	&&  \\
	{\mathrm{H}^1_\mathrm{dR}(X/S, \mathcal{V})\otimes_Rk}\ar[rr]_\cong && {\mathrm{H}^1_\mathrm{dR}(X_{s}/k,\mathcal{V}_{s})}.
}$$
From \eqref{161}, we can see that the map $\nu^1 \otimes k $ is bijective. Thanks to Lemma \ref{lem-998}, the morphism $\delta^1(X_{s}/k)$ is also bijective. Thus, the morphism $\lambda_{k}^1$ is bijective. 

\noindent\textit{Injectivity of $\lambda^2_{s}$}. We first need a exact sequence realizing
$\pi^\mathrm{geom}(X/S)_s$ as a normal subgroup of $\pi(X/k)$.  
Recall the exact sequence \eqref{eq_exact_diag}. Take the fiber at $s$ we obtain the following 
sequence 
$$1 \longrightarrow \pi^\mathrm{geom}(X/S)_s\longrightarrow \Pi(X/k)^\Delta_s \longrightarrow \Pi(S/k)^\Delta_s \longrightarrow 1, $$
which is exact. Indeed, the faithfully flatness of the  map is known, cf. Theorem \ref{HES};
while the exactness in the middle and on the left is expressed as a fiber product (cf. Appendix 
\ref{sect_kernel}), 
hence it is preserved by base change. 

We also notice that $\Pi(X/k)^\Delta_s =\pi(X/k)$ -- the
differential fundamental group of $X$ (with base point at $x$), and same for $S$. Thus the above 
exact sequence can be expressed as
$$1 \longrightarrow \pi^\mathrm{geom}(X/S)_s\longrightarrow \pi(X/k) \longrightarrow \pi(S/k)
\longrightarrow 1.$$
According to \cite[Corollary 5.13 - p.78]{Ja03} we have the following:

\noindent \textbf{Claim}. 
\textit{The induction functor $\mathrm{Ind}_{\pi^\mathrm{geom}(X/S)_s}^{\pi(X/k)}$ is an exact functor.}

Let $\mathcal{V}$ be an inflated connection. We prove the injectivity of $\lambda^2_{s}$ by apply Lemma \ref{lem-injective} to
$$\lambda^1_{s}: \mathrm{H}^1(\pi^\mathrm{geom}(X/S)_s, V_s)\longrightarrow \mathrm{H}^1(\pi(X_s/k),{V}_s).$$
Notice that $\mathcal{V}$ as an absolute connection corresponds to a representation $\mathcal V_x$
of $\pi(X/k)$ which when restricted to $\pi^\mathrm{geom}(X/S)_s$ is just $V_s$.
In fact, we have
$$\mathcal V_x=x^*(\mathcal V)=s^*\circ\eta^*(\mathcal V)=V_s.$$

Consider the following short exact sequence: 
$$ 0 \longrightarrow V_s\stackrel{\rho_{ V_s}} \longrightarrow  V_s \otimes_t \mathcal{O}(\pi(X/k)) \longrightarrow Q \longrightarrow 0,$$
in $\mathrm{Rep}(\pi(X/k))$ (which is equivalent to $\mathrm{Ind-MIC}(X/k)$ 
by means of
the fiber functor at $x$).
 
We now check the conditions in Lemma \ref{lem-injective}.
\begin{itemize}    
    \item [i)] We have shown above that $\lambda^1_{s}$ is bijective at $V_s$. 
    \item  [ii)] Choose 
$J = V_s\otimes\mathcal{O}(\pi(X/k)).$
According to \cite[Proposition 4.10 - p.54]{Ja03} we have
 $$\mathrm{H}^i(\pi^\mathrm{geom}(X/S)_s,J) = \mathbf{R}^i\mathrm{Ind}_{\pi^\mathrm{geom}(X/S)_s}^{\pi(X/k)}(V_s) = 0$$ for all $i\geq 1$.  Since $J$ is a direct limit of representations that corresponding to absolute connections,  $\lambda^1_{s}$ is bijective at $J$, consequently  
 $$\mathrm{H}^1(\pi(X_{s}/k), J)=0.$$  Since $Q_s$ is also a direct limit of representations that corresponding to inflated connections
we also have $\lambda^1_{s}$ is injective at $Q$.
\end{itemize}
Therefore, the map $\lambda^2_{s}$ is bijective at $V_s$.
\end{proof}

 \begin{cor} \label{Cor-144}
Let $f:X\longrightarrow S$ be a projective curve with genus $g\geq 1.$ Let $\mathcal V$ be an inflated connection, 
$V=\eta^*(\mathcal V)$. Then the map
 $$\nu^2 \otimes k:\mathrm{H}^2(\pi^\mathrm{geom}(X/S), V)_s \longrightarrow \mathrm{H}^2_\mathrm{dR}(X/S,\mathcal{V})_s$$ 
is an injection.
 \end{cor}
 \begin{proof}
According to  \cite[4.18 - p.57]{Ja03}, the  map
$$ \mathrm{H}^i(\pi(X/A),M)_s  \hookrightarrow \mathrm{H}^i(\pi(X/A)_s, M_s), $$
for $i\geq 0,$ is injective. 
Now, the base change theorem for de Rham cohomology \cite[(8.0.4.1)]{Ka70} yields the following diagram:
$$\xymatrix{
\mathrm{H}^2(\pi^\mathrm{geom}(X/S),V)_s\ar[rr]^{\nu^2_s} \ar@{^{(}->}[d] && \mathrm{H}^2_{\mathrm{dR}}(X/S,\mathcal{V})_s\ar[dd]^{\cong} \\
	\mathrm{H}^2(\pi^\mathrm{geom}(X/S)_s,V_s)\ar[d]_{\lambda^2_{s}}	&&  \\
	\mathrm{H}^2(\pi(X_{s}/k), V_s)\ar[rr]^{\delta^2(X_{s}/k)} && {\mathrm{H}^2_\mathrm{dR}(X_{s}/k,\mathcal{V}_{s})}.
}$$
Since  $\lambda_s^2 $ is injective and so is $\delta^2(X_s/k)$, 
according to Lemma \ref{lem-998}, we conclude that the map $\nu^2_s$ is injective.   
 \end{proof}

\subsection{Comparison theorem}
  We are now ready to prove: 
\begin{thm}\label{thm-comparison}
Let $f:X\longrightarrow S$ be a projective curve with genus $g\geq 1$, where $S$
is a smooth affine curve. Assume that $S$ admits a section.
Let $\mathcal{V}$ be an object in $\mathrm{MIC}^\mathrm{geom}(X/S),$ 
and $V$ be the  corresponding representation of $\pi^\mathrm{geom}(X/S)$. 
Then the map
 $$\nu^i: \mathrm{H}^i(\pi^\mathrm{geom}(X/S),V)\overset{\gamma^i}{\longrightarrow}\mathrm{H}^i(\pi(X/S), V) \overset{\delta^i}{\longrightarrow} \mathrm{H}^i_\mathrm{dR}(X/S, \mathcal{V})$$ 
is  bijective for all $i
\geq 0.$ 
\end{thm}
\begin{proof} Assume that $S=\mathrm{Spec}(R)$ and $K$ be the quotient field of $R$.

\noindent\textit{Case 1}. Assume that  $V$ corresponds to an inflated connection $\mathrm{inf}(\mathcal{V}).$

\textit{The case of $i = 2.$}  It then follows from Lemma \ref{133}, Corollary \ref{143}, and Corollary \ref{Cor-144} that
\begin{itemize}
\item [(1)]   $\nu^2$ is injective
\item [(2)]   $\nu^2 \otimes K$ is bijective
\item [(3)]   $\nu^2 \otimes k_s$ is injective for any closed point $s$ of $\mathrm{Spec}(R).$
\end{itemize}
Hence $\nu^2$ is an isomorphism.

\textit{The case of $i\geq 3.$} Thanks to Corollary \ref{143} and the flat base change theorem for de Rham cohomology \cite[23.5]{ABC20}, we obtain
$$ \mathrm{H}^i(\pi^\mathrm{geom}(X/S),V) \otimes K \cong \mathrm{H}^i_{\mathrm{dR}}(X/S, \mathcal{V}) \otimes K \cong \mathrm{H}^i_{\mathrm{dR}}(X_K/K,\mathcal{V}_K) = 0.$$
For any closed point $s \in \mathrm{Spec}(A),$ since $\nu^2 $ is bijective, we use the same argument as in the proof of the bijectivity of $\lambda^1_{s}$ in Proposition $\ref{prop-20}$ to conclude that 
$$      \lambda^2_{s}: \mathrm{H}^{2}(\pi^\mathrm{geom}(X/S)_{s}, V_{s})   \longrightarrow \mathrm{H}^{2}(\pi(X_{s}/k),V_{s})  $$ 
is bijective. We now use the same argument as in the proof of Corollary \ref{Cor-144}, we obtain 
$$ \mathrm{H}^3(\pi(X/S),V)\otimes k \hookrightarrow \mathrm{H}^3_{\mathrm{dR}}(X_s/k, \mathcal{V}_s) = 0.$$
Hence,
   $$\mathrm{H}^3(\pi^\mathrm{geom}(X/S),V) \cong \mathrm{H}^3_{\mathrm{dR}}(X/S,\mathcal{V}) =0.$$
We repeat the same arguments to obtain 
$$ \mathrm{H}^i(\pi^\mathrm{geom}(X/S),V) \cong \mathrm{H}^i_{\mathrm{dR}}(X/S,\mathcal{V}) =0.$$

\noindent\textit{Case 2}. Assume that $U$ corresponds to a sub-connection $\mathcal{U}$  of $\mathrm{inf}(\mathcal{V}).$ We have the short exact sequence:
\begin{align}\label{seq-3}
     0 \rightarrow U \rightarrow V \rightarrow Q \rightarrow 0 
\end{align}
where  $Q: =M/V$ is the quotient representation corresponding to the connection $\mathcal{Q}:=\mathrm{inf}(\mathcal{V})/\mathcal{U}.$ Thus we also have an exact sequence
$$ 0 \rightarrow \mathcal{U} \rightarrow \mathrm{inf}(\mathcal{V}) \rightarrow \mathcal{Q} \rightarrow 0.$$
The  associated long exact sequences of group cohomology and de Rham cohomology yield the following commutative diagram:
$$\xymatrix{\mathrm{H}^{1}(\pi^\mathrm{geom}(X/S), Q) \ar[r] \ar[d]^{\cong}&\mathrm{H}^{2}(\pi^\mathrm{geom}(X/S), U) \ar[r] \ar@{^{(}->}[d]^{\delta^2(U)} & \mathrm{H}^{2}(\pi^\mathrm{geom}(X/S), V) \ar[r] \ar[d]^{\cong}  & \mathrm{H}^{2}(\pi^\mathrm{geom}(X/S), Q) \ar@{^{(}->}[d] \ar[r] &0\\
 \mathrm{H}^1_\mathrm{dR}(X/S, \mathcal{Q})   \ar[r]  & \mathrm{H}^2_\mathrm{dR}(X/S, \mathcal{U})\ar[r] & \mathrm{H}^2_\mathrm{dR}(X/S, \mathrm{inf}(\mathcal{V})) \ar[r] &  \mathrm{H}^2_\mathrm{dR}(X/S, \mathcal{Q}) \ar[r]&0. }$$  
The four-lemma implies that $\delta^2(U)$ is bijective.  For $i\geq 3,$ we argue similarly and obtain 
$$ \mathrm{H}^i(\pi^\mathrm{geom}(X/S), U) \cong \mathrm{H}^i_{\mathrm{dR}}(X/S, \mathcal{U}) =0.$$

\noindent\textit{Case 3}. Finally, we prove the claim for the quotient $U/T$, where $T,U$ correspond
to the sub-connections $\mathcal{T}\subset \mathcal{U}\subset \mathrm{inf}(\mathcal{V}).$
Similarly as above, we have  a commutative diagram with exact lines:
$$\xymatrix{\mathrm{H}^{2}(\pi(X/A), T) \ar[r] \ar[d]^{\cong}&\mathrm{H}^{2}(\pi(X/A), U) \ar[r] \ar[d]^{\cong} & \mathrm{H}^{2}(\pi(X/A), U/T) \ar[r] \ar@{^{(}->}[d]^{\delta^2(U/T)}  & 0  \\
 \mathrm{H}^2_\mathrm{dR}(X/A, (\mathcal{T}))   \ar[r]  & \mathrm{H}^2_\mathrm{dR}(X/A, \mathcal{U})\ar[r] & \mathrm{H}^2_{\mathrm{dR}}(X/A, \mathcal{U}/\mathcal{T}) \ar[r] &  0. }$$
This diagram implies that $\delta^2(U/T)$ is bijective. For $i\geq 3,$ we argue similarly and obtain 
$$ \mathrm{H}^i(\pi^\mathrm{geom}(X/S), U/T) \cong \mathrm{H}^i_{\mathrm{dR}}(X/S, \mathcal{U}/\mathcal{T}) =0.$$
\end{proof}


\begin{cor}\label{cor-114}
    Let $f:X\longrightarrow S$ be a projective curve with genus $g\geq 1$, where $S$
is a smooth affine curve. Assume that $S$ admits a section.
Let $\mathcal{V}$ be an object in $\mathrm{MIC}^\mathrm{geom}(X/S),$ 
and $V$ be the  corresponding representation of $\pi(X/S)$. 
Then the map
 $$\delta^i: \mathrm{H}^i(\pi(X/S), V) {\longrightarrow} \mathrm{H}^i_\mathrm{dR}(X/S, \mathcal{V})$$ 
is  bijective for all $i
\geq 0.$ 
\end{cor}
\begin{proof}
By \eqref{eq-113} and \eqref{eq-22}, the morphisms $\delta^i$ are isomorphisms for $i=0,1.$ Lemma \ref{lem-injective} implies that $\delta^2$ is injective; when combined with Theorem \ref{thm-comparison}, this shows that $\delta^2$ is bijective. For $i\geq 3,$ we use the argument from Theorem \ref{thm-comparison} to conclude that
$$ \mathrm{H}^i(\pi(X/S), V) \cong \mathrm{H}^i_\mathrm{dR}(X/S, \mathcal{V}) =0.$$\end{proof}
\subsection{An application}
As an application of  Theorem \ref{thm-comparison}, we exhibit here some examples
of surfaces which are de Rham $K(\pi,1)$. We begin with a direct corollary of \ref{thm-comparison}.
\begin{cor}\label{Cor-Kpi1}
Let $f:X\longrightarrow S$ be a relative smooth projective curve with genus $g\geq 1$ over a base $S$,
which is a smooth affine curve. Assume that $f$ admits a section. 
Then $X$, as a $k$-surface, is de Rham $K(\pi,1)$. That is, for all connections
$(\mathcal{V},\nabla)$ on $X/k$ and the corresponding representation $V$ of
the differential fundamental group $\pi(X/k,x)$ the map
    $$ \mathrm{H}^i(\pi(X/k),V) \longrightarrow \mathrm{H}^i_{\mathrm{dR}}(X/k,\mathcal{V})$$
is  bijective for all $i \geq 0.$  
\end{cor}
\begin{proof}
    We have the natural map from the Lyndon--Hochschild-Serre spectral sequence (see Proposition \ref{prop-21}):
      $$ E_2^{p,q}= \mathrm{H}^p(\Pi(S/k), \mathrm{H}^q(\pi^\mathrm{geom}(X/S), V)) \Rightarrow \mathrm{H}^{p+q}(\Pi(X/k),V)$$
   to Leray spectral sequence (see \cite[23.3.1]{ABC20}):
    $$ E_2^{p,q}= \mathrm{H}^p_{\mathrm{dR}}(S/k,\mathrm{H}^q_{\mathrm{dR}}(X/S,\mathcal{V})) \Rightarrow \mathrm{H}^{p+q}_{\mathrm{dR}}(X/k, \mathcal{V}). $$
Thanks to Theorem \ref{thm-comparison}, the corresponding terms in two spectral sequences are 
isomorphic. Hence, 
$$ \mathrm{H}^i(\Pi(X/k),V) \cong \mathrm{H}^i_{\mathrm{dR}}(X/k, \mathcal{V}), \hspace{0.2cm}\text{for all}\; i\geq0.$$

On the other hand, we have the following equivalence between three categories:
$$ \mathrm{Rep}(\pi(X/k)) \cong \mathrm{MIC}^{\mathrm{ind}}(X/k)\cong \mathrm{Rep}(R:\Pi(X/k)), $$
which implies that the map 
    $$ \mathrm{H}^i(\pi(X/k),V) \longrightarrow \mathrm{H}^i(\Pi(X/k),V) \longrightarrow\mathrm{H}^i_{\mathrm{dR}}(X/k,\mathcal{V})$$
is  bijective for all $i \geq 0.$  
\end{proof}

For an arbitrary smooth projective map $f:X\longrightarrow S$, there may be no sections to $f$.
However, by shrinking $X$ as a family on $S$ to some \'etale neighborhood, we can have sections. This
observation yields the following: 
\begin{prop}\label{pro_surf_kpi1}
Let $f:X\longrightarrow S$ be a relative smooth projective curve with genus $g\geq 1$ over a base $S$ which is 
a smooth affine curve. Assume that $f$ admits a section. Then we can shrink $S$ to a (Zariski)
open $U$ so that $X_{U}$ as a $k$-surface is de Rham $K(\pi,1)$.
\end{prop}
\begin{proof}
There exists a finite (possibly ramified) Galois map $\tilde S\longrightarrow S$ such that
the map $f_{\tilde S}: X_{\tilde S}\longrightarrow \tilde S$ admits a section (take
a Galois extension $\tilde K$ of $K$ such that $X_{\tilde K}$ admits a rational point and let $\tilde R$
be the normalization of $R$ in $\tilde K$, $\tilde S:=\mathrm{Spec}\tilde R$ and apply 
\cite[Thm~II.4.7]{Ha77}). Let $S'$ be an affine 
open of $\tilde S$ which is a Galois \'etale cover of its image $U$ which is open in $S$.
Thus, we have $f_{S'}: X_{S'}\longrightarrow S'$ which admits a section and $X_{S'}\longrightarrow X_U$
which is Galois \'etale. Now Corollary \ref{Cor-Kpi1} ensures that $X_{S'}$ is de Rham $K(\pi,1)$, and
Lemma \ref{lem_descend_kp1} below shows that $X_U$ is also de Rham $K(\pi,1)$. 
\end{proof}

\begin{lem}\label{lem_descend_kp1}
Let $f:Y\longrightarrow X$ be a Galois \'etale cover under a group $G$, where $X,Y$ are
smooth schemes over a field $k$ of characteristic zero. Assume that $Y$ is de Rham $K(\pi,1)$.
Then so is $X$.
\end{lem}
\begin{proof}
Let $(\mathcal{V},\nabla)$ be an object in $\mathrm{MIC}^{\mathrm{coh}}(X)$ and
let $V$ be the corresponding representation of $\pi(X)$ (for simplicity, we omit the base points).

According to \cite[Proposition 1.4.4]{Ka87}, there is an exact sequence
\begin{align}\label{eq-galois}
    1 \longrightarrow \pi^{\mathrm{diff}}(Y) \longrightarrow \pi^\mathrm{diff}(X)\longrightarrow G \longrightarrow 1.
\end{align}
The Lyndon-Hochschild-Serre spectral sequence associated with this exact sequence  
yields
$$ E_2^{p,q} = \mathrm{H}^p(G, \mathrm{H}^q(\pi^\mathrm{diff}(Y),V) \Longrightarrow \mathrm{H}^{p+q}(\pi^\mathrm{diff}(X),V).$$
Since $G$ is a finite group and $k$ has characteristic zero, 
$$ \mathrm{H}^i(\pi^\mathrm{diff}(X),V) \cong \mathrm{H}^i(\pi^\mathrm{diff}(Y),V)^G,$$
for all $i\geq 0.$

On the other hand, 
$$ \mathrm{H}^0_{\mathrm{dR}}(X, \mathcal{V}) = \mathrm{H}^0_{\mathrm{dR}}(Y, f^\ast(\mathcal{V}))^G.$$
Since $(-)^G$ is an exact functor, the Grothendieck spectral sequence (see Remark \ref{rem-2}) degenerates
so that
$$\mathrm{H}^i_{\mathrm{dR}}(X, \mathcal{V})\cong \mathrm{H}^i_{\mathrm{dR}}(Y, f^\ast(\mathcal{V}))^G,$$
for all $i\geq 0.$ 

Finally, since the pullback functor $f^\ast$ corresponds to 
the restriction functor along the map  $\pi^\mathrm{diff}(Y)) \longrightarrow\pi^\mathrm{diff}(X)),$ the result follows.
\end{proof}

\appendix

\section{Affine group schemes over Dedekind domain}  \label{app_group_scheme}
Let $k$ be a field of characteristic $0$.
Let $R$ be a $k$-algebra, which is a Dedekind ring.
In what follows $R$-group schemes are always assumed to be \textit{flat and affine}. 
Our reference is \cite{DH18} and \cite{De90}, see also \cite{Ja03, Sa72}. 
\subsection{Representations ($G$-modules)}\label{specialsub-quotients}
Let $G$ be an affine flat $R$-group scheme.
A representation of $G$ in $R$-modules  (or a $G$-module) is $R$-modules equipped with a (rational) action of $G$. Since $G$ is affine, this is nothing but a comodule over
the Hopf algebra $R[G]$.

Since $G$ is flat, $\mathrm{Rep}(G)$ is an abelian category. 
Furthermore, it is a  tensor category. The full subcategory consisting of $R$-finite 
(resp. finite projective) 
representations will be denoted by  $\mathrm{Rep}^\mathrm{coh}(G)$
(resp.  $\mathrm{Rep}^{\circ}_R(G)$). $\mathrm{Rep}^{\circ}_R(G)$ is an $R$-linear, additive, rigid tensor category. Moreover, it \textit{dominates}
$\mathrm{Rep}^\mathrm{coh}(G)$, i.e., each object of $\mathrm{Rep}^\mathrm{coh}_R(G)$ 
is a quotient of an object in
$\mathrm{Rep}^{\circ}_R(G)$. Further,  $\mathrm{Rep}(G)$ is and ind-category of
$\mathrm{Rep}^\mathrm{coh}(G)$, i.e. each $G$-module is the union of its
$R$-coherent $G$-submodules.
 
As $R$ is a Dedekind ring, torsion free, flat and projective finite $R$-modules are the same. 
We say  that $M\subset N$ is a \textit{special} sub-object in $\mathrm{Rep} _R(G)$ if the quotient $N/M$ is $R$-flat. 
A \textit{special sub-quotient} is by definition a special sub of a quotient, or equivalently,   
a quotient of a special sub. This can be seen from the following diagram, where the left square is a 
push-out, or equivalently, a pull-back:
$$\xymatrix{
Q\ar@{^(->}[r]\ar@{->>}[d] \ar@{}[rd]|\Box& P\ar@{->>}[d]\ar[r]& P/Q
\ar[d]^\simeq  \\
M\ar@{^(->}[r] & N\ar[r]& M/N.}$$
Thus, if $N/M\cong P/Q$ is $R$-flat then $M$ is a special sub-quotient of $P$: it is a quotient of a special sub $Q$ and a special sub of $N$. 

We also formulate these in terms of comodules over $R[G]$ for the later use. 

 A subcomodule $M$ of an $R[G]$-comodule $N$ is said to be \textit{special} if $N/M$ is flat over
$R;$ a \textit{special sub-quotient} $M$ of an $R[G]$-comodule N is a special sub-module of a quotient of $N,$ or,
equivalently, a quotient of a special sub-module of $N.$
We also use the similar notations for comodules as $\mathrm{Comod}_R C$, 
$\mathrm{Comod}^\mathrm{coh}_RC$ or
$\mathrm{Comod}^\circ_RC$ for any $R$-flat coalgebra $C$.

\subsection{Tannakian duality over a Dedekind ring}\label{sect-TanDual2} 
The reference for this subsection is \cite{Sa72}, see also \cite{DH18}.  
Assume that $\mathcal C$ is an $R$-linear 
abelian tensor category.  
Denote by  $\mathcal C^{\rm o}$ the full subcategory of  $\mathcal C$ 
consisting of rigid (i.e. dualizable)
objects. We say that $\mathcal C$ is {\em dominated} by $\mathcal C^\text{o}$ if each object of 
$\mathcal C$ is a quotient of an object from $\mathcal C^\text{o}$.

 \begin{defn}\label{defn-Tannakaovering}
 A (neutral) Tannakian category  over a Dedekind ring $R$ is an $R$-linear abelian tensor category $\mathcal C$, dominated by $\mathcal C^{\rm o}$, together with an exact faithful tensor functor $\omega:\mathcal C\to {\sf Mod}(R)$.  
\end{defn}

\begin{thm}[{\cite[Thm.~II.4.1.1]{Sa72}}] \label{th_duality}
Let $(\mathcal C,\omega)$ be a neutral Tannakian category over a Dedekind ring $R$.
Then the group functor $A\mapsto {\sf Aut}_A^\otimes(\omega\otimes A)$ is representable by a flat group scheme $G$ and $\omega$ factors through an equivalence between 
$\mathcal C$ and $\mathrm{Rep}_R(G)$.
\end{thm}

Let $f:G\to G'$ be a homomorphism of flat affine group schemes over $R$. 
Then it induces the \textit{restriction functor}
$f^*:\mathrm{Rep}(G')\longrightarrow \mathrm{Rep}(G)$
which is  a tensor functor. This functor will also be denoted by $\mathrm{Res}^{G'}_G$
when we want to emphasize on the groups. This functor also
restricts to the subcategories of $R$-coherent,
$R$-finite projective representations.

\begin{thm}[{\cite[Proposition 3.2.1 and Theorem 4.1.2]{DH18}}] \label{th1}
Let $f: G  \longrightarrow G'$ be a homomorphism of  affine flat groups over $R,$ and 
a tensor functor   be the  restriction functor, which also 
 \begin{enumerate}
 \item $f$ is (generically) faithfully flat if and only if $f^*: \mathrm{Rep}_R ^{\circ}(G') \longrightarrow \mathrm{Rep}_R ^{\circ}(G)$ is fully faithful and its image is closed under taking
 (special) sub-objects.  
 \item  $f$ is a closed immersion if and only if every object of  $\mathrm{Rep}_R ^{\circ}(G)$ is isomorphic to a {\em special} sub-quotient of
 an object of the form  $f^*(X'),$ for $X' \in \mathrm{Obj}(\mathrm{Rep}_R ^{\circ}(G')).$
 \end{enumerate}
\end{thm}
We say that $f$ is \textit{surjective} or a {\em quotient map} if it is faithfully flat.

Let $ G \longrightarrow A$ be a homomorphism of affine group schemes over $R$. 
The kernel of this map is defined to be
$$L:=G\times_A\text{\rm Spec}(R).$$
This is a closed subgroup of $G$. 
The sequence 
\begin{displaymath}
\xymatrix{1 \ar[r] &  L \ar[r]^{q}&G\ar[r]^{p}& A \ar[r] &1}    
 \end{displaymath}                  
is said to be \textit{exact} if  $p$  is a quotient map with kernel $L.$ We will provide a criterion for the exactness in terms of the functors 
\begin{eqnarray}\label{eq}
\xymatrix{\mathrm{Rep}_R ^{\circ}(A)\ar[r]^{p^*}&\mathrm{Rep}_R^{\circ}(G) \ar[r]^{q^*}& \mathrm{Rep}_R ^{\circ}(L).}
\end{eqnarray}

\begin{thm}[Theorem 4.2.2 \cite{DH18}] \label{thm-ExactSequence}
  Let us be given a  sequence of homomorphisms 
 \begin{displaymath}
 \xymatrix{L \ar[r]^{q}&G\ar[r]^{p}& A}
 \end{displaymath}
 with $q$ a  closed  immersion  and $p$ faithfully  flat. 
 Then this sequence is  exact  if  and  only  if  the  following  conditions  are  fulfilled: \begin{itemize}
  \item[(a)]  For  an  object  $V$ in $\mathrm{Rep}_R ^{\circ}(G),$  $q^*(V )$  in  $\mathrm{Rep}_R ^{\circ}(L)$ is  trivial  if  and  only  if 
            $V \cong p^*U$  for some object $U$ in $\mathrm{Rep}_R ^{\circ}(A).$ 
 \item[(b)]  Let $W_0$   be the maximal trivial sub-object of $q^*(V )$ in $\mathrm{Rep}_R ^{\circ}(L).$ Then there 
            exists $V_0 \subset V \in \mathrm{Obj}(\mathrm{Rep}_R ^{\circ}(G)),$ such that $q^*(V_0) \cong W_0.$
 \item[(c)]  Any object $W$ in $\mathrm{Rep}_R ^{\circ}(L)$  is a quotient in (hence,  by taking duals,  a sub-object 
            of) $q^*(V )$ for some  $V \in \mathrm{Obj}(\mathrm{Rep}_R ^{\circ}(G)).$                             
\end{itemize}
\end{thm}

\subsection{Fiber criterion for surjective maps of flat affine group schemes} \label{sect_morphism_group_schemes}
The materials in this subsection are taken from \cite{DHdS18}. 
The aim is to prove Lemma \ref{lem-DHdS} which apprears as Theorem
\ref{prop-ismorphicofgalois}.


 Let $N$ be an $R[G]-$comodule. Then $N_{\text {tor}}$, the $R$-torsion sub-module of $N$ is an $R[G]$-subcomodule. Hence, for any $R[G]$-subcomodule $M$, the pre-image of $(N / M)_{\text {tor}}$ in $N$, denoted $M^{\text {sat }},$ is an $R[G]$-comodule. Since $R$ is a Dedekind ring, the quotient $N / M^{\text {sat }}$ is flat, being torsion-free. Thus, $M^{\text {sat }}$ is the smallest special subcomodule of $N$, containing $M$. It is called the \textit{saturation} of $M$ in $N.$  


 Let $\rho: \Pi \rightarrow \mathrm{G}$ be a morphism of
(flat affine) $R$-group schemes. We describe the ``images'' of $\rho$ in two ways.

\begin{defn}[The diptych]\label{defndiptych}
     Define $\Psi_\rho$ as the group scheme whose Hopf algebra is the image of $R[G]$ in $R[\Pi]$. Define $R[\Psi_\rho^{\prime}]$ as the saturation of $R[\Psi_\rho]$ inside $R[\Pi]$. The obvious commutative diagram
     $$\xymatrix{ \Psi'_{\rho} \ar[r]& \Psi_{\rho}\ar[d]\\ \Pi\ar[r]_{\rho} \ar[u]& G}$$
is called the diptych of $\rho$.
\end{defn}

\begin{rem} We have 
\begin{itemize}
        \item [(1)] The image of a Hopf algebra homomorphism is a Hopf sub-algebra (of the target).
        \item [(2)]  Implicit in the above definition is the fact that 
        ${R}[\Psi_\rho^{\prime}]=R[\Psi_{\rho}]^{\text {sat }}$ is a Hopf algebra. Indeed, we have the filtration $$R[\Psi_{\rho}]^{\text {sat }} \otimes R[\Psi_{\rho}]^{\text {sat }} \subset R[\Psi_{\rho}]^{\text {sat }} \otimes R[\Pi] \subset R[\Pi] \otimes R[\Pi],$$ the successive quotients of which are flat, and therefore $(R[\Pi] \otimes R[\Pi]) / (R[\Psi_{\rho}]^{\text {sat }} \otimes R[\Psi_{\rho}]^{\text {sat }})$ is also flat. Thus
$$
(R[\Psi_{\rho}'] \otimes R[\Psi_{\rho}])^{\text {sat }} \subset R[\Psi_{\rho}]^{\text {sat }} \otimes R[\Psi_{\rho}]^{\text {sat }} .
$$
Hence, by the definition of $R[\Psi_{\rho}']^{\text {sat }}$, we have
$$
\Delta\left(R[\Psi_{\rho}]^{\text {sat }}\right) \subset(R[\Psi_{\rho}] \otimes R[\Psi_{\rho}])^{\text {sat }} \subset R[\Psi_{\rho}]^{\text {sat }} \otimes R[\Psi_{\rho}]^{\text {sat }}.
$$
    \end{itemize}
\end{rem}

\begin{lem}\label{faithfullyflat}
    The morphism $\Pi \rightarrow \Psi_\rho^{\prime}$ is faithfully flat and the
morphism $\Psi_\rho\longrightarrow G$ is a closed immersion.
\end{lem}
\begin{proof}
The first claim follows from \cite[Theorem 4.1.1]{DH18} and the second claim
follows from definition of the diptych.
\end{proof}

 The following proposition generalizes Lemma 4.5 in \cite{DHdS18} from PID rings to Dedekind rings.
\begin{prop}\label{proplocaltoglobal}
     Assume that $\Psi_{\rho, {k}_s}^{\prime} \rightarrow \Psi_{\rho, {k}_s}$ is faithfully flat for every residue field $k,$ where $s$ is a closed point of $\mathrm{Spec}(R).$ Then $\Psi_\rho^{\prime} \rightarrow \Psi_\rho$ is an isomorphism.
\end{prop}
\begin{proof} This follows from \cite[Proposition~3.2]{DHH17} or  \cite[Theorem~1]{HHdS25}).
\end{proof}
%
Fix a closed point $s$ of $S$ and denote $k:=k_s$ for simplicity.
Over  $k$, there is   interesting group scheme in sight: the image of $\rho_{k}$. We then have the \textit{triptych} of $\rho_{k}$, which is the commutative diagram

\begin{align}\label{triptych}
    \begin{xy}
(15,30)*+{\Psi_{\rho,k}^{\prime}}="v1";(45,30)*+{\Psi_{\rho,k}}="v2";%
(30,15)*+{\mathrm{Im}(\rho_{k})}="v4";%
(15,0)*+{\Pi_{k}}="v6";(45,0)*+{G_{k}}="v7";%
 {\ar@{->>} "v6"; "v1"};%
{\ar@{->} "v1"; "v2"};{\ar@{->>} "v1"; "v4"};%
{\ar@{^(->} "v2"; "v7"};%
{\ar@{^(->} "v4"; "v2"};%
{\ar@{^(->} "v4"; "v7"};%
{\ar@{->} "v6"; "v7"};%
{\ar@{->>} "v6"; "v4"};%
\end{xy}.
\end{align}

Together with Proposition \ref{proplocaltoglobal}, diagram \eqref{triptych} proves the following:
\begin{cor}\label{Corim}
    
The following claims are true.
\begin{itemize}
    \item [i)] If $\operatorname{Im}\left(\rho_{k_s}\right) \rightarrow \Psi_{\rho, k_s}$ is an isomorphism for every residue field $k,$ then $\Psi_\rho^{\prime} \rightarrow \Psi_\rho$ is an isomorphism.
    \item [ii)]  If $\Psi_\rho^{\prime} \rightarrow \Psi_\rho$ is an isomorphism, then $\operatorname{Im}\left(\rho_{k_s}\right) \rightarrow \Psi_{\rho, k_s}$ is an isomorphism.
    \item [iii)] The image of $\Psi_{\rho, k_s}^{\prime}$ in $\Psi_{\rho, k_s}$ is none other than $\operatorname{Im}\left(\rho_{k_s}\right)$.
\end{itemize}
\end{cor}

%

\begin{defn}
    Let $\Pi$ be a (flat affine) $R$-group scheme and $V \in \mathrm{Rep}_R^{\circ}(\Pi).$ 
Denote $T^{a,b}(V):=V^{\otimes a}\otimes V^{\vee\otimes b}$. 
The category of all special sub-quotients of various $T^{a_1, b_1} ({V}) 
\oplus \cdots \oplus T^{a_m, b_{m}} ({V})$ is denoted by 
$\langle V\rangle_{\otimes}^{s}$.
\end{defn}

Let $V$ be a finite projective $R$-module. Then the group functor 
$\mathrm{GL}(V)$ is represented by an affine group scheme (see \cite[2.2. p. 20]{Ja03}). 
We assume that our $\mathrm{G}$ (in the diptych) equals $\mathrm{GL}(V ).$
We now interpret $\Psi_\rho$ and $\Psi_\rho^{\prime}$ in terms of their representation categories. 
\begin{prop}\label{PropTannaka}
 Let $ {V}$ be an object of 
 $\operatorname{Rep}_{\mathrm{R}}^{\circ}(\Pi)$ and $\rho$ be the natural homomorphism $\Pi \rightarrow \mathrm{G}:=\mathrm{GL}({V}).$
 \begin{itemize}
     \item [(1)] The obvious functor $\operatorname{Rep}_R\left(\Psi_\rho\right) \rightarrow \operatorname{Rep}_R(\Pi)$ defines an equivalence of categories between $\operatorname{Rep}_R^{\circ}\left(\Psi_\rho\right)$ and $\langle V\rangle_{\otimes}^s$.
     \item [(2)] The obvious functor $\operatorname{Rep}_R\left(\Psi_\rho^{\prime}\right) \rightarrow \operatorname{Rep}_R(\Pi)$ defines an equivalence between $\operatorname{Rep}_R\left(\Psi_\rho^{\prime}\right)$ and $\langle V \rangle_{\otimes}$.
 \end{itemize}

\end{prop}
\begin{proof}
Proof of (1).  According to Theorem \ref{th1}, the functor $\operatorname{Rep}^{\circ}_R\left(\Psi_\rho\right)  \rightarrow \operatorname{Rep}^{\circ}_R(\Pi)$ 
is fully faithful with image closed under taking special subjects. 

The co-action $\rho:  {V} \longrightarrow  {V} \otimes R[\Psi_{\rho}]$ induces a map
$$
\mathrm{Cf}: V^{\vee} \otimes V \longrightarrow R[\Psi_{\rho}],\qquad
\varphi \otimes m \mapsto \sum \varphi\left(m_i\right) m_i^{\prime},$$
where $\varphi \in V^{\vee}, m \in V$, $\Delta(m)=\sum_i m_i \otimes m_i^{\prime}.$
The image of this map is called the 
\textit{coefficient space} of $V$, and denoted by $\operatorname{Cf}(V)$. 
According to \cite[Lemma~1.1.7]{DH18}, we have
$$\mathrm{Comod}^{\circ}(\mathrm{Cf}(V) )\cong \langle V \rangle^s,$$
where $\mathrm{Comod}^{\circ}(\mathrm{Cf}(V))$ denotes the category
of comodules of $\mathrm{Cf}(V)$ in finite projective $R$-modules, and
 $\langle V \rangle^s$ denotes the full subcategory of special subquotients of 
\textit{direct sum} of copies of $V$. Now we have
$$ \langle V \rangle^s_\otimes=\bigcup_{a,b}  \langle T^{a,b}(V) \rangle^s=
\mathrm{Comod}^{\circ}(R[\Psi_{\rho}])=\mathrm{Rep}^\circ_R(\Psi_{\rho}).$$

%
%
%
%

Proof of (2).     
According to Theorem \ref{th1} the functor $\mathrm{Rep}_R (\Psi_{\rho}') \longrightarrow \mathrm{Rep}_R(\Pi)$ is fully faithful and its image closed under taking sub-object -- and
hence under taking subquotients. Consequently, $\langle V \rangle_{\otimes}$ is a full subcategory
of $\mathrm{Rep}_R (\Psi_{\rho}')$.

Thus it suffices to show that any $W\in \mathrm{Rep}_R (\Psi_{\rho}')^s$ belongs to 
$\langle V \rangle_{\otimes}$. Indeed, the coaction $W\longrightarrow W\otimes R[\Psi_{\rho}']$
realizes $W$ as a subquotient of the representation $W\otimes R[\Psi_{\rho}']$ where 
$\Psi_{\rho}'$ acts solely on the second tensor component. Since $W$ can be embedded in a 
free $R$-module, we conclude that, as a representation of $\Psi_\rho'$, it can be embedded
in some direct sum $R[\Psi_{\rho}']^{\oplus d}$. Now $R[\Psi_{\rho}']$ is the saturation
of $R[\Psi_{\rho}]$ in $R[\Pi]$, we conclude that $a\cdot W\subset R[\Psi_{\rho}]^{\oplus d}$ 
for some non-zero element $a\in R$. This induces an injective map from $W$ to $R[\Psi_{\rho}]^{\oplus d}$, which is, by construction $\Psi_\rho'$-linear.
Using part (1) we conclude that $W$ belongs to $\langle V \rangle_{\otimes}$. 
%
\end{proof}

Note that, in general, $\operatorname{Rep}_R\left(\Psi_\rho\right)$ is not a full subcategory of $\operatorname{Rep}_R(\Pi)$. This means that we have the following interpretation of the diptych (Definition \ref{defndiptych}) in terms of representation categories:
$$\xymatrix{ \langle V \rangle_{\otimes} \ar[d]_{\text{fully faithful}}&
      \langle V \rangle_{\otimes}^s\ar[d]^{\text{faithful}}
       \ar[l]\\ \mathrm{Rep}_R(\Pi)& \mathrm{Rep}_R^{\circ}(\mathrm{GL}(V)).  \ar[l]} $$

We now express the triptych 
(Diagram \eqref{triptych}) of $\rho$ in categorical terms. For that, given an $R$-linear category 
$\mathcal C,$  we define the special fiber $\mathcal C_s$ at a closed point $s$ of $S:=\mathrm{Spec}(R)$ 
to be the full subcategory whose objects $W$ 
are annihilated by $\mathfrak{m}_s,$ i.e. $m_s\cdot\mathrm{id}_W = 0$ in 
$\mathrm{Hom}_{\mathcal  C}(W,W) = 0,$ where $\mathfrak{m}_s$ is the maximum ideal of 
$R,$ which determines $s.$  One can show that $\mathcal C_s$ is equivalent to the scalar 
extension $\mathcal C_{k}$ where $k = R/\mathfrak m_s$ is the residue field of $R$ at $s.$ 
We then have a commutative diagram of solid arrows between $k$-linear abelian categories:

\begin{equation}
\label{dot}
\xymatrix{\mathrm{Rep}^\mathrm{coh}_R(\Psi_{\rho}^{\prime})_{(k)}
\ar[dd]_{\  \ \text{fully faithful} }
&&
			 \mathrm{Rep}^\mathrm{coh}_R(\Psi_{\rho})_{(k)}\ar[ll]\ar[ld]
			 \\
&\langle V\otimes k\rangle_\otimes 
\ar[ul]_{\ \ \text{fully faithful}}   \ar[ld]^{\  \ \text{fully faithful}}
\\ 
\mathrm{Rep}^\mathrm{coh}_R(\Pi)_{(k)}. &&
}
\end{equation}

From \cite[Part I, 10.1, 141]{Ja03} the categories $\operatorname{Rep}_R(-)_{(k)}$ are simply the corresponding representations categories of the group schemes obtained by base change $R \rightarrow k$. Since $V$ is a faithful representation of $\Psi_\rho,$  $V\otimes k$ is a faithful representation of $\Psi_\rho \otimes k$, so that each object of 
$\operatorname{Rep}^\mathrm{coh}_R\left(\Psi_\rho\right)_{(k)}$ is a sub-quotient of some 
$\bigoplus T^{a_i,b_i}(V  \otimes k)$. This means that the upper horizontal arrow in 
Diagram \eqref{dot} factors through $\langle V  \otimes k\rangle_{\otimes}$, i.e. 
the dotted arrow exists and still produces a commutative diagram. We conclude that Diagram 
\eqref{dot} is Tannakian dual to Diagram \eqref{triptych} as the former can easily be 
completed by introducing the representation category of the general linear group in the lower right corner.

\begin{thm}[Lemma \ref{lem-DHdS}]\label{prop-ismorphicofgalois}
Let $(\mathcal{V},\nabla)$ be a connection on $X/k$. Then each locally free relative connection
in  $\langle \inf (\mathcal{V})\rangle_\otimes$ is indeed an 
object of $\langle\inf (\mathcal{V})\rangle_\otimes^s$.
\end{thm}
\begin{proof}
    The $R$-point $\eta$ induces an equivalence of abelian tensor categories 
 \begin{align}\label{1}
     \eta^{\ast}: \langle \inf (\mathcal{V})\rangle_{\otimes} \longrightarrow 
     \mathrm{Rep}^\mathrm{coh}_R(G'),
 \end{align}
where $G'$ is a flat group scheme over $R$ (see Theorem \ref{th_duality}).   
We consider the diptych of 
$\rho: G'\longrightarrow \mathrm{GL}(\eta^{\ast}\inf (\mathcal{V})) $  (Definition \ref{defndiptych}):
$$\xymatrix{
\Psi'_{\rho}\ar[r]  & G \ar@{^(->}[d]\\
G'  \ar[r]_{\hspace{-1.3cm}\rho}   \ar[u]& \mathrm{GL}(\eta^{\ast}\inf (\mathcal{V})). 
}$$
Theorem \ref{th1} and Proposition \ref{PropTannaka}-(2) show that the left
vertical arrow above is an isomorphism. Moreover, according to Proposition \ref{PropTannaka}-(1), the functor $\eta^{\ast}$ induces an equivalence 
\begin{align}\label{25}
    \langle \inf (\mathcal{V}) \rangle^s_{\otimes} \overset{\sim}{\longrightarrow} \mathrm{Rep}_R^{\circ}(G).
\end{align}
 Thus, the statement of this lemma can be proved if we can show that $G$ isomorphic to $G'.$

Let $\nu:G'\longrightarrow G$ be the natural map. 
To prove it is an isomorphism, we only need to treat the local property, that is, 
to prove that for any closed point $s\in S$, $k=k_s$,
 $\nu_k:G'_{k} \longrightarrow G_{k} $ is faithfully flat
 (cf. Proposition \ref{proplocaltoglobal}). 
By means of Corollary \ref{Corim} (i) and Diagram \eqref{dot}, this reduces to showing
that the functor 
 \begin{align*}
   \mathrm{Rep}^\mathrm{coh}_k(G_{k}) \longrightarrow \langle \inf (\mathcal{V})|_{X_{s}} \rangle_{\otimes}
 \end{align*}
is an equivalence of categories. By abusing notation, we shall denote this functor
by $\nu^*_k$.

As $\nu^*_k$ is a tensor functor between Tannakian categories over a field $k$ it is faithful. We show that it is full and essentially surjective.

We first show the essential surjectivity. 
Let $\mathcal W $ be an object in $\langle \inf (\mathcal{V})|_{X_{s}}\rangle_\otimes$. 
It follows from the exact sequence \ref{HES} 
and the criterion for the exactness of sequences \ref{thm-ExactSequence} (for group schemes over
a field, cf. \cite{EHS08}), that $\mathcal W$ can be realized as the image of a morphism 
$f_k:\mathcal N_k\to\mathcal  N'_k$ where $\mathcal N$ and $\mathcal N'$ are objects in 
$\langle \inf (\mathcal{V})\rangle_\otimes$.
Using the exactness criterion \ref{thm-ExactSequence} we also have
\begin{equation}\label{eq.R}
\mathrm{Hom}(\inf(\mathcal N),\inf(\mathcal N'))\otimes_Rk\cong 
\mathrm{Hom}(\mathcal N_k,\mathcal N'_k).\end{equation}
In deed, the hom-set 
$$\mathrm{Hom}(\mathcal N_k,\mathcal N'_k)\cong 
\mathrm{Hom}(\mathbf 1,\mathcal N'_k\otimes\mathcal N_k^\vee)$$ 
is equal to the fiber at $x$ of the maximal trivial subobject of 
$\mathcal N'_k\otimes \mathcal N_k^\vee$, which, according to 
\ref{thm-ExactSequence}, is the pull-back of a subobject $\mathcal T$ of 
$\mathcal N'\otimes\mathcal  N^\vee$ in $\mathrm{MIC}^\mathrm{coh}(X/k)$, moreover, 
$\mathcal T$ is the pull-back of some object in $\mathrm{MIC}^\mathrm{coh}(S/k)$.
Consequently, the inflated module $\inf({\mathcal T})$ is trivial as an object in 
$\mathrm{MIC}^\mathrm{coh}(X/S)$. Therefore the map in \eqref{eq.R} is surjective, hence an isomorphism. 

As a consequence, we can lift $f_k$ to a morphism $f:\inf(\mathcal N)\to\inf(\mathcal  N')$. Further
we can saturate $f$ (using the Tannakian duality) so that $\mathrm{Im}(f)\otimes_Kk=\mathrm{Im}(f_k)$.
 Thus we obtain an object $\mathrm{Im} f$ 
in $\langle \inf(\mathcal V)\rangle_{\otimes}^s$ with the property 
$(\mathrm{Im} f)|_{X_s}=\mathcal W$. 
That means, $\mathcal W$ is in the image of $\nu^*_k$, in other words, $\nu^*_k$ is essentially surjective.

The fullness of $\nu^*_k$ is proved in a similar way. Let $W$ be a representation in 
$\mathrm{Rep}^\mathrm{coh}_k(G_k)$. Then 
$W$ is a quotient of an object in $\mathrm{Rep}^\circ_R(G)$. According to Lemma \ref{lem-Deligne},  
$W$ is then a quotient of some $\eta^*(\mathcal N_k)$, where $\mathcal N$ is an object in
$\langle\mathcal V\rangle_\otimes\subset \mathrm{MIC}^\mathrm{coh}(X/k)$.  
Dualizing this, for each $W$ in 
$\mathrm{Rep}^\mathrm{coh}_k(G_k)$ we can embed 
it into some $\eta^*(\mathcal N_k)$ with 
$\mathcal N\in \langle \mathcal V\rangle_\otimes$.

Let $\varphi: \nu^*_k(U)\to \nu^*_k(W)$ be a morphism in 
$\langle \inf(\mathcal V)\rangle_{\otimes k}$.  
We represent $U$ as a quotient of $\eta^*(\mathcal N_k)$ and $W$ as a 
subobject of $\eta^*(\mathcal N'_k)$, where 
$\mathcal N,\mathcal N'\in \langle \mathcal V\rangle_\otimes$. 
Then $\varphi$ is represented by a map $f:\mathcal N_k\to \mathcal N'_k$ in 
$\langle \inf(\mathcal V)\rangle_{\otimes k}$. 
The previous discussion shows that $f$ has a pre-image under $\nu^*_k$, $f=\nu^*_k(\psi)$.
As $\nu^*_k$ is 
a faithfully exact functor between abelian categories and 
$\nu^*_k(\psi)=\nu^*_k(p)\circ \varphi \circ \nu^*_k(q)$ where $p$ in injective and $p$ is surjective,
we conclude that  $\varphi$ also has pre-image under $\nu^*_k$, 
\end{proof}

\section{Cohomology of groupoid schemes} \label{sect_groupoid_scheme}
\subsection{Representation of groupoid schemes}\label{sect_repr}
Our reference is \cite[Section 2--4]{De90}.  
\subsubsection{Groupoid schemes} Let  $S$ be a $k$-scheme.
An \textit{affine $k$-groupoid scheme acting on $S$}
is an $S\times_kS$-affine scheme $G$ (with the structure morphisms being $s,t:G\to S$, which are called the source and the target morphisms),
together with the following data:  
\begin{itemize}
\item[(i)]   a morphism $m:G\sttimes G\to G$, called the product of
$G$, satisfying the following associativity property:
$$m(m\sttimes\text{\rm id}_G)=m(\text{\rm id}_G\sttimes m)$$
\item[(ii)]  a morphism $\varepsilon:S\to G$, called the unit
element map, satisfying:
$$m(\varepsilon\sttimes\text{\rm id}_G)=
m(\text{\rm id}_G\sttimes m)=\text{\rm id}_G$$
\item[(iii)]  a morphism $\iota:G\to G$, called the inverse map,
satisfying:
$$\iota\circ s=t;\quad \iota\circ t=s$$
$$m(\iota \sttimes \text{\rm id}_G)=\varepsilon\circ s,\quad
m(\text{\rm id}_G\sttimes\iota)=\varepsilon\circ t,$$
\end{itemize}
where $\sttimes$ denotes the fiber product over $S$ with respect
to the morphisms $s$ and $t$.

A \textit{morphism} of $k$-groupoid schemes acting on a $k$-scheme
$S$ is a morphism of the underlying $k$-schemes which is
compatible with all structure morphisms.

An affine $k$-groupoid scheme $G$ acting on $S$  is called \textit{flat} if the source and target morphisms $s,t: G \longrightarrow S$ are flat. The groupoid scheme $G$ said to be \textit{acting transitively on $S$}
if for any pair of morphism $(a,b):T\times U\to S$,
if there is a faithfully flat quasi-compact morphism $\phi:W\to T\times U$ such that the
set
$$\text{\rm Mor}_{S\times S}(W,G)\neq\emptyset.$$ 
This implies that $(s,t):G\to S\times S$
is a faithfully flat morphism.

\begin{rem}
    
    When $S = \mathrm{Spec}(k),$ a $k$-groupoid acting on $S$ is nothing but an affine group scheme over $k.$ It is automatically transitive.
\end{rem}

\subsubsection{Normal discrete subgroup scheme}
\begin{defn}\label{A.2.1}
 Let $G$  be an affine $k$-groupoid scheme acting on $S.$ We  define the \textit{diagonal group scheme}
$G^\Delta$ of $G$ as the
pull-back  of $G$ along the diagonal map $\Delta:S\longrightarrow S\times S$.
$$\xymatrix{
G^\Delta \ar[r]\ar[d]& G\ar[d]\\ 
S\ar[r]&  S\times_k S.}
$$

\end{defn}
\begin{defn}
 Let $G$  be an affine $k$-groupoid scheme acting on $S.$ We call $H$ a \textit{subgroupoid scheme} if $H$ is a closed sub-scheme of $G$ such that the morphisms $(m|_H, \varepsilon|_{H}, \iota|_{H})$ makes $H$ into $k$-groupoid scheme acting on $S.$ We call $H$ a \textit{discrete subgroupoid scheme} if $H$ is a closed sub-scheme of the diagonal group scheme $G^{\Delta}.$ A discrete subgroupoid scheme $H$  is said to be \textit{normal} if $H$ is a normal subgroup of $G^\Delta.$
\end{defn}

\subsubsection{Kernel of morphism of affine $k$-groupoid schemes}\label{sect_kernel}
Let $f: G_1 \longrightarrow G$ be a morphism of $k$-groupoid schemes acting on a $k$-scheme $S.$  We define the kernel of a homomorphism $f:G_1\to G$
as the fiber product $\ker f:=S\times_GG_1$:
$$\xymatrix{\ker f\ar[r]\ar[d] & G_1\ar[d]^f\\ S\ar[r]^{\epsilon} &G.}$$
Thus, $\ker f$ is a  group scheme over $S$.  
Taking the diagonal group schemes, we see
that $\ker f$ is isomorphic to the kernel of the  homomorphism
$G_1^\Delta\to G^\Delta$ of group schemes:
$$\xymatrix{ \ker f\ar[r]\ar@{=}[d]& G_1\ar[r]^f& G \\
\ker f^\Delta\ar[r]& G_1^\Delta\ar[r]^{f^\Delta}\ar[u]&G^\Delta.\ar[u].}$$
Indeed, the canonical map $\ker  f^\Delta\longrightarrow \ker f$
comes from the definition of $\ker  f$ and the  (outer)  commutative diagram
$$\xymatrix{\ker  f^\Delta\ar[r]\ar[d] &G_1^\Delta \ar[r]\ar[d] &G_1\ar[d]^f\\
S\ar[r]&G^\Delta\ar[r]& G.}$$
And the canonical map $\ker  f\longrightarrow  \ker  f^\Delta$
comes from the map $ \ker  f \longrightarrow G_1^\Delta$ which satisfies the commutative diagram
$$\xymatrix{ \ker  f \ar[r]\ar[d]& G_1^\Delta\ar[d]\\ S\ar[r]& G^\Delta.}$$

\subsubsection{Representations ($G$-modules)}\label{sub-representationofgroupoid}
 Let $V$ be a quasi-coherent
sheaf on $S$. A \textit{representation}  of $G$ in $V$ (a \textit{$G$-module}) is an operation
 $\rho$, that assigns to each $k$-scheme $T$ and each
 morphism $\phi:T\to G$ a $T$-isomorphism
\ga{A8}{\rho(\phi):a^*V\to b^*V}
where $(a,b)=(s,t)\circ \phi$, the source and the target of $\phi$, and
$a^*$ (resp. $b^*$) denotes the pull-back of $V$
along $a$ (resp. $b$). One requires that this operation
be compatible with the composition law of the groupoid
$(S(T),G(T))$ and with the base change.  The latter means:
for any morphism $r:T'\to T$
\ga{A9}{\rho(r^*\phi)=r^*\rho(\phi).} A representation is called \textit{finite} if its underlying sheaf is coherent. 

Let $S= \mathrm{Spec}(R)$ be an affine scheme. Any $R$-module $V$ defines the functor $V_l$ and $V_r$ from the category of $R$-algebras to category the of groups as follows:
\begin{align*}
  V_r: \;  R-\text{algebras} &\longrightarrow  \text{groups}\\
             B &\mapsto (B \otimes V, +) \hspace{0.1cm} 
             \;B \;\text{consider as right}\;R-\text{module}, 
\end{align*}
and
\begin{align*}
  V_l: \;  R-\text{algebras} &\longrightarrow  \text{groups}\\
             B &\mapsto (V\otimes B, +) \hspace{0.1cm} 
             \;B \;\text{consider as left}\;R-\text{module}.
\end{align*}
 From the perspective of functorial language, we have another way to define the representation of a groupoid scheme. Let $G$ be a flat affine $k$-groupoid scheme acting on $S.$ For any $k$-algebra $B,$ consider $_sB, B_t$ are $R$-modules through the morphisms $s$ and $t,$ respectively.  A \textit{representation} of $G$ on $V$ (a \textit{$G$-module}) is an action of $G$ on
the functor $V_r,$ defined as follows:
$$  G(B) \times V_r(_sB) \longrightarrow V_l(B_t)   $$
where $G(B)$ acts on $V_r(_sB)$ by $B$-linear maps and this action commutes with base change.

A \textit{morphism} of $G$-modules is an $R$-module homomorphism that is compatible with the $G$-action. 
We denote this category by
$\mathrm{Rep}  (S:G)$ and denote the full subcategory of finite representations by 
$\mathrm{Rep}^\mathrm{f}(S:G).$

\subsubsection{Representations as comodules}

Let $S$ be an affine scheme, i.e., $S=\mathrm{Spec}(R)$ for some commutative ring $R,$ and let $G$ be an affine group scheme, with its coordinate ring denoted by $\mathcal{O}(G).$ The groupoid structure on $G$ induces the structures of a Hopf algebroid on $\mathcal O(G)$.
The source and the target map for $G$ induce algebra maps
$s,t:R\to \mathcal O(G)$.
The transitivity of $G$ on $S$ can be rephrased by saying that
$\mathcal O(G)$ is faithfully flat over $R\otimes_kR$ with
respect to the base map $t\otimes_ks:R\otimes_kR\to \mathcal O(G)$.

The composition law for $G$ induces an $R\otimes_kR$-algebra map
\ga{A10}{\Delta: \mathcal O(G)\longrightarrow \mathcal O(G)\stotimes  \mathcal O(G)}
satisfying $(\Delta\otimes \text{\rm id})\Delta=(\text{\rm id}\otimes \Delta)\Delta$.
The unit element of $G$ induces a $R\otimes_kR$-algebra map
\ga{A11}{\varepsilon:\mathcal O(G)\longrightarrow R}
where $R\otimes_kR$ acts on $R$ diagonally (i.e., $(\lambda\otimes_k\mu)\nu
=\lambda\mu\nu$).
One has
\begin{align}\label{A12}
    m(\varepsilon\otimes \text{\rm id})\Delta=m(\text{\rm id}\otimes\varepsilon)\Delta
=\text{\rm id}.
\end{align}

Finally, the inverse element map on $G$ induces an automorphism
$\iota$ of $\mathcal O(G)$, which interchanges the actions $t$ and $s$:
\begin{align} \label{A13}\iota(t(\lambda)s(\mu)h)=s(\lambda)t(\mu)\iota(h),
\end{align}
and satisfies the following equations:
\begin{align}\label{A14}
    m(\iota\otimes\text{\rm id})\Delta=s\circ\varepsilon\quad
m(\text{\rm id}\otimes\iota)\Delta=t\circ\varepsilon.
\end{align}
Since $S=\text{\rm Spec} (R)$, quasi-coherent sheaves on $S$ are $R$-modules and coherent sheaves are finite $R$-modules.
A representation $\rho$ of $G$ in $V$ induces a map
$\rho:V\to V\otimes_t \mathcal O(G)$, called co-action of $\mathcal O(G)$ on $V$, such that
\ga{A15}{(\text{\rm id}_V\otimes \Delta)\rho=(\rho\otimes \text{\rm id}_V)\rho,\quad
(\text{\rm id}_V\otimes\varepsilon)\rho=\text{\rm id}_V.}
An $R$-module equipped with such an action is called an $\mathcal O(G)$-comodule.
Conversely, any co-action of $\mathcal O(G)$ on an $R$-module $V$ defines
a representation of $G$ in $V$. In fact, we have an equivalence
between the category of $G$-representations and the category
of $\mathcal O(G)$-comodules. 


In particular, the co-product on $\mathcal O(G)$ can be considered as a co-action
of $\mathcal O(G)$ on itself and hence defines a representation of $G$ in
$\mathcal{O}(G)$, called the \textit{(left) right regular representation}.

\begin{lem}\label{lem-13}
   Let $G$ be an  affine groupoid scheme acting transitively on $S,$ and let $H$ be a normal flat discrete  subgroup scheme of $G.$ If $V$ is a $G$-module, then $V^H$ is a $G$-sub-module of $V.$ 
\end{lem}
\begin{proof}
      Let $\rho_V: V \rightarrow V \otimes \mathcal{O}(H)$ denote the comodule map of $V,$ considered as an $H$-module. We regard $V \otimes \mathcal{O}(H)$ as a $G$-module via the given action on $V$ and the conjugation action on $\mathcal{O}(H).$ Using Sweedler's notation, 
    the conjugate action of $G$ on $\mathcal{O}(H)$ can be rewritten as follows:
    \begin{align*}
        \Delta_c: \mathcal{O}(H) &\longrightarrow \mathcal{O}(H) \otimes_t \mathcal{O}(G) \\
        \overline{v}&  \mapsto \sum_{(v)}\overline{v_{(2)}}\otimes v_{(1)} \iota(v_{(3)}).
    \end{align*}
   To prove that $\rho_V$ is a $G$-morphism, we  need to prove   that the following diagram: 
    $$ \xymatrix{V \ar[r]^{\rho_V} \ar[d]^{\rho} & V \otimes \mathcal{O}(H) \ar[d]^{\mathrm{id}\otimes \Delta_c} \\
    V \otimes_t \mathcal{O}(G) \ar[r]^{\rho_V \otimes \mathrm{id}} & V \otimes \mathcal{O}(H) \otimes_t \mathcal{O}(G)} $$
is commutative. We have
\begin{align*}
    (\mathrm{id} \otimes \Delta_c)\circ \rho_V(v) & = (\mathrm{id} \otimes \Delta_c)(\sum_{(v)} v_{(0)} \otimes v_{(1)})\\
    & = \sum_{(v)} v_{(0)(0)} \otimes \overline{v_{(1)(2)}} \otimes v_{(1)(1)}\iota(v_{(1)(3)})\\
    &= \sum_{(v)} v_{(0)}\otimes \overline{v_{(3)}} \otimes v_{(2)}\iota(v_{(4)})\\
    &= \sum_{(v)} v_{(0)} \otimes \overline{v_{(3)}}\otimes \iota(v_{(2)})v_{(4)}\\
    & = \sum_{(v)} v_{(0)}\otimes \overline{v_{(1)}}\otimes v_{(2)}\\
    & = \sum_{(v)} v_{(0)(0)} \otimes \overline{v_{(0)(1)}} \otimes v_{(1)}\\
    & = (\rho_V \otimes \mathrm{id})\rho(v).
\end{align*}

Thus, the map $\rho_V: V \to V \otimes \mathcal{O}(H)$ is a homomorphism of $G$-modules. The same holds for the map $v \mapsto \rho_V(v) - v \otimes 1$. Since the category $\mathrm{Rep}^\mathrm{f}(R:G)
$ is abelian (see Subsection \ref{sect_tann_dual}), the space of $H$-invariants, $V^H$, which is the kernel of this map, is a $G$-sub-module.
 \end{proof}

\subsection{Tannakian duality over a field} \label{sect_tann_dual} 

 \subsubsection{General Tannakian duality}
  Reference for this subsection is \cite{De90}.
\begin{defn}  Let $\mathcal C$ be a rigid tensor  $k$-linear abelian category with
$\mathrm{End}(I)\cong k$ ($I$ is the unit object). 
A \textit{fibre functor} for $\mathcal C$ with values in $R$-modules, where $R$ is a $k$-algebra,
is an exact faithful $k$-linear tensor functor $\omega:\mathcal C\to\mod R$.
Such a pair $(\mathcal C,\omega)$ 
is called a \textit{(general) Tannakian category} over $k$.  
 \end{defn} 
 
\begin{thm}\label{the-main-theoremB}
Let  $(\mathcal C,\omega)$ be a general Tannakian category with fiber functor to the category of $R$-modules.
Then, there exists a $k$-groupoid scheme $\mathcal G$, acting transitively upon 
 $\Spec R$, such that $\omega$ induces a tensor equivalence between $\mathcal C$ and $\mathrm{Rep}(R:\mathcal G)$. \end{thm}
The groupoid $\mathcal G$ is called the Tannakian groupoid of $(\mathcal C,\omega)$.
Conversely, if we start from a groupoid scheme $\mathcal G$ acting transitively upon a ring $R$, then $\mathrm{Rep}^\mathrm{coh}(R:\mathcal G)$ equipped with the forgetful functor is a Tannakian category.
The corresponding Tannakian groupoid is isomorphic to $\mathcal G$. 
 

\subsubsection{Neutral Tannakian duality}\label{sect_Tannakian_field} 
 Reference for this subsection is \cite{DM82}. When $R=k,$ we obtain the  duality for a neutral Tannakian category. 
\begin{defn}
 A rigid abelian tensor category $\mathcal C$ equipped with an exact faithful $k$-linear tensor functor $\omega:\mathcal C\to\mathrm{Vec}_k$ is called a \textit{neutral Tannakian category} over $k.$ The functor $\omega$ is called a \textit{fibre functor} for $\mathcal C$.
 \end{defn} 
 \begin{thm}\label{the-main-theoremA}
  Let  $(\mathcal C,\omega)$ be a neutral Tannakian category. Then, there exists a $k$-group scheme $G$, such that $\omega$ 
  induces an equivalence between $\mathcal C$ and $\mathrm{Rep}^\mathrm{coh}_k(G)$.
 \end{thm}
 The group scheme $G$ above is called the \textit{Tannakian group} of the category $(\mathcal C,\omega)$.
An example of a Tannakian category is the category of finite dimensional representations of an affine group scheme $G$ over $k$, equipped with the forgetful functor of $k$-vector
spaces. The resulting Tannakian group is isomorphic to $G$.

\subsection{The induction from normal sub-group schemes}\label{sub-induction} In this subsection, we extend the result of \cite{EH06}. We continue to assume $S$ and $G$ are affine, $S=\text{\rm Spec} (R)$. Let $G$ be an  affine groupoid scheme acting transitively on $S,$ and  let $H$ be a flat discrete subgroup scheme of $G.$ In this subsection, we want to find the right adjoint functor of the restriction functor $\mathrm{Res}^G_H(-).$
\subsubsection{The induction functor from discrete normal subgroup} \label{A.2.5}
Let $W$ be a representation of $H$. We consider an action of $G$ on $W\otimes \mathcal{O}(G)$
by letting it act trivially 
on $W$ and via the left regular action on $\mathcal{O}(G)$.
We let $H$ act on $W\otimes \mathcal{O}(G)$ by the given action on $W$ and by the right
regular representation on $\mathcal{O}(G)$. Consider the 
$H$-invariant space $(W\otimes_t \mathcal{O}(G))^H.$ We want to show that 
$(W\otimes_t \mathcal{O}(G))^H$ is stable under the action of $G$ defined above. 
This follows from:
\begin{lem}\label{lem-23}
Let $G$ be an  affine groupoid scheme acting transitively on $S,$ and  let $H$ be a flat discrete subgroup scheme of $G.$ For any representation $W\in\mathrm{Obj}(\mathrm{Rep}_R (H)).$ Consider the following maps: 
\ga{6.10}{\begin{array}{l} p: W\otimes_t\mathcal O(G)
\stackrel{\rho_W\otimes\text{\rm id} }\longrightarrow
W\otimes\mathcal O(H)\otimes_t\mathcal O(G)\\
q:W\otimes_t\mathcal O(G)\stackrel{\text{\rm id} \otimes\Delta}\longrightarrow
 W\otimes_t\mathcal O(G)\stotimes\mathcal O(G)\stackrel{\mathrm{id} \otimes \pi \otimes\mathrm{id}}\longrightarrow
W\otimes\mathcal O(H)\otimes_t\mathcal O(G),\end{array}}
where $\rho_W:W\to W\otimes \mathcal O(H)$ is the co-action of
$\mathcal O(H)$ on $W$ and $\Delta$ is the co-product on $\mathcal O(G),$ and $\pi: \mathcal{O}(G)\longrightarrow \mathcal{O}(H)$ is induced by the closed immersion $H$ to $G.$
Then $(W\otimes_t \mathcal{O}(G))^H =\mathrm{Eq}(p,q)$ -- the equalizer of $p$ and $q$. 
\end{lem}
\begin{proof}
The action of $H$ on $W\otimes_t \mathcal{O}(G)$ corresponds to the following
coaction of $\mathcal{O}(H)$: 
$$
\xymatrix{ W\otimes_t \mathcal{O}(G) \ar[rr]^{\rho_W\otimes \Delta_H} \ar[d]_{\rho_{W\otimes_t \mathcal{O}(G)}} & &   W\otimes \mathcal{O}(H)\otimes\mathcal{O}(H)   \otimes_t \mathcal{O}(G) \ar[d]_{\mathrm{id}\otimes \mathrm{id}\otimes \iota \otimes \mathrm{id}}  \\
   W \otimes \mathcal{O}(H)\otimes_t \mathcal{O}(G) & &   W\otimes \mathcal{O}(H)\otimes \mathcal{O}(H)   \otimes_t \mathcal{O}(G) \ar[ll]^{\mathrm{id} \otimes m \otimes \mathrm{id}}},  $$
where $\Delta_H = (\pi \otimes \mathrm{id}) \circ \Delta,$ and $m$ is multiplication.  In Sweedler notation, we have 
$$ \rho_{W\otimes_t \mathcal{O}(G)}(w \otimes g) = \sum_{(w),(g)} w_{(0)}\otimes w_{(1)}\iota(\pi(g_{(1)}))\otimes g_{(2)},$$
for any pure tensor $w\otimes g \in W \otimes_t \mathcal{O}(G).$

Now let a pure tensor $w\otimes g \in \mathrm{Eq}(p,q).$ We have
$$\sum_{(w)} w_{(0)}\otimes w_{(1)}\otimes g
=
\sum_{(g)} w\otimes \pi(g_{(1)})\otimes g_{(2)}.$$
We apply $\mathrm{id}_{W\otimes \mathcal{O}(H)} \otimes \Delta_H,$ then apply $(\mathrm{id}\otimes \mathrm{id}\otimes \iota \otimes \mathrm{id}),$ and finally apply $\mathrm{id} \otimes m \otimes \mathrm{id}$ to obtain:
\begin{align}\label{eq-30}
    \sum_{(w),(g)} w_{(0)}\otimes w_{(1)}\iota(\pi(g_{(1)}))\otimes g_{(2)} = \sum_{(w),(g)} w_{(0)}\otimes\pi(g_{(1)}) \iota (\pi(g_{(2)(1)}))\otimes g_{(2)(2)}.   
\end{align}
The left hand side of \eqref{eq-30} is precisely $\rho_{W\otimes_t \mathcal{O}(G)},$ while the right hand side of \eqref{eq-30} is  $w \otimes 1 \otimes g.$ Therefore, $w\otimes g \in \mathrm{Ind}^G_H(W).$

On the other hand, assume that a pure tensor $w \otimes g$ belongs to $\mathrm{Ind}^G_{G^{\Delta}}(W).$ We have 
$$  \sum_{(w),(g)} w_{(0)}\otimes w_{(1)}\iota(\pi(g_{(1)}))\otimes g_{(2)} = w \otimes 1 \otimes g.$$
We apply $\mathrm{id}\otimes \iota \otimes \mathrm{id},$ then apply  $\rho_{W}\otimes \mathrm{id}\otimes \mathrm{id},$ and finally apply $\mathrm{id}\otimes m \otimes \mathrm{id}$ to obtain:
\begin{align}\label{eq-31}
    \sum_{(w),(g)} w_{(0)(0)}\otimes w_{(0)(1)}\iota(w_{(1)}\iota(\pi(g_{(1)})))\otimes g_{(2)} = \sum_{(w)} w_{(0)}\otimes w_{(1)} \otimes g.   
\end{align} 
The left hand side of \eqref{eq-31} is precisely $q(w\otimes g),$ while the right hand side of \eqref{eq-31} is $p(w\otimes g).$ Therefore, $w\otimes g \in \mathrm{Eq}(p,q).$
\end{proof}
\begin{lem}
With the assumption as in Lemma \ref{lem-23}, we have:  $(W\otimes_t \mathcal{O}(G))^H$ is a $G$-module (i.e. an $\mathcal{O}(G)$-comodule).
\end{lem}
\begin{proof}
    To prove this lemma, it suffices to show that the morphisms $p$ and $q$ in Lemma \ref{lem-23} are $G$-morphisms.    To prove that $p$ is a $G$-morphism, we  need to prove   that the following diagram: 
\begin{equation}\label{dia-10}
         \xymatrix{W \otimes_t \mathcal{O}(G) \ar[r]^{p} \ar[d]^{\rho_G} & W \otimes \mathcal{O}(H) \otimes_t \mathcal{O}(G)\ar[d]^{\mathrm{id}_{W\otimes \mathcal{O}(H)}\otimes \rho_G} \\
    W \otimes_t \mathcal{O}(G) \stotimes \mathcal{O}(G)  \ar[r]^{p \otimes \mathrm{id}} & W \otimes \mathcal{O}(H) \otimes_t \mathcal{O}(G) \stotimes \mathcal{O}(G)}
\end{equation} 
is commutative, where $\rho_G$ is $G$-comodule structure of $W\otimes_t \mathcal{O}(G).$ Let a pure tensor $w\otimes g \in W \otimes_t \mathcal{O}(G).$ We have 
$$ (\mathrm{id}_{W\otimes \mathcal{O}(H)}\otimes \rho_G)(p(w\otimes g)) = \sum_{(w),(g)}w_{(0)}\otimes w_{(1)}\otimes g_{(1)}\otimes g_{(2)},$$
and
$$ (p\otimes \mathrm{id})(\rho_{G}(w\otimes g)) = \sum_{(w),(g)}w_{(0)}\otimes w_{(1)}\otimes g_{(1)}\otimes g_{(2)}.$$
Thus, the Diagram \eqref{dia-10} is commutative. 

On the other hand, to prove that $q$ is a $G$-morphism,  we  need to prove   that the following diagram: 
\begin{equation}\label{dia-11}
         \xymatrix{W \otimes_t \mathcal{O}(G) \ar[r]^{q} \ar[d]^{\rho_G} & W \otimes \mathcal{O}(H) \otimes_t \mathcal{O}(G)\ar[d]^{\mathrm{id}_{W\otimes \mathcal{O}(H)}\otimes \rho_G} \\
    W \otimes_t \mathcal{O}(G) \stotimes \mathcal{O}(G)  \ar[r]^{q \otimes \mathrm{id}} & W \otimes \mathcal{O}(H) \otimes_t \mathcal{O}(G) \stotimes \mathcal{O}(G)}
\end{equation} 
is commutative. Let a pure tensor $w\otimes g \in W \otimes_t \mathcal{O}(G).$ We have 
$$ (\mathrm{id}_{W\otimes \mathcal{O}(H)}\otimes \rho_G)(q(w\otimes g)) = \sum_{(g)}w\otimes \pi(g_{(1)})\otimes g_{(2)(1)}\otimes g_{(2)(2)},$$
and
$$ (q\otimes \mathrm{id})(\rho_{G}(w\otimes g)) = \sum_{(g)}w\otimes \pi(g_{(1)(1)})\otimes g_{(1)(2)}\otimes g_{(2)}.$$
Thus, the Diagram \eqref{dia-11} is commutative. 
\end{proof}

 We denote this $G$-module $(W\otimes_t \mathcal{O}(G))^H$ by $\mathrm{
Ind
}^G_H(W)$ and call it the \textit{induced module} of $W$ from $H$ to $G.$ 
The name will be justified by the Frobenius reciprocity proven in the next subsection.

Let $f: W_1 \longrightarrow W_2$ be a morphism in $\mathrm{Rep}_R(H)$, that is,
\[
\rho_{W_2} \circ f = (f \otimes \mathrm{id}) \circ \rho_{W_1}.
\]
We define
\[
\mathrm{Ind}_H^G(f) := (f \otimes \mathrm{id}_{\mathcal{O}(G)})\big|_{(W_1 \otimes_t \mathcal{O}(G))^H}.
\]
Since $f$ is $H$-equivariant, the map $f \otimes \mathrm{id}_{\mathcal{O}(G)}$ is also
$H$-equivariant, hence it preserves the subspace of $H$-invariants.
Thus, $\mathrm{Ind}_H^G(f)$ is well-defined and is $G$-equivariant. Moreover, for any morphisms
$f: W_1 \to W_2$ and $g: W_2 \to W_3$, one has
\[
\mathrm{Ind}_H^G(g \circ f)
= (g \circ f) \otimes \mathrm{id}
= (g \otimes \mathrm{id}) \circ (f \otimes \mathrm{id})
= \mathrm{Ind}_H^G(g) \circ \mathrm{Ind}_H^G(f),
\]
and
\[
\mathrm{Ind}_H^G(\mathrm{id}_W) = \mathrm{id}_{\mathrm{Ind}_H^G(W)}.
\]
Thus, we get a covariant functor
\[
\mathrm{Ind}_H^G : \mathrm{Rep}_R(H) \longrightarrow \mathrm{Rep}(S:G),
\]
called the \textit{induction functor} from $H$ to $G.$

\subsubsection{Frobenius Reciprocity}
Let $G$ be an  affine groupoid scheme acting transitively on $S,$ and  let $H$ be a  flat discrete subgroup scheme of $G.$ 

For any representation $V$ of $H$, 
let $\varepsilon_V: V \otimes_t \mathcal{O}(G) \rightarrow V$ be the linear map  
\begin{equation}\label{999}
v\otimes g \mapsto  (\operatorname{id}_V \bar{\otimes} \varepsilon_G) (v\otimes g) = 
v\cdot\varepsilon_G(g). 
\end{equation}
The map $\epsilon_V$ is $H$-equivariant. 
Indeed, let $x \in (V\otimes_t \mathcal{O}(G))^H.$ Then 
$ \rho_{V\otimes_t\mathcal{O}(G)}(x) = x \otimes 1.$
Therefore,
\begin{align*}
    (\epsilon_V \otimes \mathrm{id})(\rho_{V\otimes_t\mathcal{O}(G)}(x)) & = (\epsilon_V \otimes \mathrm{id})(x\otimes 1)\\
    & = \epsilon_V(x) \otimes \mathrm{id}\\
     & = \rho_V (\epsilon_V(x)).
\end{align*}
 \begin{lem}[Frobenius Reciprocity]\label{lem-frobeniusreciprocity}
 Let $G$ be an affine $k$-groupoid scheme acting transitively on $S = \mathrm{Spec}(R),$ and let $H$ be a flat discrete  subgroup scheme of $G.$ Let $W$ be an $H$-module. For each $G$-module $V$ the map $\varphi \mapsto \varepsilon_V \circ \varphi$ is an isomorphism
\begin{align}\label{37}
    \text{\rm Hom} _G(V,\text{\rm Ind}_{H}^G (W))\cong \text{\rm Hom} _{H}(V,W),
\end{align}
i.e., the functor $\text{\rm Ind}_{H}^G$ is the right adjoin to the functor restricting
$G$-representations to $H.$
\end{lem}
\begin{proof}
 The converse map is given by $f \mapsto (f\otimes \mathrm{id})\rho_V, $ where $f$ belongs to $\mathrm{Hom}_{H}(V,W).$ We first check that this map is well defined, meaning that   $(f\otimes \mathrm{id})\rho_V$ is the morphism between $V$ and $\text{\rm Ind}_{H}^G (W).$ According to Lemma \ref{lem-23}, we need to prove that 
 $$ (\rho_W \otimes \mathrm{id})(f\otimes \mathrm{id})\rho_V = \pi (\mathrm{id}\otimes \Delta)(f\otimes \mathrm{id})\rho_V.   $$
 From the left hand side, we obtain the following commutative diagram:
 $$ \xymatrix{V \ar[r]^{\rho_V} \ar[d]^{\rho_V} & V\otimes_t \mathcal{O}(G) \ar[d]^{\rho_V \otimes \mathrm{id}} \ar[r]^{f\otimes \mathrm{id}}& W\otimes_t \mathcal{O}(G)\ar[r]^{\rho_W \otimes \mathrm{id}}& W \otimes \mathcal{O}(H)\otimes_t \mathcal{O}(G)\\ V\otimes_t \mathcal{O}(G) \ar[r]^{\mathrm{id}\otimes \Delta}  & V\otimes_t \mathcal{O}(G) \stotimes \mathcal{O}(G) \ar[r]^{\mathrm{id}\otimes \pi \otimes \mathrm{id}}& V\otimes \mathcal{O}(H)\otimes_t \mathcal{O}(G) \ar[ur]^{f\otimes \mathrm{id}\otimes \mathrm{id}},}$$
 so
 $$ (\rho_W \otimes \mathrm{id})(f\otimes \mathrm{id})\rho_V  = (f\otimes \mathrm{id}\otimes \mathrm{id}) (\mathrm{id}\otimes \pi \otimes \mathrm{id}) (\mathrm{id}\otimes \Delta) \rho_V. $$
Besides, we also have the following commutative diagram:
$$ \xymatrix{W \otimes_t \mathcal{O}(G) \ar[r]^{1 \otimes \Delta\qquad } & W\otimes_t \mathcal{O}(G) \stotimes \mathcal{O}(G) \ar[r]^{\mathrm{id} \otimes \pi \otimes \mathrm{id}} & W \otimes \mathcal{O}(H) \otimes_t \mathcal{O}(G) \\
V\otimes_t \mathcal{O}(G) \ar[u]^{f\otimes \mathrm{id}} \ar[r]^{\mathrm{id}\otimes \Delta\qquad } & V \otimes_t \mathcal{O}(G) \stotimes \mathcal{O}(G) \ar[r]^{\mathrm{id}\otimes \pi \mathrm{id}} &V \otimes \mathcal{O}(H)\times_t \mathcal{O}(G)\ar[u]^{f\otimes \mathrm{id}\otimes \mathrm{id}},} $$
this implies that 
\begin{align*}
      (\rho_W \otimes \mathrm{id})(f\otimes \mathrm{id})\rho_V &= (f\otimes \mathrm{id}\otimes \mathrm{id}) (\mathrm{id}\otimes \pi \otimes \mathrm{id}) (\mathrm{id}\otimes \Delta) \rho_V \\
      & = \pi (\mathrm{id}\otimes \Delta)(f\otimes \mathrm{id})\rho_V.
\end{align*}

We now prove that the the maps $\phi \mapsto \varepsilon \circ \phi$ and $f \mapsto   (f\otimes \mathrm{id})\rho_V$ are inverse to each other. Let $f$ be a $G$-morphism from $V$ to $\mathrm{Ind}^G_{H}(W).$ Then we have
\begin{align*}
     ((\varepsilon_W \circ f) \otimes \mathrm{id}) \circ \rho_V &= \varepsilon_W \circ (f \otimes \mathrm{id}) \circ \rho_V  \\  
 & = \varepsilon_W \circ (\rho_{\mathrm{Ind}_{H}^G(W)} \circ f) \\
 & = f.
\end{align*}
On the other hand, let $\phi$ be a $H$-morphism between $V$ and $W.$ We also have
 \begin{align*}
 \varepsilon_{W}((\phi \otimes \mathrm{id}) \circ \rho_V) &= \varepsilon_W (\rho_W \circ \phi) \\
 & = \phi.
 \end{align*}
\end{proof}

\begin{cor}\label{cor-23}
     Let $G$ be an affine $k$-groupoid scheme acting transitively on $S = \mathrm{Spec}(R),$ and let $H$ be a flat discrete  subgroup scheme of $G.$ Let $W$ be an $H$-module. There is a natural isomorphsim
     $$ \mathrm{Ind}_{G^{\Delta}}^G \circ \mathrm{Ind}_H^{G^{\Delta}} = \mathrm{Ind}_H^{G}.$$
\end{cor}
\begin{proof}
One has 
$$ \mathrm{Res}^{G^{\Delta}}_H  \circ\mathrm{Res}^G_{G^{\Delta}} = \mathrm{Res}^G_H.   $$
 The result follows by the uniqueness of the adjoint.
\end{proof}

\begin{thm}[{Cf. \cite[Remark 6.7]{EH06}}]\label{thm-exact}
Let $G$ be an  affine groupoid scheme acting transitively on $S.$ The functor $\text{\rm Ind}_{G^\Delta}^G$ is faithfully exact. Hence, the canonical map
$$\text{\rm Ind}_{G^\Delta}^G(W)\longrightarrow W$$ 
is surjective for any $G$-representation $W$. 
\end{thm} 
\begin{proof} Before giving the proof of this theorem, we discuss the algebra $\mathcal{O}(G^{\Delta}).$ By the definition of $\mathrm{G}^{\Delta}$, we have
$$
\mathcal{O}\left(G^{\Delta}\right) \cong \mathcal{O}(G) \otimes_{R \otimes_k R} R
$$
where $R \otimes_k R \rightarrow R$ is the product map. Then $J:=\operatorname{Ker}\left(R\otimes_k R \rightarrow R\right)$ is generated by elements of the form $\lambda \otimes 1-1 \otimes \lambda, \lambda \in R.$ Since $\mathcal{O}(G)$ is faithfully flat over $R \otimes_k R,$ tensoring the exact sequence $0 \rightarrow J \rightarrow R\otimes_k R \rightarrow R \rightarrow 0$ with $\mathcal{O}(G)$ over $R \otimes_{k} R$, one obtains an exact sequence
\begin{align}\label{31}
    0 \longrightarrow J\mathcal{O}(G) \longrightarrow \mathcal{O}(G) \stackrel{\pi}{\longrightarrow} \mathcal{O}\left(G^{\Delta}\right) \longrightarrow 0.
\end{align}
That is, we can identify $J \otimes_{R \otimes_k R}  \mathcal{O}(G)$ with its image $J \mathcal{O}(G)$ in $\mathcal{O}(G).$ In order to prove the faithfully exactness of $\mathrm{Ind}^G_{G^{\Delta}},$ we need the following claim.

\noindent \textbf{{Claim}}. 
The map:
\begin{equation}\label{32}
    \begin{aligned}
 \varphi: \mathcal{O}(G) \otimes_{R \otimes_{k} R} \mathcal{O}(G) &\longrightarrow \mathcal{O}\left(G^{\Delta}\right) \otimes_{t} \mathcal{O}(G), \\
 g \otimes h &\longmapsto \sum_{(g)} \pi\left(g_{(1)}\right) \otimes g_{(2)} h,
\end{aligned}
\end{equation}
is an isomorphism, where $\pi$ is defined in the formula \eqref{31}.
Here we use the Sweedler's notation:  
$
\Delta(g)=\sum_{(g)} g_{(1)} \otimes g_{(2)}.
$

\noindent \textit{Verification}. We define an inverse to this map. Let
$$
\bar{\psi}: \mathcal{O}(G)_s \otimes_t \mathcal{O}(G) \longrightarrow \mathcal{O}(G) \otimes_{R \otimes_{k}R} \mathcal{O}(G)
$$
be the map that maps $g \otimes h \mapsto \sum_{(g)} g_{(1)} \otimes \iota\left(g_{(2)}\right) h$. We have for $\lambda \in R$ and the structure maps $t, s: R \rightarrow \mathcal{O}(G)$
$$
\begin{aligned}
 \bar{\psi}\left(t(\lambda) g \stotimes h\right)&=\sum_{(g)} g_{(1)} \otimes \iota\left(t(\lambda) g_{(2)}\right) h \\
& =\sum_{g} g_{(1)} \otimes_{R \otimes_k R} s(\lambda) \iota\left(g_{(2)}\right) h \quad \text { by }\eqref{A13}  \\
& =s(\lambda) \sum_{g} g_{(1)} \otimes_{R \otimes_k R} \iota\left(g_{(2)}\right) h \\
& =\bar{\psi}\left(s(\lambda) g \stotimes h\right) \text {. } \\
&
\end{aligned}
$$
Thus, $\bar{\psi}$ maps $J \mathcal{O}(G) \stotimes \mathcal{O}(G)$ to $0,$ hence factors through a map $$\psi: \mathcal{O}\left(G^{\Delta}\right) \otimes_t \mathcal{O}(G) \rightarrow \mathcal{O}(G) \otimes_{R \otimes_{k} R} \mathcal{O}(G).$$ Checking $\varphi \psi=\mathrm{id}, \psi \varphi=\mathrm{id}$ can be done by using the property \eqref{A14} of $\iota.$ \hfill $\triangleleft$

We now prove that the functor $\mathrm{Ind}^G_{G^{\Delta}}(-)$ is faithfully exact. Let $W \in \mathrm{Obj}(\mathrm{Rep}_R(G^{\Delta})).$ Consider the following map
\begin{equation}\label{34}
    \begin{aligned}
 \Phi: \operatorname{Ind}^G_{G^{\Delta}} (W)\otimes_{R \otimes_k R} \mathcal{O}(G) &\xrightarrow{\mathrm{id}\otimes\varphi} W \otimes_{t} \mathcal{O}(G), \\
 w \otimes g \otimes h &\longmapsto w \otimes g h .
\end{aligned}
\end{equation}\label{35}

\noindent{\bf Claim}. This map is an isomorphism with an inverse given as follows:
$$
\begin{aligned}\Psi=\left(\mathrm{id}_W \otimes \psi\right)\circ\left(\rho_W \otimes \mathrm{id}\right):  W \otimes_{t} \mathcal{O}(G) \stackrel{\rho_W\otimes \mathrm{id}}{\longrightarrow} W \otimes \mathcal{O}\left(G^{\Delta}\right) \otimes \mathcal{O}(G) \stackrel{\mathrm{id}\otimes\psi}{\longrightarrow} W \otimes_{t} \mathcal{O}(G) \otimes_{R \otimes_{k} R} \mathcal{O}(G).
\end{aligned}
$$

\noindent{\it Verification}. 
It is to check that the image of $\Psi$ belongs to 
$\mathrm{Ind}^G_{G^{\Delta}}(W)\otimes_{R\otimes_kR}\mathcal{O}(G)$ 
and that $\Psi$ is the inverse of $\Phi.$  In Sweedler's notations, we have 
$$ \Psi(w \otimes g) = \sum_{(w)} w_{(0)}\otimes w_{(1)(1)}\otimes \iota(w_{(1)(2)})g$$
for any pure tensor $w\otimes g \in W \otimes_t \mathcal{O}(G).$ 
It is easy to check that 
$$p(\sum_{(w)} w_{(0)}\otimes w_{(1)(1)}) = q(\sum_{(w)} w_{(0)}\otimes w_{(1)(1)}).$$
So, by Lemma \ref{lem-23},
we have
$\Psi(w\otimes g) \in \mathrm{Ind}^G_{G^{\Delta}}(W)\otimes_{R\otimes_kR}\mathcal{O}(G)$. Next, for a pure tensor $w\otimes g$ in $W \otimes_t \mathcal{O}(G)$ we have
\begin{align*}
    \Phi(\Psi(w\otimes g)) &= \Phi(\sum_{(w)}w_{(0)}\otimes w_{(1)(1)} \otimes \iota(w_{(1)(2)})g)\\
    & = \sum_{(w)} w_{(0)}  \otimes w_{(1)(1)} \iota(w_{(1)(2)})g\\
    &= 
    w\otimes g.
\end{align*}
Further, if a tensor $\sum_i w_i\otimes g_i \otimes h$ belongs to $\mathrm{Ind}^G_{G^{\Delta}}(W) \otimes_{R\otimes_k R}\mathcal{O}(G)$ then
\begin{equation} \label{eq-34}
    \begin{aligned}
    \Psi(\Phi(\sum_iw_i\otimes g_i \otimes h)) &= \sum_i\Psi(w_i\otimes g_ih)\\
    &=\sum_i (\mathrm{id}\otimes \psi) (\rho_W \otimes \mathrm{id})(w_i\otimes g_ih)\\
    &=\sum_i (\mathrm{id}\otimes \psi) (\sum_{(w_i)} w_{i(0)}\otimes w_{i(1)}\otimes g_ih).
\end{aligned}
\end{equation}
By Lemma \ref{lem-23}, we have
$$  \sum_{(w)} w_{(0)}\otimes w_{(1)}\otimes gh = \sum_{(g)} w \otimes \pi(g_{(1)}) \otimes g_{(2)}h.$$ Thus  \eqref{eq-34} becomes: 
\begin{align*}
      \Psi(\Phi(\sum_iw_i\otimes g_i \otimes h)) &=\sum_i 
       (\mathrm{id}\otimes \psi) (\sum_{(g_i)} w_i \otimes \pi(g_{i(1)}) \otimes g_{i(2)}h)\\
      & =\sum_i \sum_{(g_i)} w_i \otimes g_{i(1)(1)} \otimes \iota(g_{i(1)(2)})g_{i(2)}h \\
      &=\sum_i 
      (w_i\otimes g_i \otimes h).
\end{align*}
\  \hfill $\triangleleft$

According to \eqref{34}, the functor
$$
\operatorname{Ind}^G_{G^{\Delta}}(-) \otimes_{R \otimes_k R} \mathcal{O}(G) \cong(-) \otimes_t \mathcal{O}(G)
$$
is faithfully exact. Since $\mathcal{O}(G)$ is faithfully flat over $R \otimes_{k} R$, the functor $\mathrm{Ind}_{G^{\Delta}}^{G}$ is also faithfully exact.

We now prove the last part of the theorem. Setting $V = \mathrm{Ind}^G_{G^{\Delta}}(W)$ in \eqref{37}, we define the canonical map
 $u_W: \operatorname{Ind}^G_{G^{\Delta}}(W) \rightarrow W$ as follows:
\begin{align*}
      \text{\rm Hom} _G(\text{\rm Ind}_{G^\Delta}^G (W),\text{\rm Ind}_{G^\Delta}^G (W))&\longrightarrow \text{\rm Hom} _{G^\Delta}(\text{\rm Ind}_{G^\Delta}^G (W),W) \\
      \mathrm{id} & \longmapsto u_W.
\end{align*}
 The map $u_W$ is nonzero whenever $W$ is nonzero. Indeed, since $\mathrm{Ind}^G_{G^{\Delta}}$ is faithfully exact, the $G^{\Delta}$-module $\mathrm{Ind}^G_{G^{\Delta}} (W)$ is nonzero whenever $W$ is nonzero. Thus, if $u_W=0$, then both sides of \eqref{37} are zero for any $V$. On the other hand,  the right-hand side contains the identity map. A contradiction. 
 

We now turn to show that $u_W$ is surjective. Let ${U}=\operatorname{Im}\left(u_W\right)$ and let $ {T}=W / {U}$. We have the following diagram:
$$\xymatrix{0 \ar[r] &\mathrm{Ind}^G_{G^{\Delta}}(U)\ar[r]\ar[d] &\mathrm{Ind}^G_{G^{\Delta}}(W) \ar[r]\ar[d] &\mathrm{Ind}^G_{G^{\Delta}}(T)\ar[d] \ar[r] & 0\\
0 \ar[r]& U\ar[r]&W\ar[r]& T \ar[r]& 0.}$$

The composition $\operatorname{Ind}^G_{G^{\Delta}}(W) \rightarrow \operatorname{Ind}^G_{G^{\Delta}}(T) \rightarrow T$ is $0.$ Therefore, $\operatorname{Ind}^G_{G^{\Delta}}(T) \rightarrow T$ is a zero map, implying $T=0.$
\end{proof}

\begin{cor}\label{lem-RepGDiagonal}
Let $G$ be an  affine groupoid scheme acting transitively on $S = \mathrm{Spec}(R).$ Any $G^\Delta$-representation is a quotient the restriction to $G^\Delta$ of a $G$-representation. 
Consequently, any $R$-projective finite representation of $G^\Delta$ is  also a special sub-object of a finite $G$-representation considered as representation of $G^{\Delta}.$
\end{cor}
\begin{proof}
Let $W$ be a representation of $G^{\Delta}$ and $u_W: \mathrm{Ind}^{G}_{G^{\Delta}}(W) \longrightarrow W$ be the canonical map as in Theorem \ref{thm-exact}. This theorem implies that $W$ is a quotient of $\mathrm{Ind}^{G}_{G^{\Delta}}(W).$
Since $\mathrm{Ind}^G_{G^{\Delta}}(W)$ is a union of its finite sub-representations, we can  find a finite  $G$-sub-representation $W_0(W)$ of $\operatorname{Ind}^G_{G^{\Delta}}(W)$, which still maps surjectively on $W$. In order to obtain the statement on the embedding of $R$-projective representation $W$, one repeats the above argument for $W^{\vee}$ to get the surjective map $W_0(W^{\vee}) \twoheadrightarrow
W^{\vee},$ and then dualizes to $W \xhookrightarrow{} (W_0(W^{\vee}))^{\vee}.$
\end{proof}

\subsection{Cohomology}\label{sect_cohomology_of_groupoid}
Let $G$ be a flat affine $k$-groupoid scheme acting  on $S.$  
\subsubsection{Injective objects in the category $\mathrm{Rep}(S:G)$}
\begin{lem}\label{lem-enoughinjectives}
 Let $G$ be a flat affine $k$-groupoid scheme acting  on $S.$ The following statements hold true:
\begin{itemize}
    \item [(1)] The category $\mathrm{Rep}(S:G)$ has enough injectives.
    \item [(2)] A $G$-module $V$ is injective if and only if there is an injective $R$-module $I$ such that $V$ is isomorphic to a direct summand of $I \otimes_t \mathcal{O}(G)$ with $I$ regards as a trivial $G$-module.
\end{itemize}
\end{lem}
\begin{proof}
We have $\mathrm{Ind}_1^G(-)=(-)\otimes_t\mathcal{O}(G)$, hence
for any $R$-module $I$ and  $G$-module $U$, the Frobenius reciprocity (Lemma \ref{lem-frobeniusreciprocity}) yields a functorial isomorphism
$$ \mathrm{Hom}_G(U,I \otimes_t \mathcal{O}(G)) \xrightarrow\cong \mathrm{Hom}_R(U,V).
$$
If $I$ is injective, the functor on the right hand
side is exact, hence so is that on the left hand side, that is, $I \otimes_t \mathcal{O}(G)$ is
an injective $G$-module. 

Let $V$ be a $G$-module. We can embed $V$ into an injective $R$-module $I,$ 
and have injective $G$-maps:
$$ V \overset{\rho_V}{\longrightarrow} V \otimes_t \mathcal{O}(G) \xhookrightarrow{} I \otimes_t \mathcal{O}(G).$$ 
Hence $\mathrm{Rep}(S:G)$ has enough injectives.
If $V$ is an injective $G$-module, the above injective maps show 
that $V$ is a direct summand of $I \otimes_t \mathcal{O}(G),$ where $I$ is the injective envelope
of $V$ as $R$-modules.  
The converse direction holds  because every direct summand of an injective $G$-module is
injective.
\end{proof}

\subsubsection{Cohomology groups} 
   Let $G$ be a flat affine $k$-groupoid scheme acting on $S,$ and let $V$ be a $G$-module. We define the set of fixed points by
$$ V^{G} = \{v \in V | \rho_V(v) = v \otimes 1 \in V \otimes_t \mathcal{O}(G) \}.  $$
If $\phi: V \rightarrow V'$ is a homomorphism of $G$-modules, then $\phi(V^{G}) \subset {V'}^{G}.$ Thus, we obtain the fixed point functor
\begin{align*}
    \mathrm{Rep}(S:G) &\longrightarrow \mathrm{Vec}_k\\
    V & \mapsto V^{G}=\mathrm{Hom}_G(k,V),
\end{align*}
where $k$ is equipped with the trivial action.
We observe that if $G$ is flat over $S$, then the fixed point functor is left exact. 
The category $\mathrm{Rep}(S:G)$ has enough injectives (Lemma \ref{lem-enoughinjectives}), 
so we define the derived functors of $(-)^G$:
$$      V \mapsto \mathrm{H}^i(G,V), \quad i\geq 0.        $$
The module $\mathrm{H}^i(G,V)$ is called the $\textit{$i$-cohomology group}$ of $V.$

\subsubsection{Shapiro's lemma}\label{sec-1}
Grothendieck's spectral sequence is
standard in Homological algebra. We recall it briefly here for the reader's sake, our reference is \cite{We94}.
 
Let $\mathcal{C}, \mathcal{C}', \mathcal{C}''$ be  abelian categories with $\mathcal{C}, \mathcal{C}'$ having enough injectives. Suppose now that $\mathcal{F}: \mathcal{C} \rightarrow \mathcal{C}^{\prime}$ and $\mathcal{F}^{\prime}: \mathcal{C}^{\prime} \rightarrow \mathcal{C}^{\prime \prime}$ are additive (covariant) functors. If $\mathcal{F}'$ is left exact and if $\mathcal{F}$ maps injective objects in $\mathcal{C}$  to objects acyclic for $\mathcal{F}^{\prime}$, then there is a spectral sequence for each object $M$ in $\mathcal{C}$ with differentials $d_r$ of bi-degree $(r, 1-r)$, and
\begin{align}
    E_2^{p,q} = (\mathbf{R}^p\mathcal{F}')(\mathbf{R}^q\mathcal{F})M \Rightarrow \mathbf{R}^{p+q}(\mathcal{F}'\circ \mathcal{F})M.
\end{align}
\begin{rem}\label{rem-2}
    We notice the following:
    \begin{itemize}
        \item [(1)] If $\mathcal{F}^{\prime}$ is exact, then $\mathcal{F}^{\prime} \circ \mathbf{R}^q \mathcal{F} \simeq \mathbf{R}^q\left(\mathcal{F}^{\prime} \circ \mathcal{F}\right)$ for all $n \in \mathbb{N}$.
        \item [(2)] If $\mathcal{F}$ is exact and maps injective objects to objects acyclic for $\mathcal{F}^{\prime}$, then
        $$(\mathbf{R}^i\mathcal{F}')\circ \mathcal{F} \simeq \mathbf{R}^i(\mathcal{F}'\circ \mathcal{F}),$$
    for all $n\in \mathbb{N}.$
    \end{itemize}
\end{rem}
  
\begin{rem}
    Let $G$ be an affine $k$-groupoid scheme acting transitively on $S = \mathrm{Spec}(R)$ and $H$ be its flat subgroupoid scheme. Then $\mathrm{Ind}_H^G(-)$ is a left exact functor. Thus, we can take the right derived functors $\mathbf{R}^i\mathrm{Ind}_H^G(-).$
\end{rem}
\begin{prop}\label{lem-spectralsequences}
Let $G$ be a flat affine $k$-groupoid scheme acting transitively on $S = \mathrm{Spec}(R)$ and $H$ be its flat discrete subgroupoid scheme. Let $W$ be an $H$-module. There is a spectral sequence with
$$
E_2^{p, q}=\mathrm{H}^p\left(G, R^q \mathrm { Ind }_H^G W\right) \Rightarrow \mathrm{H}^{p+q}(H, W) .
$$

\end{prop}
\begin{proof}
 Lemma \ref{lem-frobeniusreciprocity} can be interpreted as an isomorphism of functors (choose $V=R$)
$$
\operatorname{Hom}_G(R, -) \circ \operatorname{Ind}_H^G \simeq \operatorname{Hom}_H(R, -) .
$$
Since $\mathrm{Ind}_H^G$ is the right adjoin to the restriction functor which is exact,  
it maps injective $H$-modules to injective $G$-modules. 
Hence, we can apply Grothendieck's spectral sequence.
\end{proof}

\begin{defn}
Let $G$ be an affine $k$-groupoid scheme acting transitively on $S = \mathrm{Spec}(R)$ and $H$ be its flat subgroupoid scheme. We call $H$ \textit{exact in $G$} if $\mathrm{Ind}_H^G$ is an exact functor.    
\end{defn}

Using Proposition \ref{lem-spectralsequences}, we obtain the following result.
\begin{cor}[Shapiro's lemma]
Let $G$ be a flat affine $k$-groupoid scheme acting transitively on $S = \mathrm{Spec}(R)$ and $H$ be its flat discrete subgroupoid scheme. Suppose that $H$ is exact in $G$. Let $W$ be an $H$-module. For each $n \in \mathbb{N},$ there is an isomorphism
$$
\mathrm{H}^i\left(G, \operatorname{Ind}_H^G W\right) \simeq \mathrm{H}^i(H, W) .
$$
\end{cor}

Since $\mathrm{Ind}_1^G =-\otimes_t \mathcal{O}(G)$ is an exact functor, we obtain the following result.
\begin{lem}\label{cor-2}
Let $G$ be a flat affine $k$-groupoid scheme acting transitively on $S = \mathrm{Spec}(R)$.
 We have for each $G$-module $V:$ 
$$
\mathrm{H}^i(G, V \otimes_t \mathcal{O}(G)) \simeq\left\{\begin{array}{lll}
V & \text { if } & i=0, \\
0 & \text { if } & i>0 .
\end{array}\right.
$$
\end{lem}

\begin{prop}\label{lem-11}
Let $G$ be a flat affine $k$-groupoid scheme acting transitively on $S = \mathrm{Spec}(R)$ and $H$ be a flat discrete subgroupoid scheme. We have for each $H$-module $V$ and each $i \in \mathbb{N}$ an isomorphism of $R$-modules
$$
\mathrm{H}^i(H, V \otimes_t \mathcal{O}(G)) \simeq \left(R^i \mathrm{ Ind }_H^G\right) V .
$$
\end{prop}
\begin{proof}
The proof is based on the definition of the induction functor. Indeed, the definition of $\mathrm{Ind}_H^G$ yields an isomorphism of functors
$$   \mathrm{Res}^G_1 \circ \mathrm{Ind}^G_H \simeq (-)^H \circ (-\otimes_t \mathcal{O}(G)).$$
Since the functor $- \otimes_t \mathcal{O}(G)$ is exact and maps injective $H$-modules to modules 
that are acyclic for the fixed points functor (see Lemma \ref{cor-2}-(1)), we can apply Remark 
\ref{rem-2}-(1),(2) and the result  follows.
\end{proof}

 \subsection*{Acknowledgments }
 We  thank Francesco Baldassarri and H\'el\`ene Esnault
for engaging in insightful discussions. 
 
\subsection*{Funding} 
 
The work of Tran Phan Quoc Bao is supported by VAST, under grant number
CTTH00.02/23-24 ``Arithmetic and Geometry of schemes over function fields and applications''
and is funded by Vingroup Joint Stock Company and supported by Vingroup Innovation Foundation under the project code VINIF.2021.DA00030.

The work of Vo Quoc Bao is partially supported by the International Centre of Research and
Postgraduate-Training, IM-VAST,  under the grant number ICRTM03\_2020.03 and by the Vingroup Innovation Foundation under grants number VINIF.2021.TS.116 and  VINIF.2023.TS.011
and by the NAFOSTED under the grant number  101.04-2025.20 ``Some problems on fundamental groups, cohomology, and moduli space''.

Phung Ho Hai is partially supported by VAST under grant number
CTTH00.02/23-24 ``Arithmetic and Geometry of schemes over function fields and applications'' 
and by the NAFOSTED under the grant number  101.04-2025.20 ``Some problems on fundamental groups, cohomology, and moduli space''.

\end{document}